\documentclass[11pt, letterpaper]{amsart}
\usepackage{indentfirst}
\usepackage[iso-8859-7]{inputenc}
\usepackage{graphics}
\usepackage{graphicx}
\usepackage{amssymb}
\usepackage{amsmath}
\usepackage{amsthm}
\usepackage{latexsym}
\usepackage{verbatim}
\usepackage{mathrsfs}
\usepackage{hyperref}
\usepackage{esint}
\usepackage{MnSymbol}
\usepackage{color}
\makeatletter
\def\verbatim@font{\usefont{OT1}{cmtt}{m}{n}}
\makeatother

\newsavebox{\ChapterTitle}
\makeatletter
\newcommand{\ps@spawn}{%
    \renewcommand{\@oddhead}{{\usebox{\ChapterTitle}\hfil\thepage}\hspace{-13.83cm}\rule[-8pt]{13.83cm}{0.5pt}}%
    \renewcommand{\@evenhead}{{\thepage\hfil {\selectlanguage{english} Chapter} \thechapter}\hspace{-13.83cm}\rule[-8pt]{13.83cm}{0.5pt}}%
    \renewcommand{\@evenfoot}{}%
    \renewcommand{\@oddfoot}{\@evenfoot}}
\makeatother

\definecolor{mygray}{gray}{0.9}

	\normalbaselines						  %

\newcommand{\ci}[1]{_{_{\scriptstyle #1}}}
\newcommand{\ti}[1]{_{\scriptstyle \text{\rm #1}}}
\newcommand{\ut}[1]{^{\scriptstyle \text{\rm #1}}}

\newtheorem{thm}{Theorem}[section]
\newtheorem{prop}[thm]{Proposition}
\newtheorem{cor}[thm]{Corollary}
\newtheorem{lm}[thm]{Lemma}

\theoremstyle{remark}
\newtheorem{rem}[thm]{Remark}
\newtheorem*{rem*}{Remark}

\theoremstyle{definition}

\newtheorem*{df*}{Definition}

\numberwithin{equation}{section}

\newcommand{\cz}{Calder\'{o}n--Zygmund\ }
\newcommand{\R}{\mathbb{R}}
\newcommand{\C}{\mathbb{C}}
\newcommand{\E}{\mathbb{E}}
\newcommand{\cD}{\mathcal{D}}
\newcommand{\cF}{\mathcal{F}}
\newcommand{\cS}{\mathscr{S}}
\newcommand{\cT}{\mathscr{T}}
\newcommand{\cR}{\mathscr{R}}
\newcommand{\cE}{\mathscr{E}}

\newcommand{\La}{\langle}
\newcommand{\Ra}{\rangle}

\newcommand{\im}{\operatorname{Im}}
\newcommand{\re}{\operatorname{Re}}

\newcommand{\ch}{\operatorname{ch}}

\newcommand{\wt}{\widetilde}

\newcommand{\e}{\varepsilon}

\newcommand{\bff}{\mathbf f}
\newcommand{\bg}{\mathbf g}

\newcommand{\ssDelta}{{\scriptscriptstyle{\Delta}}}

\newcommand{\1}{\mathbf{1}}

\newcommand{\ddoto}{\"{o}}
\newcommand{\ddotu}{\"{u}}

\newenvironment{entry}
{\begin{list}{X}%
		{%
			\setlength{\labelwidth}{55pt}%
			\setlength{\leftmargin}{\labelwidth}
			\addtolength{\leftmargin}{\labelsep}%
			\setlength{\itemsep}{.4pc}
	}%
}%
{\end{list}}     

\begin{document}

\title[Estimates with ``smooth'' weights]{``Small step'' remodeling and counterexamples for weighted estimates with arbitrarily ``smooth'' weights}
\author{S.~Kakaroumpas} 
\address{S.~Kakaroumpas: Department of Mathematics \\ Brown University \\ Providence, RI 02912 \\ USA}
\curraddr{Institut f\"{u}r Mathematik\\ Julius-Maximilians-Universit\"{a}t W\"{u}rzburg\\97074 W\ddotu rzburg\\ Germany}
\email{spyridon.kakaroumpas@mathematik.uni-wuerzburg.de}
\author{S.~Treil}
\thanks{Supported in part by the National Science Foundation under the grants  DMS-1600139, DMS-1856719.} 
\address{S.~Treil: Department of Mathematics \\ Brown University \\ Providence, RI 02912 \\ USA}
\email{treil@math.brown.edu}
\date{}

\begin{abstract}
For an $A_p$ weight $w$ the norm of the Hilbert Transform in $L^p(w)$, $1<p<\infty$ is estimated by $[w]\ci{A_p}^s$, where $[w]\ci{A_p}$ is the $A_p$ characteristic of the weight $w$ and $s = \max(1,1/(p-1))$; as simple examples with power weights show, these estimates are sharp. 

A natural question to ask, is  whether it is possible to improve the exponent $s$ in the above estimate if one replaces the $A_p$ characteristic by its ``fattened'' version, where the averages are replaced by Poisson-like averages. For power weights (for example with $p=2$ and Poisson averages) one can see that there is indeed an improvement in the exponent: but is it true for general weights?

In this paper we show that the optimal exponent $s$ remains the same by constructing counterexamples for arbitrarily ``smooth'' weights (in the sense that the doubling constant is arbitrarily close to $2$), so the ``fattened'' $A_p$ characteristic is equivalent to the classical one, and such that $\|T\|\ci{L^p(w)} \sim [w]\ci{A_p}^s$. 

 We use the ideas from the unpublished manuscript by F. Nazarov disproving Sarason's conjecture. We start from simple classical  counterexamples for dyadic models, and then by using what we call ``small step construction'' we transform them into examples with weights that are arbitrarily dyadically smooth. F.~Nazarov had used Bellman function method to prove the existence of such examples, but our construction gives a way to get such examples from the standard dyadic ones. We then use a modification of ``remodeling'', introduced by J.~Bourgain and developed by F.~Nazarov, to get from examples for dyadic models to examples for the Hilbert transform. 

As an added bonus, we present a proof that the $L^p$ analog of Sarason's conjecture is false  for all $p$, $1<p<\infty$.
\end{abstract}

\maketitle
\setcounter{tocdepth}{1}

\tableofcontents
\setcounter{tocdepth}{2}

\pagebreak
\section*{Notation}

\begin{entry}
\item[$\1\ci{E}$] characteristic function of set $E$;
\item[$dx$] integration with respect to Lebesgue measure; 
\item[$|E|$] $d$-dimensional Lebesgue measure of a measurable set $E\subseteq\R^d$;
\item[$\La f\Ra\ci{E}$] average with respect to Lebesgue measure, $\La f\Ra\ci{E}:=\frac{1}{|E|}\int_{E}f(x)dx$;
\item[$L^{p}(w)$] weighted Lebesgues space, $\|f\|\ci{L^p(w)}^p := \int_{\R^d}|f(x)|^p w(x) dx$; 
\item[$\La f,g\Ra$] linear duality, $\La f,g\Ra =\int f(x) g(x) dx$;
\item[$w(I)$] Lebesgue integral of a weight $w$ over $I$, $w(I):=\int_{I}w(x)dx =\La w\Ra\ci I |I|$; 
\item[$p'$] H\"{o}lder conjugate exponent to $p$, $1/p+1/p'=1$; 
\item[$\cD$] family of all dyadic intervals in $\R$, or of all dyadic subintervals of $[0,1)$; 
\item[$\cD(I)$] family of all dyadic subintervals of a dyadic interval $I$, including $I$ itself; 
\item[$\ch(I)$] family of all dyadic children   of the dyadic interval $I$;
\item[$\ch^k(I)$] family of all dyadic descendants of order $k$  of the dyadic interval $I$, note that $\ch(I) =\ch^1(I)$; 
\item[$\ch^k(\cS)$] for a family $\cS$ of dyadic intervals  the collection $\ch^k(\cS)$ is defined as $\ch^k(\cS):= \bigcup_{I\in\cS} \ch^k(I)$, and $\ch(\cS) = \ch^1(\cS)$;
\item[$I_{-},\,I_{+}$] left, respectively right half of interval $I$; 
\item[$h\ci I$] $L^{\infty}$-normalized Haar function for interval $I$, $h\ci{I}:=\1\ci{I_{+}}-\1\ci{I_{-}}$ (note the non-standard normalization!); 
\item[$\Delta\ci{I}$] martingale difference operator,  
 $\displaystyle\Delta\ci{I}f:=\sum_{I'\in\ch(I)}\La f\Ra\ci{I'}\1\ci{I'}-\La f\Ra\ci{I}\1\ci{I}$; 
\item[$\ssDelta\ci{I}f$] difference of averages, $\ssDelta\ci{I}f := \La f\Ra\ci{I_+} - \La f\Ra\ci{I} = \left(\La f\Ra\ci{I_+} - \La f\Ra\ci{I_-}\right)/2 = \La fh\ci I\Ra\ci I$;
\end{entry}

Notation $x\lesssim  y$ means  $x\le Cy$ with an absolute constant $C<\infty$, and  $x\lesssim\ci{a,b,\ldots}y$ means that $C$ depends \emph{only} on $a, b, \ldots$; the notation $x\gtrsim y$ means $y\lesssim x$, and similarly for $x\gtrsim\ci{a,b,\ldots} y$.  We use  $x\sim y$ if \emph{both} $x\lesssim y$ and $x\gtrsim y$ hold, and   
$x\sim\ci{a,b\ldots} y$ is defined similarly. 

\section{Introduction}
\label{s:intro}

This paper deals with sharp weighted estimates for classical operators in harmonic analysis. Our starting point  is the famous Hunt--Muckenhoupt--Wheeden theorem \cite{hunt-muckenhoupt-wheeden}, which says that the so-called Muckenhoupt $A_{p}$ condition
\begin{equation}
\label{A_p}
\sup_{I}\left(\frac{1}{|I|}\int_{I}w(x)dx\right)\left(\frac{1}{|I|}\int_{I}w(x)^{-1/(p-1)}dx\right)^{p-1}=:[w]\ci{A_{p}}<\infty,
\end{equation}
(the supremum is taken over all intervals $I\subset\R$)  
is necessary and sufficient for the Hilbert transform $H$, 
\begin{equation*}
Hf(x):=\text{p.v.}\int_{\R}\frac{f(y)}{x-y}dy,
\end{equation*}
to be a bounded operator on the weighted space $L^{p}(w)$ ($1<p<\infty$). 

It is also well-known that condition \eqref{A_p} (with intervals replaced by cubes) is sufficient for the boundedness on weighted spaces of all \cz operators in any number of dimensions, and it is also necessary for the boundedness on weighted spaces of ``large" \cz operators, like the Riesz transforms. 

\begin{rem*}
Recall that by a weight  people usually understand a locally integrable non-negative function $w$, but to define $A_p$ characteristic $[w]\ci{A_p}$ in \eqref{A_p} one needs to assume that $w$ is positive a.e.  

However, if we interpret $1/0$ as $+\infty$, then for a non-trivial weight $w$ vanishing on a set of positive measure  we have $[w]\ci{A_p}=\infty$, so the condition \eqref{A_p}  fails for $w$. And it is easy to see that if a weight $w$ vanishes on a set of positive measure, the Hilbert transform is not well defined on $L^p(w)$ (take a function $f$ supported on the set where $w$ vanishes), so we can treat the Hilbert transform as unbounded in this case. 

So, one can say that the Hunt--Muckenhoupt--Wheeden theorem holds for arbitrary non-negative weights, if everything is interpreted the right way. However, to avoid confusing the reader with irrelevant technical details, we assume in this paper that a weight is always locally integrable a.e.~positive function. 
\end{rem*}

\subsection{Sharp estimates}
Qualitative results like the one mentioned above are usually easier to prove than quantitative counterparts. In fact, it had been an open problem for some time to find a sharp estimate of the norm of $H$ (and other \cz operators) over $L^p(w)$  in terms of the powers of the $A_p$ characteristic $[w]\ci{A_p}$ defined in \eqref{A_p} above. It was proved by S.~Petermichl in \cite{petermichl} that $\|H\|\ci{L^2(w)} \lesssim [w]\ci{A_2}$. She then proved the same estimate for the Riesz Transform, and after some results by different authors gradually expanding the class of operators for which such an estimate holds, the linear estimate $\| T\|\ci{L^2(w)}\lesssim\ci{T,d} [w]\ci{A_2}$ was established by T.~Hyt\"{o}nen \cite{hytonen} (where $d$ is the dimension of the underlying Euclidean space).

Using the method Rubio De Francia extrapolation (see e.g. \cite{extrapolation}), one then can show that for $p>2$ the estimate $\|T\|\ci{L^p(w)} \lesssim\ci{T,p,d} [w]\ci{A_p}$ holds; by duality one finally gets the estimate  $\|T\|\ci{L^p(w)} \lesssim\ci{T,p,d} [w]\ci{A_p}^{1/(p-1)}$ for $1<p<2$. 
Thus for $1<p<\infty$ one can write 
\begin{align}
\label{e:sharp 01}
\|T\|\ci{L^p(w)} \lesssim\ci{T, p, d} [w]\ci{A_p}^s, 
\end{align}
where
\begin{equation}
\label{charexp}
s= \max\lbrace 1, 1/(p-1))\rbrace.
\end{equation}

Note, that for the Hilbert transform the above estimate \eqref{e:sharp 01} is sharp. Namely,  even before the upper bound for the Hilbert transform was proved by S.~Petermichl \cite{petermichl}, it had already been shown
by S.~Buckley \cite{Buckley} that 
given $p\in(1,\infty)$ one can find  $A_p$ weights $w$ with arbitrarily large $[w]\ci{A_p}$ for which $\|H\|\ci{L^p(w)} \gtrsim_{p} [w]^{s}\ci{A_p}$ (with $s$ is given by \eqref{charexp}). It is also not hard to show that the estimate \eqref{e:sharp 01} is sharp for the Riesz transforms.

\subsection{Considering ``larger'' characteristics} A reasonable attempt to lower the optimal exponent $s$ given by \eqref{charexp} might involve considering ``larger'' variants of $A_p$ characteristics where weights are not averaged over intervals (or cubes) as in \eqref{A_p}, but rather integrated against kernels with slower decay. Such characteristics arise in fact naturally in many problems not directly related to sharp weighted estimates.

For instance, it was proved by the second author and A. Volberg in \cite{stoch} for $p=2$, and by  F. Nazarov and the second author in \cite{bellman} for general $p$, that the following ``fattened" $A_{p}$ condition%
\footnote{In fact, in both \cite{stoch} and \cite{bellman} the estimates with matrix-valued weights were considered, and the Hunt--Muckenhoupt--Wheeden theorem for the matrix-valued weights was obtained.  A matrix-valued analogue of the condition $[w]\ut{fat}\ci{A_p}<\infty$ was introduced there, and its necessity was proved. The necessity of the scalar condition follows immediately from the matrix-result, although just following the proofs from \cite{stoch},  \cite{bellman} and not bothering with the non-commutativity of the matrix-valued case gives a very simple proof for the scalar situation.}
 $A_p\ut{fat}$,  $[w]\ut{fat}\ci{A_p}<\infty$, where
\begin{equation}
\label{fat-A_p}
[w]\ut{fat}\ci{A_p}:=\sup_{\lambda\in\C_+}\left(\int_{\R }\frac{(\im(\lambda))^{p-1}}{|x-\lambda|^{p}}w(x)dx\right)\left(\int_{\R }\frac{(\im(\lambda))^{p'-1}}{|x-\lambda|^{p'}}w(x)^{-1/(p-1)}dx\right)^{p-1},
\end{equation}
is necessary for the boundedness of the Hilbert transform on the weighted space $L^{p}(w)$; here, $p'$ denotes the H\"{o}lder conjugate of $p$, $1/p+1/p'=1$. Note that for $p=2$, the integrals in (1.2) are just Poisson extensions of the weights $w$ and $w^{-1}$ (up to multiplicative constants), so one can think of $[w]\ut{fat}\ci{A_p}$ as a ``Poisson-like'' $A_p$ condition for any $1<p<\infty$. Motivation for considering such ``Poisson-like'' $A_p$ conditions stems from the theory of Toeplitz operators, see for example \cite[s.~7.9]{ProbBook3_v1}.

It is easy to see that $[w]\ci{A_{p}}\lesssim_{p}[w]\ci{A_{p}}\ut{fat}$. Since the $A_{p}$ condition is already sufficient for the boundedness of $H$ on $L^{p}(w)$, it follows that the $A_{p}$ condition and the ``fattened" $A_{p}$ condition $A_p\ut{fat}$ are equivalent. However, simple examples involving power weights show that for every fixed $p$, the two characteristics themselves are not equivalent: for any fixed $1<p<\infty$, one can find $A_{p}$ weights $w$ with arbitrarily large quotient $[w]\ci{A_{p}}\ut{fat}/[w]\ci{A_{p}}$. 
Moreover, it was shown in \cite{stoch} that for $p=2$ the lower bound for the Hilbert Transform
\[
\|H\|\ci{L^2(w)}\gtrsim \left([w]\ci{A_2}\ut{fat}\right)^{1/2}
\] 
holds for \emph{all} weights $w$.  

So one could hope that  a better estimate of the norm $\|T\|\ci{L^p(w)}$, and in particular of the norm $\|H\|\ci{L^p(w)}$, in terms of the ``fattened'' $A_p$ characteristic $[w]\ci{A_p}\ut{fat}$ in \eqref{fat-A_p} is possible. One could even hope, for example, that the estimate $\|H\|\ci{L^2(w)}\lesssim \left([w]\ci{A_2}\ut{fat}\right)^{1/2}$ holds. The main result of this paper destroys all such  hopes: we show that for the Hilbert transform $H$ there exist $A_p$ weights $w$ with arbitrarily large $A_p$ characteristic $[w]\ci{A_p}\ut{fat}$, such that $\|H\|\ci{L^p(w)}\gtrsim_{p} \left( [w]\ci{A_p}\ut{fat} \right)^s$, where $s$ is given by \eqref{charexp}.

\subsubsection{``Heat'' $A_p$ characteristics} In many problems it is natural to consider other kernels besides ``Poisson-like'' ones. For example, 
S.~Petermichl and A.~Volberg \cite{heating} considered a ``heat'' $A_p$ characteristic ($1<p<\infty$) given by
\begin{equation}
\label{heat-A_p}
[w]\ut{heat}\ci{A_p}:=\sup_{\substack{y\in\R^d\\t\in(0,\infty)}}\left(\int_{\R^d }\frac{1}{t^{d/2}}e^{-|x-y|^2/t}w(x)dx\right)\left(\int_{\R^d }\frac{1}{t^{d/2}}e^{-|x-y|^{2}/t}w(x)^{-1/(p-1)}dx\right)^{p-1}.
\end{equation}
It was shown in  \cite{heating}  that in sharp contrast to the ``Poisson-like'' case, the ``heat" $A_p$ characteristic in \eqref{heat-A_p} is essentially the same as the usual Muckenhoupt $A_p$ characteristic in \eqref{A_p}, more precisely
\begin{equation}
\label{usual heat comparability}
[w]\ci{A_p}\sim\ci{d,p}[w]\ci{A_p}\ut{heat}\qquad (1<p<\infty).
\end{equation}
This fact for $d=2$ was used \cite{heating}  to establish sharp weighted estimates for the Ahlfors--Beurling operator, which allowed the authors   to deduce that weakly quasiregular maps on the plane are quasiregular. 

In view of \eqref{usual heat comparability} the problems considered in this paper are trivial for the ``heat'' $A_p$ characteristic. 

\subsection{Weights and doubling constants}
\label{s:doubling}

For a weight $w$ on $\R$ we define its doubling constant $D_w$ as 
\begin{align*}
D_w:= \sup_I w(2I)/w(I), 
\end{align*}
where the supremum is taken over all intervals $I$ in $\R$. Here $2I$ is the interval with the same center as $I$ of length $2|I|$, and slightly abusing the notation we write $w(I)$ for $\int_I w dx$. 

It is easy to show that if the doubling constant of the weight $w$ is bounded by $2+\delta$ for sufficiently small $\delta$, then we have uniformly over all $\lambda\in\C_{+}$ the estimate
\begin{align}
\label{Poiss le avg}
\int_{\R }\frac{(\im(\lambda))^{p-1}}{|x-\lambda|^{p}}w(x)dx \lesssim_{p} |I_\lambda|^{-1}\int_{I_\lambda} w(x)dx,
\end{align}
where $I_\lambda$ is the interval $[\re(\lambda) -\im(\lambda), \re(\lambda) +\im(\lambda) ]$.  We emphasize that the particular function $(\im \lambda)^{p-1}/|x-\lambda|^{p}$ in the left-hand side of \eqref{Poiss le avg} is of no importance here; any ``reasonable'' approximate identity on the real line can be used in its place.

Thus, if the doubling constants  of the weights $w$ and $\sigma=w^{-1/(p-1)}$ are bounded by $2+\delta$ for sufficiently small $\delta$, then the $A_p$ characteristics $[w]\ci{A_p}$ and $[w]\ci{A_p}\ut{fat}$ are equivalent in the sense of two sided estimate.

\subsection{Main results}
The  main result of this paper is the following theorem.
\begin{thm}\label{t:main res}
Given $p\in(1,\infty)$,  $M>2$ and arbitrarily small $\delta>0$, there exists an $A_p$ weight $w$ on $\R$ with $M\leq [w]_{\ci{A_{p}}}\leq C(p)M$, such that the doubling constants of the weights $w$ and $\sigma=w^{-1/(p-1)}$ are bounded by $2+\delta$ and
\begin{equation*}
\|H\|\ci{L^p(w)}\ge c(p) M^s, \qquad s=\max\lbrace1,{1}/{p-1}\rbrace. 
\end{equation*}
\end{thm}
By the above discussion about the equivalence of $A_p$ characteristics $[w]\ci{A_p}$ and $[w]\ci{A_p}\ut{fat}$, we can see that   Theorem \ref{t:main res} implies the following corollary.
\begin{cor}
\label{c:linear lower}
Given $p\in(1,\infty)$,  $M>2$, there exists a weight $w$ on $\R $ with $M\le [w]\ci{A_{p}}\ut{fat}\le C(p) M$, such that
\begin{equation*}
\Vert H\Vert \ci{L^{p}(w)}\ge c(p)M^{s},\qquad s=\max\lbrace 1,1/(p-1)\rbrace.
\end{equation*}
\end{cor}

\subsubsection{Two weight estimates and Sarason's conjecture}

One of the main technical tools used in this paper is inspired by the unpublished manuscript \cite{Nazarov} by F.~Nazarov, where he provided a counterexample to the so-called Sarason's conjecture. Let us briefly recall this conjecture. 

It is  natural  to consider \textit{two-weight} estimates for the Hilbert transform and other \cz operators, i.e.~to ask when they are bounded operators from $L^p(v)$ to $L^p(w)$ for potentially \textit{different} weights $v,w$. It is easy to show that the two weight $A_p$ condition
\begin{align}
\label{2w-A_p}
\sup_{I}\left(\frac{1}{|I|}\int_{I}w(x)dx\right)\left(\frac{1}{|I|}\int_{I}v(x)^{-1/(p-1)}dx\right)^{p-1}=:[w,v^{-1/(p-1)}]\ci{A_{p}}<\infty,
\end{align}
is necessary for the  Hilbert transform to be a bounded operator from $L^p(v)$ to $L^p(w)$ ($1<p<\infty$). However, as simple examples show, this condition is not sufficient (for the reader's convenience we supply an example in Subsection \ref{s: counter two-weight Muck} in the Appendix).

It had been shown long ago by the second author that the following ``fattened'' two weight $A_p$ condition 
\begin{equation}
\label{2w-A_p-fat}
\sup_{\lambda\in\mathbb{C}_{+}}\left(\int_{\R }\frac{(\text{Im}(\lambda))^{p-1}}{|x-\lambda|^{p}}w(x)dx\right)\left(\int_{\R }\frac{(\text{Im}(\lambda))^{p'-1}}{|x-\lambda|^{p'}}v(x)^{-1/(p-1)}dx\right)^{p-1}<\infty
\end{equation}
is also necessary for the Hilbert transform to act boundedly  from $L^p(v)$ to $L^p(w)$.  Note, that unlike the one-weight case, the two-weight conditions \eqref{2w-A_p} and \eqref{2w-A_p-fat} are not equivalent; simple examples can be easily constructed.

The Poisson averages are less localized than the averages over intervals, so D.~Sarason hoped that for $p=2$ the two weight Poisson $A_2$ condition \eqref{2w-A_p-fat} would capture correctly the ``far'' action of the Hilbert transform.  In \cite[s.~7.9]{ProbBook3_v1} he conjectured that (for $p=2$) the  Poisson $A_2$ condition \eqref{2w-A_p-fat}  is necessary and sufficient  for the Hilbert transform to be a bounded operator from $L^{2}(v)$ to $L^{2}(w)$.%
\footnote{It is interesting that when D.~Sarason was stating his conjecture he was not aware of the necessity of the two weight Poisson $A_2$ condition. The proof of necessity was presented to him by the second author, and this is exactly the proof presented (with attribution) in \cite[s.~7.9]{ProbBook3_v1}. 

The problem in   \cite[s.~7.9]{ProbBook3_v1} was stated a bit different, but it was equivalent to the two weight estimate for the Hilbert transform. The proof of necessity was presented there only for $p=2$, but the same proof works for all $p$, $1<p<\infty$.} 

This conjecture was disproved by F.~Nazarov in \cite{Nazarov}. In this paper we extend Nazarov's result to all $p\in(1,\infty)$ (not just $p=2$). While our proof relies heavily on the machinery developed in \cite{Nazarov}, we introduce some crucial new ideas, allowing us to treat the case of $p\neq 2$. We should also mention that our counterexample is a ``constructive'' one; unlike \cite{Nazarov} we are not using the Bellman function method. 

We prove the following theorem:

\begin{thm}\label{t:dispro sar}
Given $p\in(1,\infty)$, there exist weights $w,v$ on $\R $ satisfying \eqref{2w-A_p-fat}, such that the Hilbert transform is not a bounded operator acting from $L^p(v)$ to $L^p(w)$. In particular, this means than there exists $f\in L^p(v)$ such that $\|Hf\|\ci{L^{p}(w)}=\infty$.
\end{thm}

In light of the discussion in Section \ref{s:doubling} the above theorem follows from the corresponding counterexample with ``smooth'' weights (i.e.~weights with small doubling constants). Namely, we prove the following theorem, which implies the above Theorem \ref{t:dispro sar}. 

\begin{thm}
\label{t:cex Sarason smooth}
Given $p\in(1,\infty)$ and arbitrarily small $\delta>0$, there exist weights $w,v$ on $\R $ satisfying \eqref{2w-A_p}, such that the doubling constants of the weights $w$ and $\sigma = v^{-1/(p-1)}$ are bounded by $2+\delta$ and the Hilbert transform is not a bounded operator acting from $L^p(v)$ to $L^p(w)$. In particular, this means than there exists $f\in L^p(v)$ such that $\|Hf\|\ci{L^{p}(w)}=\infty$.
\end{thm}

\subsubsection{A counterintuitive result}
\label{s:counterintuitive}

It is an easy exercise to construct a weight with a prescribed $A_p$ characteristic. Moreover, one can find a weight taking only $2$ values. What is more interesting, and is not completely clear, is that in fact one can find such a weight with doubling constant arbitrarily  close to $2$.

\begin{prop}\label{p:counterint}
Let $p\in(1,\infty)$. Then, given $Q>1$ and arbitrarily small $\e>0$, there exists a weight $w$ on $\R$ taking only $2$ values, with $Q\leq [w]\ci{A_{p}}\leq c(p)Q$, such that the doubling constants of the weights $w$ and $\sigma=w^{-1/(p-1)}$ are bounded by $2+\e$. 
\end{prop}

\subsection{Plan of the paper.}

Our general strategy is as follows. We start with simple examples that give the desired lower bounds for dyadic (martingale) analogues of the Hilbert transform, in particular, for the so-called Haar shifts. These examples are simple ones, obtained as easy modifications of known examples; we call them the ``large step'' examples, to emphasize that we do not have any non-trivial bounds on the doubling constants of the weights involved. This is done in Section \ref{s:LS ex}.

From these examples we construct in Section \ref{s:SM constr} the so-called ``small step'' examples, where we preserve the desired lower bounds, but can make the so-called \emph{dyadic smoothness constant} (see the relevant definition in Subsection \ref{s:dyadic doubling} below) of the weights as close to $1$ as we want.  We present a general construction that allows us to do so. This step is absent in \cite{Nazarov}, where the ``small step'' example is obtained implicitly via the Bellman function method. 

The next step is to apply \emph{remodeling}, introduced in \cite{Nazarov}, which serves two purposes. First, it allows us to get from weights with dyadic smoothness constants arbitrarily close to $1$ to weights with doubling constants arbitrarily close to $2$. And second (and equally important) it allows us to get from the lower bounds for Haar shifts to the lower bounds for the Hilbert transform, which we need. However, the original remodeling from \cite{Nazarov} does not handle the one-weight situation well, since typically it gives a two-weight situation as its output. So to handle the one-weight situation we introduce the so-called \emph{iterated remodeling}, that allows us to prove Theorem \ref{t:main res} (and so Corollary \ref{c:linear lower}). The general method of iterated remodeling is presented in Section \ref{s:remodel}, while Subsection \ref{s:smooth Hilb} contains the particular application for the Hilbert transform. Subsection \ref{s:smooth Haar mult} describes analogous examples in the (easier) cases of Haar multipliers and the dyadic Hardy--Littlewood maximal function. Moreover, Subsection \ref{s:counterint proof} contains the counterintuitive result of Proposition \ref{p:counterint}, deduced as a byproduct of our general constructions.

Through a standard \emph{direct sum of singularities} type construction, the family of examples for the Hilbert transform yields in Subsection \ref{s: counter sar} a counterexample to the $L^p$ version of the Sarason's conjecture, (i.e.~Theorem \ref{t:cex Sarason smooth}, and therefore Theorem \ref{t:dispro sar}), so we are done in the two-weight case as well.

The main constructions of this paper exploit the usual structure of a filtered probability space on the unit interval $[0,1)$, and the fundamental correspondences between functions and martingales on the one hand, and martingales and random walks on graphs on the other hand. We briefly recall the relevant definitions and results in Subsections \ref{s:dyad filtr}, \ref{s:func-mart} and \ref{s:mart-rand walk}.

Finally, in the Appendix (Section \ref{s: appendix}) we collect a few results used throughout the paper: probability theoretic results on random walks (Subsection \ref{s: rand walks}), two remarks about ``stopping on the lower hyperbola" (Subsection \ref{s: stop lower hyperb}) and ``getting only a little above the upper hyperbola" (Subsection \ref{s: upper hyperb}), and we repeat the proofs of F. Nazarov's lemmas about Muckenhoupt characteristics and doubling constants from \cite{Nazarov} (Subsection\ref{s: naz lemma proof}).

\textbf{Acknowledgements.} We are grateful to Alexander Barron for reading a draft of the manuscript and for pointing out typos and other obscurities, and to the anonymous referee for the valuable feedback.

\section{Preliminaries}
\label{s:Prelim}

\subsection{Symmetric ``two weight'' setup.}
\label{s:symmetric setup}

In weighted estimates it is customary to rewrite a problem in a symmetric two-weight setup. For example, in an one-weight situation involving a weight $w$ (Theorem \ref{t:main res}) let us introduce an auxiliary weight $\sigma:= w^{-1/(p-1)}$ (the reader should have noticed that it already appears in the statement of Theorem \ref{t:main res}). If we denote $\wt{f}:= \sigma^{-1} f$, so $f=\wt{f}\sigma$, then 
\begin{align*}
\|\wt{f}\|\ci{L^p(\sigma)} = \|f\|\ci{L^{p}(w)} \qquad \text{and}\qquad T f = T(\wt{f}\sigma) ,
\end{align*}
for any linear operator $T$. 
Thus any weighted estimate of an operator $T$ over $L^p(w)$ is equivalent to the estimate of the operator $\wt{f} \mapsto T(\wt{f}\sigma)$ acting from $L^p(\sigma)$ to $L^p(w)$; note that if $T$ is an integral operator, then in the operator $f\mapsto T(f\sigma)$ integration is performed against the  measure that defines the norm in the domain $L^p(\sigma)$.

To prove Theorem \ref{t:main res} one needs to find a non-zero $f\in L^p(w)$ such that $\| H f\|\ci{L^p(w)}\ge c(p) \|f\|\ci{L^p(w)}$. This is equivalent to finding a non-zero $f\in L^p(\sigma)$ (we omit the tilde over $f$ here) such that 
\begin{align}
\label{main ineq 01}
\| H(f \sigma)\|\ci{L^p(w)}\ge c(p) M^s \|f\|\ci{L^p(\sigma)} ;
\end{align}
here, recall, $M\le [w]\ci{A_p}\le C(p) M$, and $\sigma = w^{-1/(p-1)}$. The weights $w$ and $\sigma$ should have doubling constants as close to 2 as we want. 

In a two-weight situation involving two weights $w$ and $v$ (Theorem \ref{t:cex Sarason smooth}) we  denote $\sigma = v^{-1/(p-1)}$. To prove Theorem \ref{t:cex Sarason smooth} we construct for arbitrarily large $R$ weights $\sigma$ and $w$ with doubling constants arbitrarily close to $2$ such that 
\begin{align*}
\langle w \rangle\ci I  \langle \sigma \rangle\ci I^{p-1} \le C(p)
\end{align*}
($C(p)$ does not depend on $R$) and a non-zero $f\in L^p(\sigma)$ such that
\begin{align}
\label{main ineq 02}
\| H(f\sigma)\|\ci{L^p(w)}\ge R \|f\|\ci{L^p(\sigma)}.
\end{align}

\subsection{Dyadic intervals and martingale differences}
\label{s:dyadic-mart}

For definiteness, by an interval we will always mean a half-open interval $[a,b)$. For an interval $I$ we denote by $I_{+}$ and $I_{-}$ its right and left halves respectively. The symbol $h\ci I $ denotes the $L^\infty$ normalized Haar function, 
\begin{align}
\label{Haar 01}
h \ci I = \1\ci{I_+ }  -\1\ci{I_-}. 
\end{align}
We emphasize, that in this paper we always use the $L^\infty$ normalized Haar functions. 

We say that two intervals $I,J$ in $\R $ are adjacent if $I\cap J=\varnothing$, and they have a common endpoint.

An interval $I$ in $\R $ is called a dyadic interval if $I=[k2^{n},(k+1)2^{n})$ for some $n,k\in\mathbb{Z}$. We denote by $\cD$ the family of all dyadic intervals in $\R$. For a dyadic interval $I$ we denote by $\cD(I)$ the collection of its dyadic subintervals (including $I$ itself). When there is no danger of confusion, we will denote $\cD([0,1))$ by $\cD$, abusing the notation. For all $I\in\cD$, the number $-\log_{2}(|I|)$ will be called generation of the interval $I$. Moreover, for all $N\in\mathbb{N}$ and for all $I\in\cD$, we denote by $\ch^{N}(I)$ (simply $\ch(I)$ if $N=1$) the family of all dyadic subintervals of $I$ of length $2^{-N}|I|$,  and if $\mathcal{G}$ is a family of dyadic intervals, then we set $\ch^{N}(\mathcal{G}):=\bigcup\ci{I\in\mathcal{G}}\ch^{N}(I)$. Moreover, if $\mathcal{G}$ is a family of pairwise disjoint dyadic intervals then we denote
\begin{equation*}
\mathbb{E}\ci{\mathcal{G}}[f]:=\sum_{I\in\mathcal{G}}\La f\Ra\ci{I}\1\ci{I}.
\end{equation*}

For all $f\in L^\1_{\text{loc}}(\R )$ and for all $I\in\cD$, we denote by $\Delta\ci{I}f$ the martingale difference
\begin{equation*}
\Delta\ci{I}f:=\sum_{I'\in\ch(I)}\La f\Ra\ci{I'}\1\ci{I'}-\La f\Ra\ci{I}\1\ci{I}=\La f\Ra{\ci{I_{+}}}1{\ci{I_{+}}}+\La f\Ra{\ci{I_{-}}}1{\ci{I_{-}}}-\La f\Ra\ci{I}\1\ci{I},
\end{equation*}
and by $\ssDelta\ci{I}f$ the difference of averages (or Haar coefficient)
\begin{equation*}
\ssDelta\ci{I}f := \La f\Ra\ci{I_+} - \La f\Ra\ci{I} = \frac{\La f\Ra\ci{I_+} - \La f\Ra\ci{I_-}}{2} = \La fh\ci I\Ra\ci I.
\end{equation*}
Notice that martingale differences and Haar coefficients are related by
\begin{equation*}
\Delta\ci{I}f=(\ssDelta\ci{I}f)h\ci{I}.
\end{equation*}

\subsection{Weights and doubling constants}
\label{s:dyadic doubling}

Given weights $w,\sigma$ on $\R $ and $p\in(1,\infty)$, we define the joint dyadic Muckenhoupt $A_{p}$ characteristic of $w,\sigma$ by
\begin{equation*}
[w,\sigma]\ci{A_{p},\cD}:=\sup_{I\in\cD}\La w\Ra\ci{I}\La\sigma\Ra\ci{I}^{p-1}
\end{equation*}
and the dyadic Muckenhoupt characteristic of $w$ by $[w]\ci{A_{p},\cD}:=[w,w^{-1/(p-1)}]\ci{A_{p},\cD}$. Following \cite[\S 1]{Nazarov}, we define the smoothness constant
\begin{equation*}
S_{w}=\sup_{I}\max\left(\frac{\La w\Ra{\ci{I_{-}}}}{\La w\Ra{\ci{I_{+}}}},\frac{\La w\Ra{\ci{I_{+}}}}{\La w\Ra{\ci{I_{-}}}}\right),
\end{equation*}
where the supremum is taken over all intervals $I$ in $\R$, and the dyadic smoothness constant
\begin{equation*}
S_{w}\ut{d}=\sup_{I\in\cD}\max\left(\frac{\La w\Ra{\ci{I_{-}}}}{\La w\Ra{\ci{I_{+}}}},\frac{\La w\Ra{\ci{I_{+}}}}{\La w\Ra{\ci{I_{-}}}}\right).
\end{equation*}
It is easy to see that $D_{w}\leq S_{w}+1$. Note also that $1\leq S_{w}\ut{d}\leq S_{w}$. Moreover, as in \cite[\S 6]{Nazarov}, we define the strong dyadic smoothness constant
\begin{equation*}
S_{w}\ut{sd}=\sup_{I,J}\frac{\La w\Ra\ci{I}}{\La w\Ra_{J}},
\end{equation*}
where the supremum is taken over all adjacent intervals $I,J\in\cD$ with $|I|=|J|$. Obviously $S_{w}\ut{sd}\geq S_{w}\ut{d}$. Of course all these definitions can be given over $[0,1)$, and we will use the same notation as above for Muckenhoupt characteristics and smoothness constants over $[0,1)$ (note that local integrability over $[0,1)$ means here integrability over $[0,1)$).

It turns out that the strong dyadic smoothness constant can provide some control over the smoothness constant, and the dyadic Muckenhoupt characteristic over the full Muckenhoupt characteristic, provided the strong dyadic smoothness constant is sufficiently close to 1.

\begin{lm}\label{l:naz1}(F. Nazarov, \cite[\S 6]{Nazarov})
For all $\e>0$, there exists $\delta=\delta(\e)>0$, such that for all weights $w$ on $\R $ with $S\ut{sd}_{w}\leq 1+\delta$ there holds $S_{w}\leq 1+\e$.
\end{lm}

\begin{lm}\label{l:naz2}(F. Nazarov, \cite[\S 11]{Nazarov})
For all $p\in(1,\infty)$, there exists $\delta=\delta(p)>0$, such that for all weights $w,\sigma$ on $\R $ with $[w,\sigma]\ci{A_{p},\cD}<\infty$ and $S\ut{sd}_{w},S\ut{sd}_{\sigma}\leq 1+\delta$ there holds $[w,\sigma]\ci{A_{p}}\leq(5/4)[w,\sigma]\ci{A_{p},\cD}$.
\end{lm}

For reasons of completeness, we give the proofs of both these lemmas in Subsection \ref{s: naz lemma proof} in the Appendix. In this paper, the phrase ``smoothness of weights" will always refer to the above smoothness constants.

So we see that in order to dominate Muckenhoupt characteristics and doubling constants, it suffices to dominate strong dyadic smoothness constants and dyadic Muckenhoupt characteristics. We will see in Section \ref{s:remodel} that F. Nazarov's method of remodeling will allow us to dominate strong dyadic smoothness constants by dyadic smoothness constants.

\subsection{Dyadic filtration}
\label{s:dyad filtr}

For $n=0,1,2,\ldots$, set
\begin{equation*}
\cD_{n}=\lbrace I\in\cD([0,1)):\;|I|=2^{-n}\rbrace,
\end{equation*}
and let $\cF_n$ be the $\sigma$-algebra of subsets of $[0,1)$ generated by the family $\cD_n$, 
 i.e. the smallest $\sigma$-algebra of subsets of $[0,1)$ containing $\cD_n$. 
Clearly $\cF_{n}\subseteq\cF_{n+1}$, for all $n=0,1,2,\ldots$, so the sequence $\mathbb{F}:=(\cF_n)^{\infty}_{n=0}$ of $\sigma$-algebras is a filtration on $[0,1)$ (sometimes called the dyadic filtration). Notice that the Borel $\sigma$-algebra $\cF$ of $[0,1)$ is the smallest $\sigma$-algebra containing all $\cF_n$, or equivalently, the $\sigma$-algebra  generated by the family $\bigcup_{n=0}^{\infty}\cD_n$. 

Taking for the probability measure $\mathbb{P}$ the Lebesgue measure on $[0,1)$, we can see that $([0,1),\mathcal{F},\mathbb{P},\mathbb{F})$ is a filtered probability space. Denote by $\E_n$ the conditional expectation with respect to the $\sigma$-algebra $\cF_n$, $\E_n[f] = \E(f|\cF_n)$. The operator $\E_n$ admits a simple formula
\begin{equation*}
\mathbb{E}_n[f]=\sum_{I\in\mathcal{D}_{n}}\La f\Ra\ci{I}\1\ci{I}.
\end{equation*}
We will use the symbol $\E$ for the expectation operator, $\E f:=\E_0 f = \La f\Ra\ci{[0,1)}  \1\ci{[0,1)}$. 

Recall that a sequence $(X_n)^{\infty}_{n=0}$ of integrable functions on a filtered probability space is called a martingale  if $X_n$ is $\cF_n$-measurable and
\begin{equation*}
\mathbb{E}_{n}[X_{n+1}]=X_n,
\end{equation*}
for all $n=0,1,2,\ldots$. In the sequel, all martingales on $[0,1)$ will always be considered with respect to the dyadic filtration (and called then just dyadic martingales).

Note that every dyadic interval can be given the structure of a filtered probability space by simply translating and rescaling the unit interval.

\subsection{Functions and martingales}
\label{s:func-mart}

A function $f\in L^1([0,1);\R ^{N})$  naturally induces an $\R ^N$-valued martingale $X=(X_{n})_{n=0}^{\infty}$ on $[0,1)$, 
\begin{equation*}
X_{n}=\E_n f = \sum_{I\in\mathcal{D}_{n}}\La f\Ra\ci{I}\1\ci{I},\qquad n=0,1,2,\ldots.
\end{equation*}
Note, that not all martingales are induced by a function, only the so-called \emph{uniformly integrable} ones. However, in this paper we will be considering only uniformly bounded martingales, which are trivially uniformly integrable, and so are always induced by a function. 

If a martingale $X$ is induced by a function $f$, then $f$ can be easily restored from $X$, namely $X_n\to f$ a.e.~and in $L^1$; for uniformly bounded martingales we have, in fact, convergence in all $L^p$, $1\le p<\infty$.   

It turn out that in many problems of harmonic analysis it is more convenient to work not with a function, but with the induced martingale.
In our context that means that we keep track of averages of functions, instead of the functions themselves. In our examples, we deal with functions $w,\sigma,f,g,$ and we are keeping track of the averages of functions $w,\sigma,\bff=:f\sigma,\bg=:gw$ (then $f=\bff/\sigma$ and $g=\bg/w$). 

\subsection{Martingales and random walks}
\label{s:mart-rand walk}

Let $X=(X_{n})_{n=0}^{\infty}$ be an $\R ^N$-valued martingale on $[0,1)$. For 
$I\in\cD_n$  the function $X_n$ is constant on $I$; we denote by $\La X\Ra\ci{I}$ its constant value there. Note that if the martingale $X$ is induced by a function,  which we, slightly abusing the notation, also denote by $X$,   then  $\La X\Ra\ci{I}$ as defined above is indeed the average of the function $X$.  It is easy to see that
 \begin{equation}
 \label{e:MartDyn}
 \La X\Ra\ci{I}=\frac{\La X\Ra{\ci{I_{-}}}+\La X\Ra{\ci{I_{+}}}}{2},\qquad\forall I\in\cD.
 \end{equation}
In the language of \cite[Subsection 5.1]{failure} the above identity says that the family $\lbrace\La X\Ra\ci{I}\rbrace{\ci{I\in\cD}}$ has ``martingale dynamics''. 

We also define the difference of averages (or Haar coefficient) $\ssDelta\ci{I}X$,
\begin{equation*}
\ssDelta\ci{I}X:=\La X \Ra{\ci{I_{+}}}-\La X\Ra\ci{I}=\frac{\La X\Ra{\ci{I_{+}}}-\La X\Ra{\ci{I_{-}}}}{2},\qquad I\in\cD.
\end{equation*}
Again, if, slightly abusing the notation, we denote by $X$ the function inducing the uniformly bounded martingale $X$, then the two definitions of $\ssDelta\ci{I}X$ are consistent.

The dyadic martingale $X$ can be interpreted as a random walk on an image of a binary tree; in what follows we will call this image the \emph{graph} of $X$.  

To describe this random walk, notice that the collection $\cD$ of dyadic intervals carries a natural structure of the full  binary tree, with vertices being the dyadic intervals, and the edges connecting an interval with its two children. 

The collection $\cD$ of dyadic intervals can be naturally interpreted as the standard random walk on the full dyadic tree, where one moves from a vertex $I$ to each of its children with probability $1/2$. 
Each point $x\in[0,1)$ represents a trajectory on the full binary tree  $\cD$, that at the time $n$ it is at the unique $I\in\cD_n$ containing $x$. 

The martingale $X$ naturally induces a map from the dyadic tree $\cD$ to $\R^N$, where the vertex corresponding to $I\in\cD$ goes to the point $\La X\Ra\ci I\in\R^N$, and the edges go to straight line segments connecting the corresponding points; we will call this image the \emph{graph} of $X$. The random walk on the dyadic tree $\cD$ is then mapped to the random walk on the graph of $X$, that moves from a point   $\La X\Ra\ci{I}$ by the steps $\pm {\ssDelta\ci{I}X}$ with equal probabilities $1/2$.

In view of the martingale dynamics identity \eqref{e:MartDyn} above, $\La X\Ra\ci{I}$ always occupies the midpoint of the straight line segment connecting $\La X\Ra\ci{I_{-}}$ and $\La X\Ra\ci{I_{+}}$; we will say in what follows  that this segment corresponds to the interval $I\in\cD$. 



The interpretation of dyadic martingales as random walks on images of a binary tree gives helpful intuition into the constructions we are using. While it is not required for the formal construction, we feel that it could help the reader to understand and visualize what is going on.

In our examples, we deal with functions $w,\sigma,f,g$, where $w=\sigma^{-1/(p-1)}$ for some $1<p<\infty$, and random walks correspond to the martingales induced by the functions $w,\sigma,\bff:=f\sigma,\bg:=gw$. Our transforms will be applied to the functions $w,\sigma,\bff,\bg$, to produce functions $\wt{w},\wt{\sigma},\wt{\bff},\wt{\bg}$ respectively. The random walk corresponding to the martingale induced by the function $(w,\sigma)$ terminates with probability 1 on the hyperbola given in the $uv$-plane by $uv^{p-1}=1$, because $w\sigma^{p-1}=1$ a.e. on $[0,1)$. Our transforms will need to guarantee that the new weights $\wt{w},\wt{\sigma}$ we get continue to satisfy this relation. As we will see, on the level of weights our transforms will amount to composition with measure-preserving transformations, and therefore such relations will be automatically preserved. In addition, we will see that the relevant weighted norms $\Vert \wt{f}\Vert\ci{L^{p}(\wt{\sigma})}$, $\Vert\wt{g}\Vert\ci{L^{p'}(\wt{w})}$ are not larger (up to constants depending only on $p$) than $\Vert f\Vert\ci{L^{p}(\sigma)}$, $\Vert g\Vert\ci{L^{p'}(w)}$ respectively, where $\wt{f}=\wt{\bff}/\wt{\sigma}$ and $\wt{g}=\wt{\bg}/\wt{w}$.

\section{``Large step'' examples}
\label{s:LS ex}

We construct in this section ``large step" examples for the Haar multiplier, and for a special type of Haar shift, defined in Subsection \ref{s:LS Haar shift}. 

Let $p\in(1,\infty)$ and $M>2$. Set $\beta=1-\frac{1}{2Me}\in\left(\frac{1}{2},1\right)$. Set $I_0=[0,1)$ and $I_{n}=\left[0,\frac{1}{2^{n}}\right)$, $J_{n}=\left[\frac{1}{2^{n}},\frac{1}{2^{n-1}}\right)$, for all $n=1,2,\ldots$. Consider the functions $w,\sigma$ on $[0,1)$ given by
\begin{equation*}
w=\sum_{n=1}^{\infty}2^{n\beta}1{\ci{J_{n}}},\;\;\;
\sigma=\sum_{n=1}^{\infty}2^{-n\beta/(p-1)}1{\ci{J_{n}}}.
\end{equation*}
Then, $w,\sigma$ are weights on $[0,1)$ with $\sigma=w^{-1/(p-1)}$. Note that $w([0,1))\sim M$ and $\sigma([0,1))\sim_{p} 1$. Notice that $x^{-\beta}\leq w(x)\leq 2^{\beta}x^{-\beta}$ and $2^{-\beta/(p-1)}x^{\beta/(p-1)}\leq\sigma(x)\leq x^{\beta/(p-1)}$, for all $x\in(0,1)$. Then, direct computation shows that
\begin{equation*}
M\leq2^{-\beta}\frac{(1-\beta)^{-1}}{e}\leq\La w\Ra{\ci{I_{n}}}\La\sigma\Ra{\ci{I_{n}}}^{p-1}\leq 2^{\beta}(1-\beta)^{-1}\leq 4Me,\qquad\forall n=0,1,2,\ldots.
\end{equation*}
It follows that $M\leq[w]\ci{A_{p},\cD}\leq 4Me$. Direct computation gives also $\ssDelta{\ci{I_{n}}}w<0$ and $-\ssDelta{\ci{I_{n}}}w\sim(1-\beta)^{-1}2^{n\beta}$, for all $n=0,1,2,\ldots$.

Consider the uniformly integrable real-valued martingales $X,Y$ induced by $w,\sigma$ respectively. Note that by a very easy application of Jensen's inequality as in \cite[Lemma 4.1]{Convex} we have $X_{n}Y_{n}^{p-1}\geq 1$, for all $n=0,1,2\ldots$. Also note that the graph of the martingale $Z=(X,Y)$ consists of  the straight line segments connecting $\La Z\Ra{\ci{J_{n}}}$ and $\La Z\Ra{\ci{I_{n}}}$, for $n=1,2,\ldots$, see Figure \ref{hyperbolas_large} (the constant $c_{p,\beta}$ in Figure \ref{hyperbolas_large} satisfies $1\leq c_{p,\beta}\leq4e$).

\begin{figure}[h]\centering
\includegraphics[width=\linewidth]{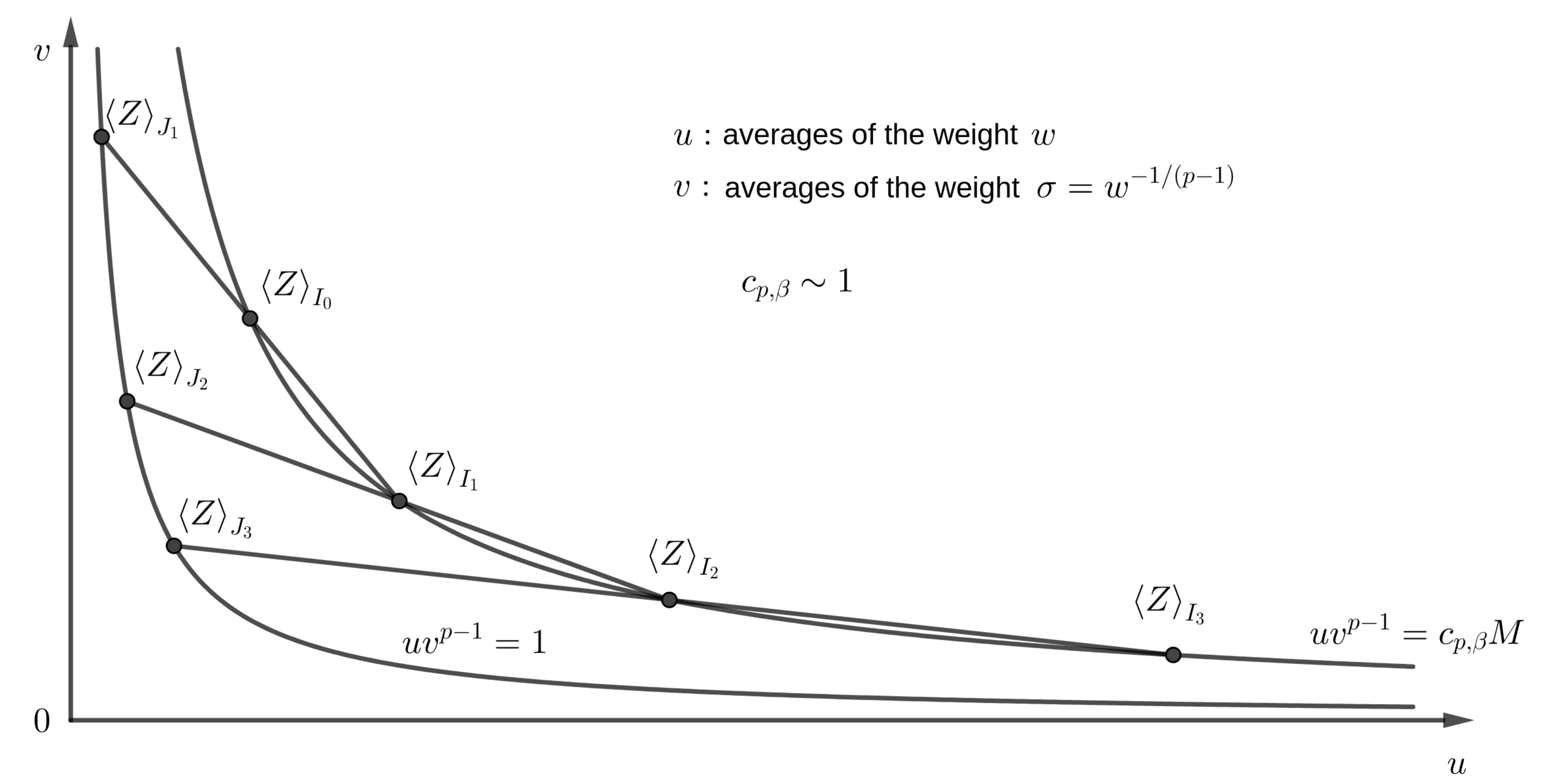}
\caption{Random walk in the $uv$-plane corresponding to the pair of weights $(w,\sigma)$}
\label{hyperbolas_large}
\end{figure}

Notice moreover that $S\ut{d}_{w}\sim(1-\beta)^{-1}\sim M$, therefore we have no control over the dyadic smoothness constant of $w$.

We will now truncate the weights $w,\sigma$. We have
\begin{equation*}
\sum_{n=0}^{\infty}2^{n(\beta-1)}=\frac{1}{1-2^{\beta-1}}\gtrsim(1-\beta)^{-1}=2Me.
\end{equation*}
Therefore, there exists a positive integer $N=N{\ci{M}}$ greater than 1, such that
\begin{equation*}
\sum_{n=0}^{N}2^{n(\beta-1)}\gtrsim M.
\end{equation*}
The folllowing lemma, whose proof is given in Subsection \ref{s: stop lower hyperb} of the appendix, implies that there exist $a_1,a_2,b_1,b_2>0$ such that $(a_1+a_2)/2=\La w\Ra{\ci{I{\ci{N+1}}}}$, $(b_1+b_2)/2=\La\sigma\Ra{\ci{I{\ci{N+1}}}}$ and $a_1b_1^{p-1}=a_2b_2^{p-1}=1$.
\begin{lm}\label{l: stop lower hyperb}
 Let $x,y>0$ be arbitrary, such that $xy^{p-1}\geq1$. Then, there exist $a_1,b_1,a_2,b_2>0$ with $a_2\leq x\leq a_1$ and $b_1\leq y\leq b_2$, such that $a_1b_1^{p-1}=a_2b_2^{p-1}=1$ and $x=\frac{a_1+a_2}{2}$, $y=\frac{b_1+b_2}{2}$.
\end{lm}
Without loss of generality, we may assume that $a_1<a_2$. Consider the bounded weights
\begin{equation*}
w'=\sum_{n=1}^{N+1}2^{n\beta}1{\ci{J_{n}}}+a_1 1{\ci{J{\ci{N+2}}}}+a_2 1{\ci{I{\ci{N+2}}}},\;\;\;\sigma'=\sum_{n=1}^{N+1}2^{-n\beta/(p-1)}1{\ci{J_{n}}}+b_1 1{\ci{J{\ci{N+2}}}}+b_2 1{\ci{I{\ci{N+2}}}}.
\end{equation*}
on $[0,1)$. Notice that $\ssDelta{\ci{I{\ci{N+1}}}}w'=(a_1-a_2)/2<0$. In what follows, we abuse the notation denoting $w',\sigma'$ by $w,\sigma$ respectively.

\subsection{Example for the Haar multiplier}
\label{s:LS Haar mult}

For any choice of signs $\e=(\e\ci{I})\ci{I\in\cD}$ denote by $T_{\e}$ the Haar multiplier on $[0,1)$ corresponding to $\e$, i.e. $T_{\e}$ acts on functions $f\in L^2([0,1))$ via
\begin{equation*}
T_{\e}(f)=\sum_{I\in\cD}\e\ci{I}(\ssDelta\ci{I}f)h\ci{I}.
\end{equation*}
Consider the function $\bff$ on $[0,1)$ given by
\begin{equation*}
\bff=\sum_{n=1}^{\infty}(-1)^{n-1}1{\ci{J_{n}}}.
\end{equation*}
Direct computation gives that for all $I\in\cD$, we have $\ssDelta\ci{I}\bff\neq0$ if and only if $I=I_{n}$ for some $n\in\mathbb{N}$, in which case $\ssDelta\ci{I}\bff=\frac{2(-1)^{n+1}}{3}$. Consider also the function $\bg=-w$ on $[0,1)$. Consider the functions $f=\bff/\sigma,\;g=\bg/w$ on $(0,1)$. Then
\begin{equation*}
\Vert f\Vert ^{p}\ci{L^{p}(\sigma)}=w([0,1))\sim M,\;\;\;\Vert g\Vert ^{p'}\ci{L^{p'}(w)}=w([0,1))\sim M.
\end{equation*}
Moreover, we have
\begin{align*}
&\sup_{\e\in\mathcal{E}}|\La T_{\e}(f\sigma),gw\Ra|=\sup_{\e\in\mathcal{E}}\left|\sum_{I\in\cD}\e\ci{I}|I|(\ssDelta\ci{I}\bff)(\ssDelta\ci{I}\bg)\right|=\sum_{I\in\cD}|I|\cdot|\ssDelta\ci{I}\bff|\cdot|\ssDelta\ci{I}w|\\
&\geq\sum_{n=0}^{N}|I_{n}|\cdot|\ssDelta{\ci{I_{n}}}\bff|\cdot|\ssDelta{\ci{I_{n}}}w|
\sim(1-\beta)^{-1}\sum_{n=0}^{N}2^{n(\beta-1)}\gtrsim (1-\beta)^{-1}M\sim M^2.
\end{align*}
It follows that
\begin{equation*}
\sup_{\e\in\mathcal{E}}\frac{|\La T_{\e}(f\sigma),gw\Ra|}{\Vert f\Vert \ci{L^{p}(\sigma)}\Vert g\Vert \ci{L^{p'}(w)}}\gtrsim_{p}\frac{M^2}{M^{1/p}M^{1/p'}}=M.
\end{equation*}

\subsection{Example for a special type of Haar shift}
\label{s:LS Haar shift}

Let $T$ be the Haar shift on $[0,1)$ acting on functions $f\in L^2([0,1))$ by
\begin{equation*}
Tf=2\sum_{I\in\cD}(\ssDelta\ci{I}f)(h{\ci{I_{+}}}-h{\ci{I_{-}}}).
\end{equation*}
Then, we have
\begin{equation*}
\La Tf,g\Ra=\sum_{I\in\cD}|I|(\ssDelta\ci{I}f)(\ssDelta{\ci{I_{+}}}g-\ssDelta{\ci{I_{-}}}g),
\end{equation*}
for all $f,g\in L^2([0,1))$.

Consider the function $\bff$ on $[0,1)$ given by
\begin{equation*}
\bff=\sum_{n=1}^{\infty}h{\ci{J_{n}}}.
\end{equation*}
Notice that $|\bff|\leq 1$. It is obvious that for all $I\in\cD$, we have $\ssDelta\ci{I}\bff\neq0$ if and only if $I=J_{n}$ for some positive integer $n$, in which case $\ssDelta\ci{I}\bff=1>0$. Consider also the function $\bg=-w$ on $[0,1)$. Consider the functions $f=\bff/\sigma$, $g=\bg/w$ on $[0,1)$. We have
\begin{equation*}
\Vert f\Vert ^{p}\ci{L^{p}(\sigma)}=\left\Vert\frac{1}{\sigma}\right\Vert^{p}\ci{L^{p}(\sigma)}=w([0,1))\sim M,\;\;\;\Vert g\Vert ^{p'}\ci{L^{p'}(w)}=w([0,1))\sim M.
\end{equation*}
Moreover, we have
\begin{align*}
\La f\sigma,T(gw)\Ra=\La \bff,T(\bg)\Ra\geq\sum_{n=0}^{N}|I_{n}|(\ssDelta{\ci{I_{n}}}\bg)(\ssDelta{\ci{J_{n+1}}}\bff)\sim
\sum_{n=0}^{N}(1-\beta)^{-1}2^{n(\beta-1)}\gtrsim M^2.
\end{align*}
It follows that
\begin{equation*}
\frac{\La f\sigma,T(gw)\Ra}{\Vert f\Vert \ci{L^{p}(\sigma)}\Vert g\Vert \ci{L^{p'}(w)}}\gtrsim_{p}\frac{M^2}{M^{1/p}M^{1/p'}}=M.
\end{equation*}

\section{``Small step" constructions}\label{smstep}
\label{s:SM constr}

We describe in this section different variants of ``small step" constructions, that allow us to get from the examples constructed above in Section \ref{s:LS ex}  examples with dyadic smoothness constant arbitrarily close to $1$.

We fix the following notation: for all intervals $J,K$ in $\R$, we denote by $\psi\ci{J,K}$ the unique orientation-preserving affine transformation mapping $J$ onto $K$.

\subsection{A warmup: the ``small step" construction for the Haar multiplier.}
\label{s: SM Haar mult}

Let $p\in(1,\infty)$ and $M>2$. Recall that in Subsection \ref{s:LS Haar mult} we constructed bounded weights $w,\sigma$ on $[0,1)$ with $\sigma=w^{-1/(p-1)}$, such that
\begin{equation*}
M\leq w([0,1))\sigma([0,1))^{p-1},\;\;[w]\ci{A_{p},\cD}\leq 4Me,\;\;w([0,1))\sim M,\;\;\sigma([0,1))\sim_{p}1,
\end{equation*}
and non-zero bounded functions $f\in L^p(w)$, $g\in L^{p'}(\sigma)$ such that 
\begin{align}
\label{main mult 01}
\sup_{\e\in\mathcal{E}}\left| \langle T_{\e} (f\sigma), gw \rangle \right|=\sum_{I\in\cD}|I|\cdot|\ssDelta\ci{I}\bff|\cdot|\ssDelta\ci{I}\bg|\geq c(p) \|f\|\ci{L^p(\sigma)}\|g\|\ci{L^{p'}(w)},
\end{align} 
where $\bff:=f\sigma$ and $\bg:=gw$. Recall that in this example we do not have any control over the dyadic smoothness constants $S_w\ut{d}$ and $S_\sigma\ut{d}$ of the weights $w$ and $\sigma$. 

Based on this example we want to construct weights $\wt{w},\wt{\sigma}$ with $\wt{\sigma}=\wt{w}^{-1/(p-1)}$, and non-zero functions $\wt f \in L^p(\wt w)$, $\wt g\in L^{p'}(\wt\sigma)$ such that \eqref{main mult 01} holds with $\wt f$, $\wt g$, $\wt w$, $\wt \sigma$ in place of $f$, $g$, $w$, $\sigma$ (with another constant $c(p)$); and what is essential, that the dyadic smoothness constants of the new weights are as close to $1$ as we want. 

As we will see, in our construction we will keep track of the averages and martingale differences of the weight $w$, $\sigma$ and of the functions $\bff$ and $\bg$, and their counterparts with tildes.

\subsubsection{A general ``small step" construction}
\label{s: SM Haar mult constr}

We begin by describing a ``small step" construction that does not exploit any intricacies of the particular ``large step" example for Haar multipliers.

Let us first give an informal description. Let $X$ be an 
$\R ^{N}$-valued martingale on $[0,1)$.
Consider the \emph{graph} of the martingale $X$, see Subsection \ref{s:mart-rand walk}. Recall, that the segment of the graph, corresponding to an interval $J\in\cD$ is a straight line segment, connecting points $\La X \Ra\ci{J_-}$ and $\La X \Ra\ci{J_+}$; note that $\La X \Ra\ci{J}$ is the midpoint of this segment.

Take a sufficiently large positive integer $d$. We divide each of the segments of the graph of $X$ in $2d$ parts, so that we get a new graph containing the vertices of the old graph, along with several new vertices, $2\cdot(d-1)$ in number, on each segment, see Figure \ref{segmentdiv}, where new points are marked in red.

\begin{figure}[h]
\centering\includegraphics[scale=0.9]{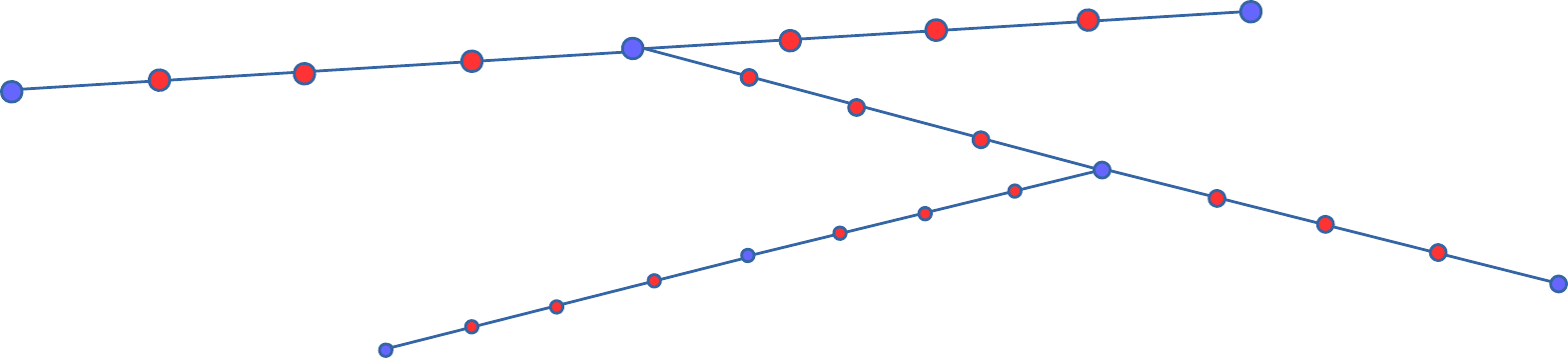}
\caption{Dividing the segments of the graph of $X$}
\label{segmentdiv}
\end{figure}

Let us describe a new random walk on the new graph, which can be thought of as a ``small step" version of the random walk corresponding to the original martingale, producing a new martingale $\wt{X}$.

As in the original random walk, we start from the average $\La X \Ra\ci{[0,1)}$, which, recall, is the midpoint of the segment corresponding to $[0,1)$. From each point $\La X\Ra\ci J$ we perform a ``small step" random walk of order $d$ along the segment corresponding to $J$, moving by $\pm\ssDelta\ci J X/d$ with probability $1/2$. Thus, from each point of the new graph, we move with equal probability $1/2$ to one of the two immediately closest  points (of the new graph) on the corresponding segment (see Figure \ref{segment}). When we reach one of the two endpoints of this segment, we get into a new segment, and we repeat this procedure along the new segment.

\begin{figure}[h]\centering
\includegraphics[scale=1]{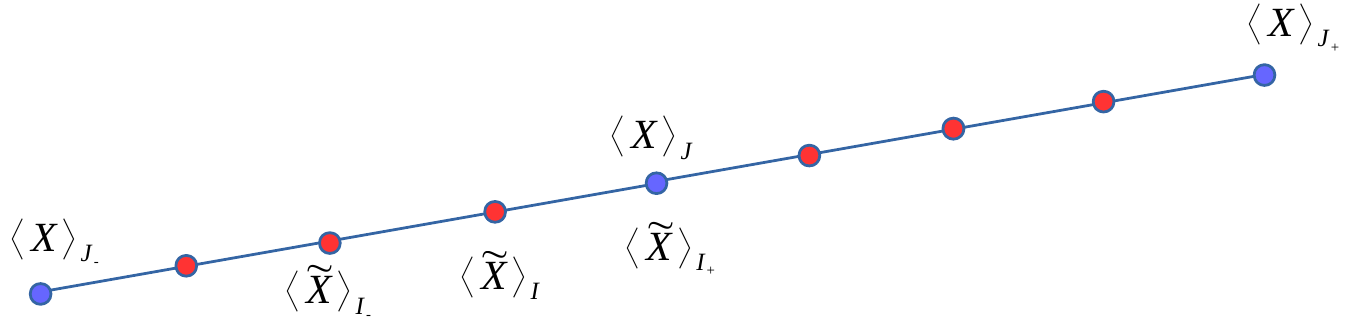}
\caption{Random walk on the new graph}
\label{segment}
\end{figure}

Let us now make all this formal. 
In our case the martingale is always a uniformly bounded one, induced by a function $F\in L^{\infty}([0,1);\R^{N})$; usually in our situation  $N=4$ and $F=(w,\sigma,\bff,\bg)$.  
The construction will be described in terms of the function $F$, so no deep knowledge of probability is required, although the above probabilistic description could help the reader to understand what is going on. 

Given a dyadic subinterval $I$ of $[0,1)$, we define the family $\cS(I)$ of stopping intervals for $I$ as the family of all maximal dyadic subintervals $J$ of $I$ such that
\begin{equation}
\label{stop int}
\bigg|\sum_{\substack{I'\in\cD(I)\\I'\supsetneq J}}h\ci{I'}\bigg|=d,
\end{equation}
and we also define the subset $\cS_{+}(I)$ as the family of all intervals $J$ in $\cS(I)$ for which the sum in \eqref{stop int} is equal to $d$, and similarly we define $\cS_{-}(I)$. Coupled with a translation and rescaling invariance lemma, part (i) of the following lemma implies that the family $\cS(I)$ forms a partition (up to a Borel set of zero measure) of $I$, and part (ii) of it implies that $\bigcup \cS_{+}(I),\bigcup \cS_{+}(I)$ have both measure equal to $|I|/2$.

\begin{lm}\label{l: res rand walks}
Consider the sequence $(r_{n})^{\infty}_{n=1}$ of Rademacher functions on $[0,1)$, i.e.
\begin{equation*}
r_{n}:=\sum_{I\in\cD_{n-1}}h\ci{I},\qquad n=1,2,\ldots.
\end{equation*}
Set $S_0=0$ and $S_{n}=\sum_{k=1}^{n}r_{k}$, for all $n=1,2,\ldots$. Let $a,b\geq0$, not both of them equal to 0. Consider the stopping times $\tau^1,\tau^2,\tau$ given by
\begin{equation*}
\tau^1:=\inf\lbrace n\in\mathbb{N}:\;S_{n}=b\rbrace,\qquad\tau^2:=\inf\lbrace n\in\mathbb{N}:\;S_{n}=-a\rbrace,\qquad\tau:=\min(\tau^1,\tau^2).
\end{equation*}
(i) There holds $\tau^1<\infty$ and $\tau^2<\infty$ a.e. on $[0,1)$.
\newline
(ii) There holds $\mathbb{P}(\tau=\tau^1)=\frac{a}{a+b}$ and $\mathbb{P}(\tau=\tau^2)=\frac{b}{a+b}$.
\end{lm}
The proof of the lemma is given in Subsection \ref{s: rand walks} of the Appedix.

The transformation we describe here acts on functions in $L^{\infty}(I)$ as follows. Let $G\in L^{\infty}(I;\R^{N})$. Then, we define the function $R\ci{I}G:=G\circ\psi\ci{I}$, where $\psi\ci{I}:I\rightarrow I$ is given by
\begin{equation}
\label{SM meas pres trans}
\psi\ci{I}(x)=
\begin{cases}
\psi_{\ci{J,I_{-}}}(x),\text{ if }x\text{ belongs to some } J\in\cS_{-}(I)\\
\psi_{\ci{J,I_{+}}}(x),\text{ if }x\text{ belongs to some }J\in\cS_{+}(I)
\end{cases},
\text{ for almost every }x\in I.
\end{equation}
It is clear that $\psi\ci{I}:I\rightarrow I$ is a measure-preserving transformation.

The ``small step" transform described here is obtained though iterating the above transform in every stopping interval. Namely, we first apply the above construction on the function $F$, along the interval $[0,1)$. We thus obtain a function $R\ci{[0,1)}F\in L^{\infty}([0,1);\R^{N})$. Then, we apply the above transform on the function $(R\ci{[0,1)}F)|\ci{I}$ along the interval $I$, producing new stopping intervals, for all $I\in\cS([0,1))$, and afterwards we repeat this along every stopping interval that will have come up, etc. Therefore, after this process has been completed we will have obtained a new function $\tilde{F}\in L^{\infty}([0,1);\R^{N})$.

It is important to note that in fact this transform (called in what follows ``small step" transform of order $d$) amounts just to a composition of limiting functions with a certain measure-preserving transformation (so in particular, it does not matter whether we apply it to a martingale as a whole or to each of its coordinates separately). Indeed, it is clear that $\tilde{F}=F\circ\Phi$, where $\Phi:[0,1)\rightarrow[0,1)$ is the measure-preserving transformation given at almost every point of $[0,1)$ as the composition of all the measure-preserving transformations $\psi\ci{I}:I\rightarrow I$, where $I$ runs over $[0,1)$ and all stopping intervals containing that point (note that the order of composition respects inclusion of dyadic intervals).

We now specialize to the case $N=4$ and $F=(w,\sigma,\bff,\bg)$. We write then $\tilde{F}=(\wt{w},\wt{\sigma},\wt{\bff},\wt{\bg})$, where tilde denotes just composition with the measure preserving tranformation $\Phi$. In particular, $\tilde{w},\tilde{\sigma}$ are weights on $[0,1)$ with $\wt{w}\wt{\sigma}^{p-1}=1$ a.e. on $[0,1)$.

\subsubsection{Getting the damage}\label{s: SM Haar mult damage}

We first show that the ``small step" transform preserves ``damage" for Haar multipliers.

\begin{lm}\label{l:SM pres}
Let the functions $\bff,\bg,\wt{\bff},\wt{\bg}$ be as above. There holds
\begin{equation*}
\sum_{I\in\cD}|I|\cdot|\ssDelta\ci{I}\wt{\bff}|\cdot|\ssDelta\ci{I}\wt{\bg}|=\sum_{J\in\cD}|J|\cdot|\ssDelta\ci{J}\bff|\cdot|\ssDelta\ci{J}\bg|.
\end{equation*}
\end{lm}
\begin{proof}
First of all, it is immediate by translation and rescaling invariance that
\begin{equation*}
\sum_{J\in\cS(I_0)}\sum_{K\in\cD(J)}|\ssDelta\ci{K}\hat{\bff}|\cdot|\ssDelta\ci{K}\hat{\bg}|\cdot|K|=\sum_{\substack{J\in\cD(I_0)\\J\neq I_0}}|\ssDelta\ci{J}\bff|\cdot|\ssDelta\ci{J}\bg|\cdot|J|,
\end{equation*}
where $I_0=[0,1)$ and $\hat{\bff}:=R\ci{I_0}\bff,~\hat{\bg}:=R\ci{I_0}\bg$. Therefore, since the transform is given by iteration of the same fundamental transform over $[0,1)$ and all stopping intervals, up to translation and rescaling, it suffices only to verify that
\begin{equation}
\label{SM pres intermediate intervals}
\sum_{K\in\cD(I_0)\setminus\left(\bigcup_{J\in\cS(I_0)}\cD(J)\right)}|\ssDelta\ci{K}\hat{\bff}|\cdot|\ssDelta\ci{K}\hat{\bg}|\cdot|K|=|\ssDelta\ci{I_0}\bff|\cdot|\ssDelta\ci{I_0}\bg|\cdot|I_0|.
\end{equation}
It is easy to verify that
\begin{equation}
\label{SM martingale diff intermediate intervals}
\ssDelta\ci{K}\hat{\bff}=\frac{1}{d}\ssDelta\ci{I_{0}}\bff,\;\forall K\in\cD(I_{0})\setminus\left(\bigcup_{J\in\cS(I_{0})}\cD(J)\right),
\end{equation}
and similarly for $\bg$. It follows that
\begin{align*}
\sum_{K\in\cD(I_{0})\setminus\left(\bigcup_{J\in\cS(I_{0})}\cD(J)\right)}|\ssDelta\ci{K}\hat{\bff}|\cdot|\ssDelta\ci{K}\hat{\bg}|\cdot|K|&=
\frac{1}{d^2}\left(\sum_{K\in\cD(I_{0})\setminus\left(\bigcup_{J\in\cS(I_{0})}\cD(J)\right)}|K|\right)|\ssDelta\ci{I_{0}}\bff|\cdot|\ssDelta\ci{I_{0}}\bg|.
\end{align*}
Therefore, it suffices to verify that
\begin{equation}
\label{SM Haar mult intermediate intervals}
\sum_{K\in\cT(I_0)}|K|=\frac{1}{d^2}|I_{0}|.
\end{equation}
where $\cT(I_0):=\cD(I_{0})\setminus\left(\bigcup_{J\in\cS(I_{0})}\cD(J)\right)$. Consider the limiting function $S=\sum_{K\in \cT(I_0)}h\ci{K}$ (the sum should be understood in both the a.e. on $I_{0}$ and $L^{2}(I_{0})$ senses). By the definition \eqref{stop int} of the stopping intervals for $I_{0}$ we obtain $|S|=d$ a.e. on $I_{0}$. In view of orthogonality of Haar functions, it folllows that
\begin{align*}
\sum_{K\in\cT(I_0)}|K|&=\sum_{K\in\cT(I_0)}\Vert h\ci{K}\Vert\ci{L^{2}(I_{0})}^{2}=\Vert S\Vert\ci{L^{2}(I_{0})}^{2}=d^2|I_{0}|,
\end{align*}
concluding the proof.
\end{proof}

\begin{rem}\label{r:SM max func} Consider the dyadic Hardy-Littlewood maximal functions $M\bff,~M\wt{\bff}$ of $\bff,\wt{\bff}$ respectively. We claim that $M\wt{\bff}\geq(M\bff)\circ\Phi$ a.e. on $[0,1)$.

Indeed, note first that $\wt{|f|}=|f|\circ\Phi=|f\circ\Phi|=|\wt{f}|$, so $|\wt{f}|$ is obtained from $|f|$ through the same ``small step" transform as $\wt{f}$ is obtained through $f$. It suffices now to note that for all $I\in\cD$ and for all $G\in L^{\infty}(I)$ we have 
\begin{equation*}
\langle R\ci{I}G\rangle\ci{J}=\langle G\rangle\ci{I_{\pm}},\;\forall J\in\cS_{\pm}(I).
\end{equation*}
\end{rem}

\subsubsection{Supressing dyadic smoothness constants}\label{s:SM Haar mult dyad smooth}

We next show that the ``small step" construction as given above provides very tight control over dyadic smoothness constants, provided $d$ is large enough.

\begin{lm}\label{s:SM doubl}
Let the weights $w,\wt{w}$ be as above. Given $\e>0$, assume that $d>(S\ut{d}_{w}-1)/\e$. Then, the dyadic smoothness constant $S\ut{d}_{\wt{w}}$ of the weight $\wt{w}$ is less than $1+\e$.
\end{lm}
\begin{proof}
First of all, it is immediate by rescaling and translation invariance that for all $I\in\cD$ and for all weights $\rho$ on $I$, the dyadic smoothness constant of the weight $(R\ci{I}\rho)|\ci{J}$ is not larger than $S\ut{d}_{\rho}$, for all $J\in\cS(I)$. Therefore, since the transform is given by iteration of the same fundamental transform over $[0,1)$ and all stopping intervals, up to tranaslation and rescaling, it suffices only to verify that
\begin{equation}
\max\left(\frac{\La\hat{w}\Ra{\ci{K_{-}}}}{\La\hat{w}\Ra{\ci{K_{+}}}},\frac{\La\hat{w}\Ra{\ci{K_{+}}}}{\La\hat{w}\Ra{\ci{K_{-}}}}\right)\leq1+\e,\qquad\forall K\in\cD(I_0)\setminus\left(\bigcup_{J\in\cS(I_0)}\cD(J)\right),
\end{equation}
where $I_0=[0,1)$ and $\hat{w}:=R\ci{[0,1)}w$, provided that $d>(S\ut{d}_{w}-1)/\e$.

Let $K\in\cD(I_0)\setminus\left(\bigcup_{J\in\cS(I_0)}\cD(J)\right)$ be arbitrary. We have $\ssDelta\ci{K}\hat{w}=(1/d)\ssDelta\ci{I_0}w$. Moreover, $K_{+}$ can be written as a union of stopping intervals (up to a set of zero measure), therefore $\langle \hat{w}\rangle\ci{K_{+}}=a\langle w\rangle\ci{(I_0)_{-}}+(1-a)\langle w\rangle\ci{(I_0)_{+}}$, for some $a\in[0,1]$. It follows that
\begin{align*}
\left|\frac{\La\hat{w}\Ra{\ci{K_{-}}}}{\La\hat{w}\Ra{\ci{K_{+}}}}-1\right|=\frac{|\La\hat{w}\Ra{\ci{K_{-}}}-\La\hat{w}\Ra{\ci{K_{+}}}|}{\La\hat{w}\Ra{\ci{K_{+}}}}\leq
\frac{1}{d}\cdot\frac{|\La w\Ra{\ci{(I_0)_{+}}}-\La w\Ra{\ci{(I_0)_{-}}}|}{\min(\La w\Ra{\ci{(I_0)_{+}}},\La w\Ra{\ci{(I_0)_{-}}})}.
\end{align*}
 Without loss of generality, we may assume that $\La w\Ra{\ci{(I_0)_{-}}}\leq\La w\Ra{\ci{(I_0)_{+}}}$ (the other case is symmetric). Then, we have
\begin{align*}
\frac{1}{d}\cdot\frac{|\La w\Ra{\ci{(I_0)_{+}}}-\La w\Ra{\ci{(I_0)_{-}}}|}{\min(\La w\Ra{\ci{(I_0)_{+}}},\La w\Ra{\ci{(I_0)_{-}}})}
=\frac{1}{d}\cdot\frac{\La w\Ra{\ci{(I_0)_{+}}}-\La w\Ra{\ci{(I_0)_{-}}}}{\La w\Ra{\ci{(I_0)_{-}}}}
\leq \frac{1}{d}(S\ut{d}_{w}-1)<\e.
\end{align*}
Similarly $\La\hat{w}\Ra{\ci{K_{+}}}/\La\hat{w}\Ra{\ci{K_{-}}}<1+\e$, concluding the proof.
\end{proof}

\subsubsection{Respecting dyadic Muckenhoupt characteristics}\label{s: SM Haar mult dyad Muck}

We next show that the ``small step" construction does not ruin dyadic Muckenhoupt constants, up to constants depending only on $p$. Namely, we claim that $[\wt{w},\wt{\sigma}]\ci{A_{p},\cD}\leq 2^{p}[w,\sigma]\ci{A_{p},\cD}$. To see that, note first that it immediate from translation and rescaling invariance that for all $J\in\cS(I_0)$ we have $[\hat{w}|\ci{J},\hat{\sigma}|\ci{J}]\ci{A_{p},\cD(J)}\leq[w,\sigma]\ci{A_{p},\cD(I_0)}$, where $I_0:=[0,1)$ and $\hat{w}:=R\ci{[0,1)}w$, $\hat{\sigma}:=R\ci{[0,1)}\sigma$. Therefore, since the transform is given by iteration of the same fundamental transform over $[0,1)$ and all stopping intervals, up to translation and rescaling, it suffices only to verify that
\begin{equation}
\La\hat{w}\Ra\ci{K}\La\hat{\sigma}\Ra\ci{K}^{p-1}\leq 2^{p}[w,\sigma]\ci{A_{p},\cD(I_0)},\qquad\forall K\in\cD(I_0)\setminus\left(\bigcup_{J\in\cS(I_0)}\cD(J)\right).
\end{equation}
Let $K\in\cD(I_0)\setminus\left(\bigcup_{J\in\cS(I_0)}\cD(J)\right)$ be arbitrary. Since $K$ can be written as a union of stopping intervals (up to a set of zero measure), we have $\langle \hat{w}\rangle\ci{K}=a\langle w\rangle\ci{(I_0)_{-}}+(1-a)\langle w\rangle\ci{(I_0)_{+}}$ and $\langle \hat{\sigma}\rangle\ci{K}=a\langle \sigma\rangle\ci{(I_0)_{-}}+(1-a)\langle \sigma\rangle\ci{(I_0)_{+}}$, for some $a\in[0,1]$. Then, the following lemma, whose proof is given in Subsection \ref{s: upper hyperb} in the Appendix, implies immediately the required result.

\begin{lm}
\label{l: upper hyperb}
Let $x_1,y_1,x_2,y_2>0$ and $A>0$, such that
\begin{equation*}
x_1y_1^{p-1},~\left(\frac{x_1+x_2}{2}\right)\left(\frac{y_1+y_2}{2}\right)^{p-1},~x_2y_2^{p-1}\leq A.
\end{equation*}
Then, there holds
\begin{equation*}
(x_1+a(x_2-x_1))(y_1+a(y_2-y_1))^{p-1}\leq 2^{p}A,\;\forall a\in[0,1].
\end{equation*}
\end{lm}

\subsubsection{Respecting weighted norms}\label{s: SM Haar mult weighted norms}

Finally, we show that weighted norms do not get larger. Consider the function $\wt{g}=\wt{\bg}/\wt{w}$. Obviously $\wt{g}=g\circ\Phi$. It follows that
\begin{equation}
\label{SM Haar mult weighted norms}
\Vert\wt{g}\Vert^{p'}\ci{L^{p'}(\wt{w})}=\int_{[0,1)}|g(\Phi(x))|^{p'}w(\Phi(x))dx=\int_{[0,1)}|g(x)|^{p'}w(x)dx=\Vert g\Vert^{p'}\ci{L^{p}(w)}.
\end{equation}
Similarly $\Vert\wt{f}\Vert^{p}\ci{L^{p}(\wt{\sigma})}=\Vert f\Vert^{p}\ci{L^{p}(\sigma)}$, where $\wt{f}=\wt{\bff}/\wt{\sigma}$.

\subsection{The ``small step" construction for Haar shifts}
\label{s:SM Haar shift}

In this section, we describe one variant of the ``small step" construction of the previous subsection which exploits the special structure of the martingales in the example of Subsection \ref{s:LS Haar shift}

Let $p\in(1,\infty)$ and $M>2$. Recall the Haar shift $T$ on $[0,1)$ considered in Subsection \ref{s:LS Haar shift}:
\begin{equation*}
Tf:=2\sum_{I\in\cD}(\ssDelta\ci{I}f)(h{\ci{I_{+}}}-h{\ci{I_{-}}}).
\end{equation*}
Then, we have
\begin{equation*}
\La Tf,g\Ra=\sum_{I\in\cD}|I|(\ssDelta\ci{I}f)(\ssDelta{\ci{I_{+}}}g-\ssDelta{\ci{I_{-}}}g),\;\forall f,g\in L^2([0,1)).
\end{equation*}

Let us first recall the ``large step" example of Subsection \ref{s:LS Haar shift}. Set $I_{0}=[0,1)$ and $I_{n}=\left[0,\frac{1}{2^{n}}\right),\; J_{n}=\left[\frac{1}{2^{n}},\frac{1}{2^{n-1}}\right)$, for all $n=1,2,\ldots$. Recall that in Subsection \ref{s:LS Haar shift} we showed that there exist bounded weights $w,\sigma$ on $[0,1)$ with $\sigma=w^{-1/(p-1)}$,
\begin{equation*}
M\leq w([0,1))\sigma([0,1))^{p-1},\qquad[w]\ci{A_{p},\cD}\leq 4Me,\qquad w([0,1))\sim M,\qquad\sigma([0,1))\sim_{p}1,
\end{equation*}
with the additional properties $\ssDelta\ci{I}w=\ssDelta\ci{I}\sigma=0$, for all $I\in\cD\setminus\lbrace I_0,I_1,I_2,\ldots\rbrace$, $\ssDelta\ci{I_{l}}w\leq0$, for all $l=0,1,2,\ldots$, and nonzero bounded functions $f\in L^{p}(\sigma)$, $g\in L^{p'}(w)$ with
\begin{equation}
\label{Haar shift initial estimate}
\La f\sigma,T(gw)\Ra\gtrsim_{p} M\Vert f\Vert \ci{L^{p}(\sigma)}\Vert g\Vert \ci{L^{p'}(w)}.
\end{equation}
We recall that $g=-\1\ci{[0,1)}$, so $\bg:=gw=-w$. Moreover, for the function $\bff:=f\sigma$ on $[0,1)$ we have $\langle f\rangle\ci{[0,1)}=0$, and for all $I\in\cD$ we have $\ssDelta\ci{I}\bff\neq0$ if and only if $I=J_{n}$ for some positive integer $n$, in which case $\ssDelta\ci{I}\bff>0$.

Based on this example we want to construct weights $\wt{w},\wt{\sigma}$ with $\wt{\sigma}=\wt{w}^{-1/(p-1)}$, and non-zero functions $\wt f \in L^p(\wt w)$, $\wt g\in L^{p'}(\wt\sigma)$ such that \eqref{Haar shift initial estimate} holds with $\wt f$, $\wt g$, $\wt w$, $\wt \sigma$ in place of $f$, $g$, $w$, $\sigma$. Again, it will be essential that the dyadic smoothness constants of the new weights are as close to $1$ as we want. This new example will be used to obtain a ``small step'' example for the Hilbert transform in Subsection \ref{s:smooth Hilb}. For reasons to become apparent there, we will want the martingale differences of the function $\wt{\bg}:=\wt{g}\wt{w}$ over dyadic intervals of \emph{odd} generation to vanish. Thus, we cannot just mimic naively the ``small step'' construction of the previous subsection.

\subsubsection{``Small step" random walk on a triangle}
\label{s: Rand triang}

Consider the $\R^{4}$-valued martingale $X$ induced by the function $F=(w,\sigma,\bg,\bff)$. Then, we have $\ssDelta\ci{I}X=0$ for all $I\in\cD$ different from $I_0,I_1,I_2,\ldots$ and $J_1,J_2,J_3,\ldots$. Notice that the vectors $\ssDelta\ci{I_n}X,~\ssDelta\ci{(I_n)_{+}}X$ are either linearly independent (in fact orthogonal to each other), or one of them is equal to 0, for all $n=0,1,2,\ldots$. Therefore, the random walk corresponding to the four-dimensional martingale $X$ takes place on the ``union" of a family of isosceles triangles in $\R^{4}$ (maybe degenerate) as in Figure \ref{trianglegraph}, corresponding to the intervals $I_{0},I_{1},I_{2},\ldots$ respectively.

\begin{figure}[h]\centering
\includegraphics[scale=0.85]{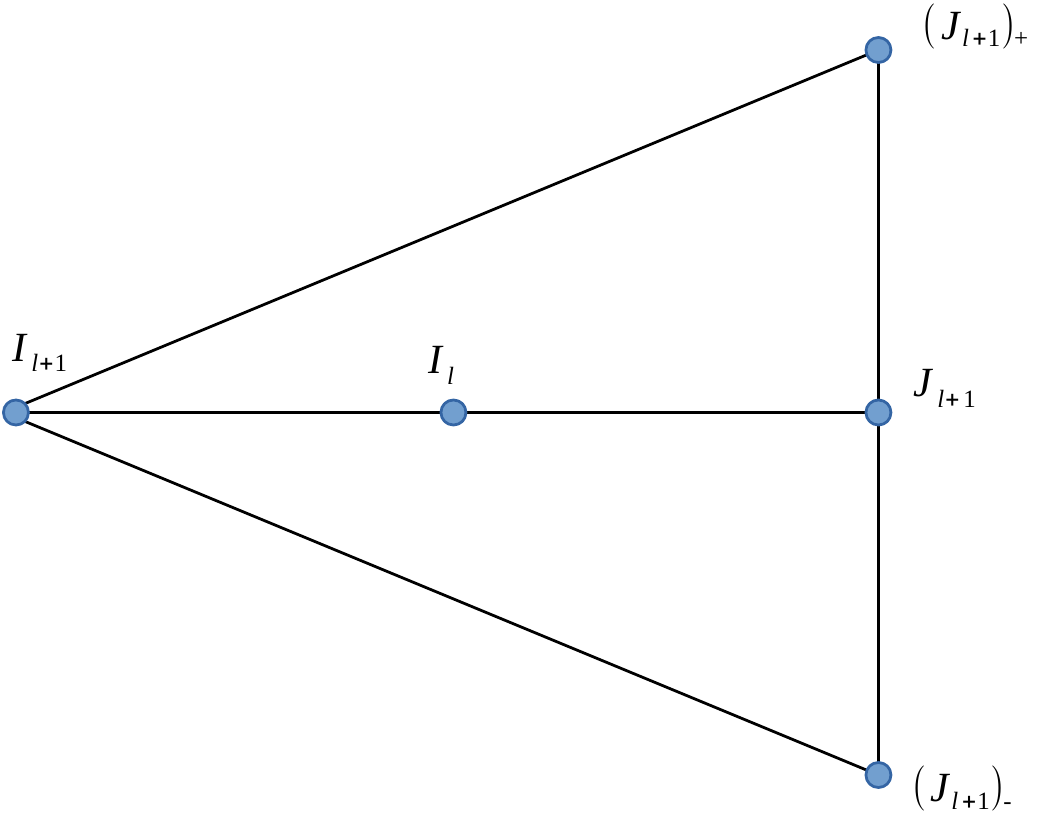}
\caption{The triangle corresponding to interval $I_l$}
\label{trianglegraph}
\end{figure}

Starting with the interval $I_0=[0,1)$, we replace the constant function $X_0\equiv\langle X\rangle\ci{I_0}$ with the function $X^1=\langle X\rangle\ci{I_0}+(\ssDelta\ci{I_0}X)h\ci{I_0}+(\ssDelta\ci{(I_0)_{+}}X)h\ci{(I_0)_{+}}$. This function is constant on $(I_0)_{-}=I_{1}$ and the children $(I_0)_{+-}=(J_0)_{-},~(I_0)_{++}=(J_0)_{+}$ of $(I_0)_{+}=J_0$. In each of the children of $(I_0)_{+}$, we just stop, i.e. the function $F$ is constant there, while in the interval $(I_0)_{-}=I_1$ we repeat this procedure, starting with the constant function $X^1|\ci{I_1}$, and using the martingale differences of $X$ over $I_1,~(I_1)_{+}$ this time, and then we repeat the same pattern in the interval $(I_1)_{-}=I_2$, etc. So the random walk corresponding to $X$ consists of rescaled and translated copies of the same pattern, independent from each other. Our main object now is to replace the term $(\ssDelta\ci{I_n}X)h\ci{I_n}+(\ssDelta\ci{(I_n)_{+}}X)h\ci{(I_n)_{+}}$ by a linear combination of Haar functions with ``smaller" coefficients, reflecting a ``small step" random walk, for all $n=0,1,2,\ldots$.

Choose a sufficiently large positive integer $d>100$. Condider the model triangle on $\R^2$ with vertices $-e_1,~e_1+e_2$ and $e_1-e_2$, where $e_1=(1,0)$ and $e_2=(0,1)$. Given a dyadic subinterval $I$ of $[0,1)$ of even generation, we can describe a random walk in $I$ as follows. Starting with the constant function taking value $c\ci{[0,1)}=0\in\R^2$, we replace it with the function $(1/d)h\ci{I}e_1+(1/2d)h\ci{I_{+}}e_2$. Notice that the latter function is constant on grandchildren on $I$. We then repeat the same pattern in the grandchildren of $I$, and we repeat again this pattern in the grandchildren of the latter intervals, etc. The pattern continues until for some interval $J$ which will have arisen as a grandchild during this process, the current constant value $c\ci{J}$ on $J$ is located on the boundary of the triangle. We will say that such intervals $J$ are preliminary stopping intervals. In particular, the preliminary stopping intervals are of even generation. Denote the family of all preliminary stopping intervals by $\wt{\cS}(I)$.

If $J$ is a preliminary stopping interval such that the constant value $c\ci{J}$ on $J$ is located on a side of the model triangle other than its base (that is the vertical side of the triangle), then we replace the constant function $c\ci{J}$ on $J$ with the function $c\ci{J}+(1/d)h\ci{J}e_1\pm(1/2d)h\ci{J}e_2$, where $\pm=+$, respectively $\pm=-$, if $c\ci{J}$ is located on the upper, respectively lower, side of the model triangle. Then we repeat this in the grandchildren of $J$, and then we repeat the pattern in the grandchildren of the latter intervals, etc. The pattern continues until for some interval $K$ which will have arisen as a grandchild during this process, the current constant value $c\ci{K}$ on $K$ is located on one of the three vertices of the triangle. We will say that such intervals $K$ are stopping intervals. In particular, these stopping intervals are of even generation.

If $J$ is a preliminary stopping interval such that the constant value $c\ci{J}$ on $J$ is located on the base of the triangle, then we replace the constant function $c\ci{J}$ on $J$ with the function $c\ci{J}+(1/2d)h\ci{J}e_2$. Then we repeat this in the grandchildren of $J$, and then we repeat the pattern in the grandchildren of the latter intervals, etc. The pattern continues until for some interval $K$ which will have arisen as a grandchild during this process, the current constant value $c\ci{K}$ on $K$ is located on one of the two vertices of the base. We will also say that such intervals $K$ are stopping intervals. In particular, these stopping intervals are of even generation.

We will denote the family of all stopping intervals by $\cS(I)$. We will also denote the family of all stopping intervals $J$ such that $c\ci{J}$ is located on the vertex (i.e. $-e_1$) opposite to the base of the model triangle, respectively on the upper vertex (i.e. $e_1+e_2$) of the base, respectively on the lower vertex (i.e. $e_1-e_2$) of the base, by $\cS_{-}(I)$, respectively by $\cS_{++}(I)$, respectively by $\cS_{+-}(I)$. We also set $\cS_{+}(I)=\cS_{++}(I)\bigcup\cS_{+-}(I)$. We will call the elements of $\cS_{-}(I)$, respectively $\cS_{+}(I)$, left, respectively right, stopping intervals.

Given now a function $G\in L^{\infty}(I;\R^{4})$, the variant of the ``small step" trasform we are describing here maps it to the function $R\ci{I}G:=G\circ\psi\ci{I}$, where (compare with \eqref{SM meas pres trans})
\begin{equation}
\label{SM triangle meas pres trans}
\psi\ci{I}(x)=
\begin{cases}
\psi\ci{J,I_{-}}(x),\text{ if }x\in J\text{ for some }J\in\cS_{-}(I)\\
\psi\ci{J,I_{+\pm}}(x),\text{ if }x\in J\text{ for some }J\in\cS_{+\pm}(I)
\end{cases}
,\forall x\in I.
\end{equation}
The symmetries of the walk imply that $\psi\ci{I}:I\rightarrow I$ is measure preserving.

The variant of the ``small step" transform described here is obtained through iterating the above fundamental transform as follows. We first apply the above construction on the function $F$, along the interval $[0,1)$. We thus obtain a function $R\ci{[0,1)}F\in L^{\infty}[0,1);\R^{4})$. In each interval in $\cS_{+}(I)$, we just stop (recall that the original function $F$ is constant on the children on $I_0$), while we apply the above transform on the function $(R\ci{[0,1)}F)|\ci{I}$ along the interval $I$, for all $I\in\cS_{-}([0,1))$, and then we stop on every right stopping interval that will have come up, while we repeat the same transform along every left stopping interval that will have come up, etc. Therefore, after this process has been completed we will have obtained a new function $\tilde{F}\in L^{\infty}([0,1);\R^4)$.

Recall that the original function $F$ is constant on the children on $(I_n)_{+}$, for all $n=0,1,2,\ldots$. Note also that $I_{n+1}=(I_n)_{-}$, for all $n=0,1,2,\ldots$. It follows that $\tilde{F}=F\circ\Phi$, where $\Phi:[0,1)\rightarrow[0,1)$ is the measure-preserving transformation given at almost every point of $[0,1)$ as the composition of all the measure-preserving transformations $\psi\ci{I}:I\rightarrow I$, where $I$ runs over $[0,1)$ and all left stopping intervals containing that point (note that the order of composition respects inclusion of dyadic intervals). We write $\wt{F}=(\wt{w},\wt{\sigma},\wt{\bg},\wt{\bff})$, where tilde denotes just composition with the measure preserving transformation $\Phi$. 

Notice that $I_0$ is an interval of even generation, so its grandchilren are also of even generation, etc. In is then clear that the functions $\wt{w},\wt{\sigma}$ are in fact obtained from the functions $w,\sigma$ respectively thought ``small step" transform of order $d$ as in the previous section, but ``skipping" intervals of odd generations (i.e. omitting the Haar functions corresponding to them). This means that dyadic intervals $I$ of odd generation ``do not split", i.e. $\La \wt{w}\Ra\ci{I}=\La \wt{w}\Ra\ci{I_{-}}=\La \wt{w}\Ra\ci{I_{+}}$, and similarly for $\wt{\sigma}$. It is clear that this will be only a minor modification of the construction described in Subsection \ref{s: SM Haar mult}. In particular, $\wt{w},~\wt{\sigma}$ are weights on $[0,1)$ with $\wt{w}\wt{\sigma}^{p-1}=1$ a.e. on $[0,1)$, and for large enough $d$ the weights $\wt{w},~\wt{\sigma}$ will possess the required dyadic Muckenhoupt characteristic and dyadic smoothness properties.

\subsubsection{Getting the damage}\label{s: SM Haar shift damage}

We show that the ``small step" transform we just described preserves damage for the Haar shift $T$, i.e. that $\La\wt{\bff},T(\wt{\bg})\Ra\gtrsim\La \bff,T(\bg)\Ra$.

\begin{lm} \label{l: SM triangle Haar shift damage}
Let the functions $\bff,\bg,\wt{\bff},\wt{\bg}$ be as above. There holds
\begin{equation*}
\sum_{I\in\cD}(\ssDelta\ci{I}\wt{\bg})(\ssDelta\ci{I_{+}}\wt{\bff}-\ssDelta\ci{I_{-}}\wt{\bff})|I|\gtrsim\sum_{J\in\cD}(\ssDelta\ci{J}\bg)(\ssDelta\ci{J_{+}}\bff-\ssDelta\ci{I_{-}}\bff)|J|.
\end{equation*}
\end{lm}
\begin{proof}
First of all, it is immediate by translation and rescaling invariance that
\begin{equation*}
\sum_{J\in\cS(I_0)}\sum_{K\in\cD(J)}(\ssDelta\ci{K}\hat{\bg})(\ssDelta\ci{K_{+}}\hat{\bff}-\ssDelta\ci{K_{-}}\hat{\bff})|K|=\sum_{\substack{J\in\cD(I_0)\\J\neq I_0,(I_0)_{+}}}(\ssDelta\ci{J}\bg)(\ssDelta\ci{J_{+}}\bff-\ssDelta\ci{J_{-}}\bff)|J|,
\end{equation*}
where $I_0=[0,1)$ and $\hat{\bff}:=R\ci{[0,1)}\bff$, $\hat{\bg}:=R\ci{[0,1)}\bg$. Note also that $\ssDelta\ci{(I_0)_{+}}\bg=0$. Therefore, as in Lemma \ref{l:SM pres}, it suffices only to prove the following analog of \eqref{SM pres intermediate intervals}:
\begin{equation}
\label{SM triangle damage intermediate intervals}
\sum_{I\in\cT(I_0)}(\ssDelta\ci{I}\hat{\bg})(\ssDelta\ci{I_{+}}\hat{\bff}-\ssDelta\ci{I_{-}}\hat{\bff})|I|\gtrsim(\ssDelta\ci{I_{0}}\bg)(\ssDelta\ci{(I_0)_{+}}\bff-\ssDelta\ci{(I_0)_{-}}\bff)|I_0|,
\end{equation}
where $\cT(I_0):=\cD(I_0)\setminus\left(\bigcup_{J\in\cS(I_0)}\cD(J)\right)$. Notice that there is an implied absolute constant in the inequality in \eqref{SM triangle damage intermediate intervals}, unlike \eqref{SM pres intermediate intervals}, where there was just equality. This is no problem (for instance, there will not be accummulation of this constant), since the transform is given by iteration of the same fundamental transform over $[0,1)$ and all left stopping intervals, up to translation and rescaling (essentially, the iterative nature of the transform and translation and rescaling invariance imply that one needs only to verify the damage inside each triangle separately, and these verifications are independent from each other).

First of all, notice that only intervals in $\wt{\cT}(I_0):=\cD(I_0)\setminus\left(\bigcup_{J\in\wt{\cS}(I_0)}\cD(J)\right)$ that are of even generation may contribute to the sum in \eqref{SM triangle damage intermediate intervals}, and for each such interval $I$ we have $\ssDelta\ci{I}\hat{\bg}=(1/d)\ssDelta\ci{I_0}\bg$, $\ssDelta\ci{I_{+}}\hat{\bff}=(1/2d)\ssDelta\ci{(I_0)_{+}}\bff$ and $\ssDelta\ci{I_{-}}\hat{\bff}=0=\ssDelta\ci{(I_0)_{-}}\bff$. Therefore, it suffices to check that
\begin{equation}
\label{SM triangle pres check even}
\sum_{I\in\wt{\cT}\ti{e}(I_0)}|I|\gtrsim d^2,
\end{equation}
where $\wt{\cT}\ti{e}(I_0)$ is the family of all intervals in $\wt{\cT}(I_0)$ that are of even generation. Orthogonality of Haar functions yields
\begin{align*}
\sum_{I\in\wt{\cT}\ti{e}(I_0)}|I|\sim\sum_{I\in\wt{\cT}\ti{e}(I_0)}\left\Vert h\ci{I}e_1+\frac{1}{2}h\ci{I_{+}}e_2\right\Vert\ci{L^{2}(I_0;\R^2)}^2=\Vert S\Vert\ci{L^{2}(I_0;\R^2)}^2,
\end{align*}
where we are considering the limiting function $S:=\sum_{I\in\wt{\cT}\ti{e}(I_0)}\left(h\ci{I}e_1+\frac{1}{2}h\ci{I_{+}}e_2\right)$ (the sum should be understood in both the pointwise a.e. on $I_0$ and $L^2(I_0;\R^2)$ senses). Rescaling the canonical triangle by $d$ we see that this limiting function is taking values on the boundary of the triangle in $\R^2$ with vertices $(-d,0),~(0,d),~(0,-d)$. Since the distance of the origin from the boundary of this triangle is $d/\sqrt{5}$, we obtain $|S|\geq d/\sqrt{5}$, therefore $\Vert S\Vert\ci{L^{2}(I_0;\R^2)}^2\gtrsim d^2$, concluding the proof.
\end{proof}

\subsubsection{Respecting weighted norms}\label{s: SM Haar shift weighted norms}

Identically to \eqref{SM Haar mult weighted norms} we have $\Vert\wt{g}\Vert^{p'}\ci{L^{p'}(\wt{w})}=\Vert g\Vert^{p'}\ci{L^{p'}(w)}$ and $\Vert \wt{f}\Vert ^{p}\ci{L^{p}(\wt{\sigma})}=\Vert f\Vert ^{p}\ci{L^{p}(\sigma)}$, where $\wt{g}:=\wt{\bg}/\wt{w}$ and $\wt{f}:=\wt{\bff}/\wt{\sigma}$.

\section{Iterated remodeling}\label{s:remodel}

In this section we describe the method of iterated remodeling, which is a variant of the powerful method of remodeling, introduced by F. Nazarov in \cite{Nazarov}.

Throughout this section, for all intervals $I,J$ we denote by $\psi\ci{I,J}$ the unique orientation-preserving affine transformation mapping $I$ onto $J$. 

\subsection{Periodisations} \label{s:period}

Let $f\in L^{\infty}([0,1);\R^{n})$. For a given interval $I$ and for a given positive integer $N$, we define the periodisation $\Pi\ci{I}^{N}f$ of $f$ of frequency $N$ over $I$ as the unique periodic function over $I$ of period $\frac{|I|}{2^{N}}$ consisting of $2^{N}$ repeated copies of the function $f$, i.e. $\Pi\ci{I}^{N}f=f\circ\psi\ci{I}^{N}$, where $\psi\ci{I}^{N}(x)=\psi\ci{J,I}(x)$ for all $x\in J$, for all $J\in\ch^{N}(I)$, see Figure \ref{periodization} (here we abuse the terminology regarding the use of the term ``frequency'').

\begin{figure}[h]
\centering
\includegraphics[scale=0.3]{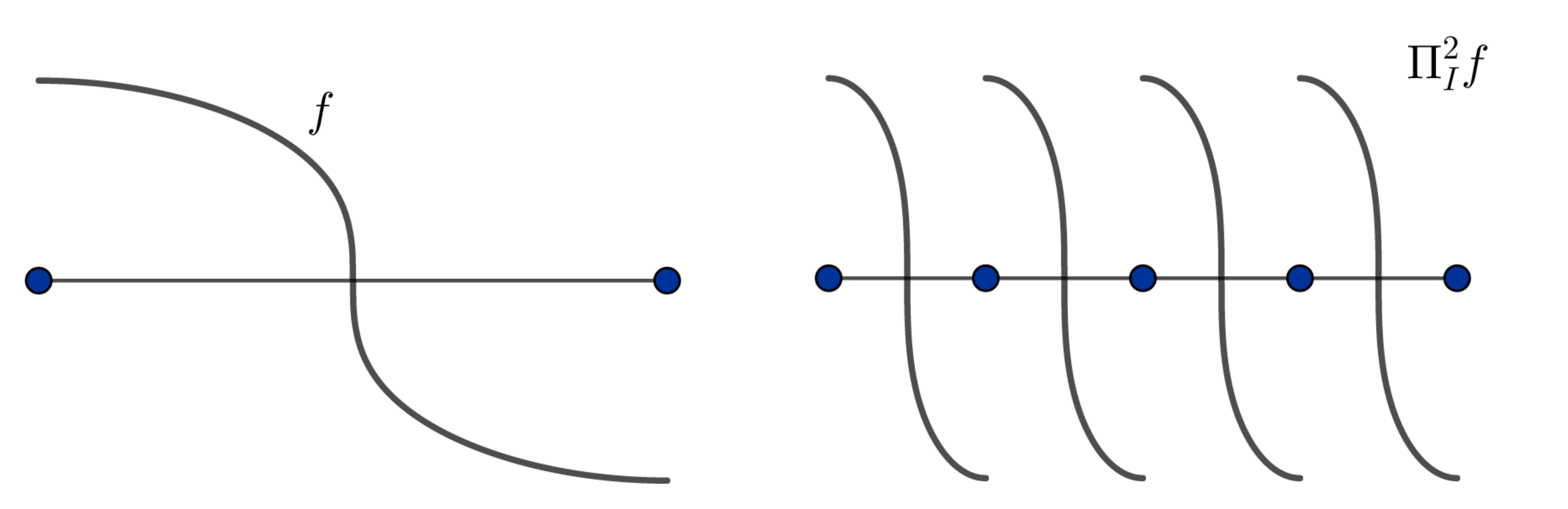}
\caption{Periodization $\Pi^{2}\ci{I}f$ of function $f$}
\label{periodization}
\end{figure}

Note that $\psi\ci{I}^{N}:I\rightarrow I$ is measure preserving. We define the family $\cE\ci{N}(I)$ of exceptional stopping intervals for $I$ of order $N$ as the family of all intervals in $\ch^{N}(I)$ that touch the boundary of $I$ (so $\cE\ci{N}(I)$ has exactly two elements), and the family $\cR\ci{N}(I)$ of regular stopping intervals for $I$ of order $N$ as the family of all intervals in $\ch^{N}(I)$ that do not touch the boundary of $I$.

\subsection{From Bourgain's localizing trick to Nazarov remodeling and iterated remodeling}

F. Nazarov's method of remodeling \cite{Nazarov} had been inspired by a new technique for localizing the action of operators introduced by J. Bourgain in \cite{Bourgain}. There, Bourgain showed that UMD property for a Banach space $X$ follows from the boundedness of the Hilbert transform over $L^{p}(\mathbb{T};X)$ for all $1<p<\infty$, where $\mathbb{T}$ denotes the unit circle. Bourgain related estimates for the $L^{p}$ norm of the Hilbert transform, a non-localized operator, to estimates for the $L^{p}$ norm of the square function, a well-localized operator, through the trick of iteratively replacing portions of functions with their periodisations. 

Bourgain's \cite{Bourgain} basic idea was the following. Given a function $f\in L^2(\mathbb{T})$, one can consider its Fourier series
\begin{equation}
\label{Fourier series}
f(0)+\sum_{m\in\mathbb{Z}\setminus\lbrace0\rbrace}\hat{f}(m)z^{m}.
\end{equation}
One way to make the action of a bounded in $L^2([0,1))$ operator on $f$ localized is to create very ``large gaps'' in expansion \eqref{Fourier series}, by considering the function $\wt{f}$ with Fourier series
\begin{equation*}
\wt{f}=f(0)+\sum_{m\in\mathbb{Z}\setminus\lbrace0\rbrace}\hat{f}(m)z^{mN_{m}},
\end{equation*}
where the $N_{m}$'s are large enough positive integers chosen through an inductive procedure. Here one exploits the fact that $(z^{N})_{N=0}^{\infty}$ converges weakly in (say)  $L^2([0,1))$ to 0 as $n\rightarrow\infty$. Note that the ``transformed'' Fourier series is still a Fourier series.

Given now an $X$-valued function $f$ (say bounded) on $[0,1))$ (we freely identify $\mathbb{T}$ with $[0,1)$), one can consider its martingale difference decomposition in $L^2([0,1);X)$:
\begin{equation}\label{martingale decomposition}
f=\La f\Ra\ci{[0,1)}+\sum_{I\in\cD}\Delta\ci{I}f.
\end{equation}
In general, when the Hilbert transform acts on $f$ its action will not be localized, i.e. there will be interactions between martingale differences over different intervals. One could then think of attempting to somehow introduce very large ``gaps'' in \eqref{martingale decomposition}, inspired from the respective situation in Fourier series. This is not directly possible, and instead on has to notice that the idea in the setting of Fourier series was to replace each $z^{m}$ with $(z^{N_{m}})^{m}$, which is a just a periodisation of $z^{m}$. Then one notices that the periodisations of a given martingale difference converge weakly to 0 in (say) $L^2$ as the frequency increases (see Lemma \ref{l:osc}). Therefore, one can attempt to replace each martingale difference in \eqref{martingale decomposition} with a periodisation of it. The frequencies would be chosen large enough through an inductive procedure. Note that the ``transformed'' martingale difference decomposition should be still a martingale difference decomposition, thus the periodised martingale differences should still somehow respect the hierarchy of dyadic intervals.

Bourgain \cite{Bourgain} not only came up with the above intuition, but also found a sleek way to make it precise. Namely, given an $X$-valued function $f$ (say bounded) on the unit interval $I_0:=[0,1)$, one begins by choosing a frequency $N(I_0)$ and replacing $f$ with its periodisation $\wt{f}^1:=\Pi\ci{I_0}^{N(I_0)}f$. Consider the collection $\cS^1:=\ch^{N(I_0)}(I_0)$. Note that
\begin{equation}
\label{periodised martingale differences}
\mathbb{E}_{\ch(\cS^1)}[\wt{f}^1]-\La f\Ra\ci{I_0}=\Pi\ci{I_0}^{N(I_{0})}(\Delta\ci{I_0}f).
\end{equation}

Then, for all $I\in\ch(\cS^1)$, one can replace the function $\wt{f}^1|\ci{I}$ with a periodisation $\Pi\ci{I}^{N(I)}(\wt{f}^1|\ci{I})$ of it over $I$, for some choice of frequency $N(I)$. After this has been completed for every interval in $\ch(\cS^1)$, one will have obtained a new function $\wt{f}^2$ and a new collection of intervals $\cS^{2}:=\bigcup_{I\in\ch(\cS^1)}\ch^{N(I)}(I)$. Then, one can repeat this process in each of the intervals in $\ch(\cS^2)$ for the function $\wt{f}^2$, etc.

One finally obtains a new function $\wt{f}$. Note that this function is given as a composition of $f$ with a certain measure-preserving transformation (depending only on the choices of the frequencies), basically because each step in the iterative procedure amounted to composing with a measure-preserving transformation. Notice also that the choices of frequencies of each step of periodisation are separated from each other, so one has really complete freedom in performing them.

It is important to note that the function $\wt{f}$ can be obtained as the limit (in any reasonable sense) of a sequence of \textit{averaged} periodisations $\mathbb{E}\ci{\ch(\cS^1)}[\wt{f}^1],~\mathbb{E}\ci{\ch(\cS^2)}[\wt{f}^2],~\mathbb{E}\ci{\ch(\cS^3)}[\wt{f}^3],\ldots$, enabling us to keep track of the averages of the new function. It is also essential to note that since the iterative scheme consists in an iteration of the same fundamental construction (that of replacing by a periodisation), up to translating and rescaling, one deduces that an appropriately rescaled and translated version of \eqref {periodised martingale differences} will hold for each iteration over every interval in $\ch(\cS^1),~\ch(\cS^2),~\ch(\cS^3),\ldots$, namely
\begin{equation*}
\mathbb{E}\ci{\ch(\ch^{N(I)}(I))}[\wt{f}^{k+1}]-\La\wt{f}^{k}\Ra\ci{I}=\Pi\ci{I}^{N(I)}(\Delta\ci{I}\wt{f}^{k}),\;\forall I\in\ch(\cS^{k}),\;\forall k=1,2,\ldots,
\end{equation*}
so each difference $\mathbb{E}\ci{\ch(\cS^{k+1})}[\wt{f}^{k+1}]-\mathbb{E}\ci{\ch(\cS^{k})}[\wt{f}^{k}]$ can be written as a sum of periodisations of the martingale differences of $f$ over the intervals in $\ch^{k}([0,1))$. Thus $\wt{f}$ satisfies the original intuition. It is also worth noting that for the purpose of just obtaining estimates it is not necessary to go all the way down to $\wt{f}$, one can stop only after a finite number of steps.

J. Bourgain's technique in \cite{Bourgain} works really well in the unweighted setting of Banach space valued estimates, but in situations of weighted estimates, such as the setup of Sarason's conjecture, it has the drawback that in general it gives no control over strong dyadic smoothness of weights, even if the original weights are dyadically smooth, basically because it gives no control over averages taken over consecutive dyadic intervals, so it is not well-suited for problems involving fattened $A_p$ characteristics. In order to overcome this difficulty, F. Nazarov \cite{Nazarov} came up with the idea of ``keeping endpoints'', as a means of controlling intervals touching each other.

Namely, one replaces $f$ with with the function $\wt{f}^1$ which is equal to $\Pi\ci{I_0}^{N(I_0)}f$ on each interval in $\ch^{N(I_0)}(I_0)$ not touching the boundary of $I_0$, but equal to just the average $\La f\Ra\ci{I_0}=\La \Pi\ci{I_0}^{N(I_0)}f\Ra\ci{J}$ over each interval $J\in\ch^{N(I_0)}(I_0)$ that touches the boundary of $I_0$. Moreover, one considers the collection $\cS^1$ of intervals in $\ch^{N(I_0)}(I_0)$ that do not touch the boundary of $I_0$, and simply forgets the ones that touch it.

Then, one follows the same iterative scheme as above, always putting averages over intervals touching the boundary, and then forgetting those intervals. One has again complete freedom in choosing the frequencies, and this allowed F. Nazarov to reduce the estimate of the norm of the Hilbert transform over a \emph{weighted} $L^2$ space to estimating the norm of the square function over the same weighted $L^2$ space. Just as before, $\wt{f}$ can be realised as the limit of the sequence of averaged counterparts $\mathbb{E}\ci{\ch(\cS^1)}[\wt{f}^1],~\mathbb{E}\ci{\ch(\cS^2)}[\wt{f}^2],~\mathbb{E}\ci{\ch(\cS^3)}[\wt{f}^3],\ldots$. The latter sequence allowed F. Nazarov to deduce that this process, termed by F. Nazarov remodeling, produces (as will be explained below in \ref{s: smooth Haar mult strong dyad smooth}) strongly dyadically smooth weights, provided that the original weights are dyadically smooth, precisely because original averages are put in intervals that touch the boundary. Of course, one can again stop only after a finite number of steps.

Although F. Nazarov's remodeling from \cite{Nazarov} behaves really well with respect to smoothness, it has the drawback that the new functions are not given just as composition of the original functions with a certain measure-preserving transformation (as was the case in Bourgain's technique \cite{Bourgain}) due to putting averages over intervals touching the boundary and then forgetting these intervals. As a consequence, one-weight situations of weights $w,\sigma$ satisfying $w\sigma^{p-1}=1$ a.e. on $[0,1)$, as the ones that we are primarily interested in here, will in general be transformed to two-weight situations of weights $\wt{w},\wt{\sigma}$ not satisfying any such relation. To overcome this difficulty and at the same time preserve smoothness, one has essentially to not just forget the intervals that touch the boundary, but rather apply again remodeling in them, and do the same for all intervals touching the boundary that will ever come up. Thus, one can say that one has to apply \textit{iterated} remodeling.

We also note that if one is interested in estimates for the norm of the Hilbert transform over weighted $L^{p}$ spaces for any $1<p<\infty$ (not just $p=2$), then one cannot just reduce the estimate of this norm to the estimate of the norm of the square function or the Haar multiplier over the same weighted space, but rather one has to use some other slightly more complicated Haar shift, like the one introduced in Subsection \ref{s:LS Haar shift}:
\begin{equation*}
Tf:=2\sum_{I\in\cD}(\ssDelta\ci{I}f)(h\ci{I_{+}}-h\ci{I_{-}}).
\end{equation*}
This will force us to move one generation deeper during remodeling, that is to consider grandchildren rather than just children of intervals in $\cS^1,\cS^2,\ldots$, essentially because this Haar shift involves interaction of intervals with their children. We emphasize (and it will become clear in Subsection \ref{s:smooth Haar mult}) that for the purpose of obtaining examples just for dyadic operators (e.g. Haar multipliers, dyadic maximal function) one can use just children of intervals. The reduction of the estimate for the Hilbert transform to that for the special Haar shift of Subsection \ref{s:LS Haar shift} is done in Subsection \ref{s:smooth Hilb}.

\subsection{The iterative construction}\label{s:iterated remodel}

We now describe in detail iterated remodeling.

Let $X$ be a uniformly bounded $\R ^n$-valued martingale on $[0,1)$, induced by a function $F\in L^{\infty}([0,1);\R^{n})$ (one should again think here of the special case of weighted estimates, where $n=4$ and $X$ is induced by the bounded function $F=(w,\sigma,\bff,\bg)$, where $\bff:=f\sigma$ and $\bg:=gw$).

Set $I_0:=[0,1)$ and $\wt{F}^{0}:=F$. Pick a frequencly $N(I_0)$ and replace $F$ with the function $\tilde{F}^{I_0}:=\Pi\ci{I_0}^{N(I_0)}F$. We can consider a family $\cR\ci{N(I_0)}(I_0)$ of regular stopping intervals (intervals not touching the boundary) and a family $\cE\ci{N(I_0)}(I_0)$ of exceptional stopping intervals (intervals touching the boundary).

Then, for all $J\in\cE\ci{N(I_0)}(I_0)$, we do the same thing in $J$ for the function $(\tilde{F}^{I_{0}})|\ci{J}=F\circ\psi\ci{J,I_0}$, with respect to some new choice of frequency $N(J)$, obtaining a family $\cR\ci{N(J)}(J)$ of regular stopping intervals and a family $\cE\ci{N(J)}(J)$ of exceptional stopping intervals. We afterwards repeat this in each new exceptional stopping interval that will have come up, etc. We continue this until the entire $I_0$ has been covered, up to a Borel set of zero measure, by regular stopping intervals. We note that this will happen because the sum of the measures of the exceptional stopping intervals decays at each step at least geometrically with ratio 1/2.

After this process has been completed, we will have obtained a new function $\wt{F}^1$. We denote by $\cS^1$ the family of all regular stopping intervals that will have been collected during this procedure. We also denote by $\hat{\cS}^1$ the family of all exceptional stopping intervals that will have been collected during this procedure, together with $I_0$. We define the starting intervals of order 1 as all elements of the family $\hat{\cS}^1$. Note that the elements of $\cS^1$ are pairwise disjoint and $\bigcup\cS^1=I_0$ up to a Borel set of zero measure. Note also that $\wt{F}^1|\ci{I}=F\circ\psi\ci{I,I_0}$, for all $I\in\cS^1$.

For the next step, we repeat the same procedure in the interval $I$ and for the function $\wt{F}^1|\ci{I}$, for all $I\in\ch^{2}(\cS^1)$ (and not just $\ch(\cS^1)$). Here we note that $\wt{F}^1|\ci{I}=F\circ\psi\ci{I,J}$ for some \emph{grandchild} $J$ of $I_0$, for all $I\in\ch^{2}(\cS^1)$. After this has been completed for all intervals in $\ch^{2}(\cS^1)$, we will have obtained a new function $\tilde{F}^2\in L^{\infty}([0,1);\R^n)$. We denote by $\cS^2$ the family of all regular stopping intervals that will have been collected during this step. Moreover, we denote by $\hat{\cS}^2$ the family of all new exceptional stopping intervals that will have been collected during this step, together with all intervals in $\ch^2(\cS^1)$. We define the starting intervals of order 2 as all elements of the family $\hat{\cS}^2$.

Afterwards, we repeat the same procedure along the interval $I$ and for the function $\tilde{F}^2|\ci{I}$, for all $I\in\ch^{2}(\cS^2)$, etc.

After this process has been completed, we will have obtained a sequence of functions $\wt{F}^1,\wt{F}^2,\wt{F}^3,\ldots$ and a new function $\wt{F}\in L^{\infty}([0,1);\R^{n})$.

\subsubsection{Measure-preserving transformation} \label{s:Measure-preserving transformation} It is important to note that this process of iterated remodeling amounts just to composition of limiting functions with a certain measure-preserving transformation that depends only on the choice of frequencies. Indeed, is is clear that for all $l=0,1,2,\ldots$, there exists a measure-preserving transformation $\Psi_{l}:[0,1)\rightarrow[0,1)$ such that $\wt{F}^{l}=\wt{F}^{l-1}\circ\Psi_{l}$, for all $l=1,2,\ldots$. Then, we have $\wt{F}=F\circ\Psi$, where $\Psi:[0,1)\rightarrow[0,1)$ is the measure-preserving transformation given at almost every point of $[0,1)$ as the composition of these measure-preserving transformations $\Psi_\1\circ\Psi_2\circ\Psi_3\circ\ldots$. Note that $\Psi$ depends only on the choices of frequecies $N(I),\; I\in\hat{\cS}:=\bigcup_{k=1}^{\infty}\hat{\cS}^{k}$.

So in particular, it does not matter whether we apply iterated remodeling with respect to a given choice of frequencies to a martingale as a whole or to each of its coordinates separately with respect to the same choice of frequencies.

\subsubsection{Averaged counterparts}\label{s:averaged counterparts} Note that the inductive procedure will have also produced the families $\cS^1,\cS^2,\cS^3,\ldots$ of all  regular stopping intervals that will have been collected during the first, second, third etc respectively step, and the families $\hat{\cS}^1,\hat{\cS}^2,\hat{\cS}^3,\ldots$ of all starting intervals of order $1,2,3,\ldots$ respectively. Then, one can realise $\wt{F}$ as the limit, say, pointwise a.e. on $[0,1)$ and in $L^2([0,1);\R^{n})$, of the sequence of averaged counterparts $\wt{X}^0,\wt{X}^1,\wt{X}^2,\ldots$, where $\wt{X}^{0}:=X_0\equiv\La F\Ra\ci{I_0}$ and $\wt{X}^{k}:=\mathbb{E}\ci{\ch^{2}(\cS^{k})}[\wt{F}^{k}]$, for all $k=1,2,\ldots$.

\begin{rem}\label{r:pres aveg}
It is clear that for all $k=1,2$ we have $\mathbb{E}\ci{\ch^{k}(\cS^1)}[\wt{F}]=\mathbb{E}\ci{\ch^{k}(\cS^1)}[\wt{F}^1]=X_{k}\circ\Psi$, where recall that $X_{k}=\mathbb{E}\ci{\ch^{k}(I_0)}[F]$. Since the iterative scheme consists in an iteration of the same fundamental construction, up to translating and rescaling, we deduce that
\begin{equation*}
\mathbb{E}\ci{\ch^{k}(\cS^{l})}[\wt{F}]=X_{2l+k}\circ\Psi,\qquad\forall k=1,2,\qquad\forall l=1,2,\ldots,
\end{equation*}
where $X_{2l+k}=\mathbb{E}\ci{\ch^{2l+k}(I_0)}[F]$. In particular, the family of all averages of $\wt{F}$ over dyadic intervals coincides with the family of all averages of $F$ over dyadic intervals. This had been noted in \cite[\S 10]{Nazarov}.
\end{rem}

\subsubsection{Martingale difference decomposition} We now provide a description for the martingale difference decomposition of the function $\wt{F}$. Note here that the iterative scheme involved considering \emph{grandchildren} of $\cS^1,\cS^2,\ldots$, rather than just children. This means that the martingale difference decomposition of $\wt{F}$ will involve periodisations of \emph{second order} martingale differences of $F$, and not just of  martingale differences of $F$ (unlike Bourgain's \cite{Bourgain} construction and Nazarov's \cite{Nazarov} constructions). At the same time, the fact that we do distinguish between intervals that touch the boundary and intervals that do not means that these periodisations will extend only over intervals that do not touch the boundary, so there will be \emph{quasi-periodisations} rather than just periodisations (like Nazarov's \cite{Nazarov} construction, but unlike Bourgain's \cite{Bourgain} construction).

Namely, define the second order martingale difference $\Delta^{2}\ci{I}f$ of a function $f\in L^{\infty}(I;\R^{n})$ over an interval $I$ by
\begin{equation}
\label{second order martingale difference}
\Delta^{2}\ci{I}f:=\mathbb{E}\ci{\ch^{2}(I)}f-\La f\Ra\ci{I}\1\ci{I}=\Delta\ci{I}f+\sum_{J\in\ch(I)}\Delta\ci{J}f.
\end{equation}
Moreover, given a frequency $N$, define the \emph{averaged quasi-periodisation} $\overline{\text{Q}\Pi}\ci{I}^{N}f$ of $f$ of frequency $N$ over $I$ as the function $\overline{\text{Q}\Pi}\ci{I}^{N}f:=\mathbb{E}\ci{\cE\ci{N}(I)\cup\ch^{2}(\cR\ci{N}(I))}[\Pi\ci{I}^{N}f]$, i.e.
\begin{equation}
\label{averaged quasi-periodisation}
\overline{\text{Q}\Pi}\ci{I}^{N}f(x)=\
\begin{cases}
(\mathbb{E}\ci{\ch^2(I)}[f]\circ\psi\ci{J,I})(x),\text{ if }x\text{ belongs to some }J\in\cR\ci{N}(I)\\\\
\La f\Ra\ci{I},\text{ if }x\text{ belongs to some }J\in\cE\ci{N}(I)
\end{cases}
.
\end{equation}
Note that
\begin{equation}
\label{one-step remodel martingale diff}
\overline{\text{Q}\Pi}\ci{I}^{N}f-\langle f\rangle\ci{I}\1\ci{I}=\overline{\text{Q}\Pi}\ci{I}^{N}(\Delta^{2}\ci{I}f).
\end{equation}
(notice that $\Delta^{2}\ci{I}f$ is \emph{constant} on the grandchildren of $I$). It is clear that
\begin{equation*}
\wt{F}^1(x)=\Pi\ci{J}^{N(J)}(F\circ\psi\ci{J,I_0})(x),\qquad\forall x \in\bigcup\cR\ci{N(J)}(J),\qquad\forall J\in\hat{\cS^1}.
\end{equation*}
Therefore, we deduce
\begin{align}
\label{from previous to next averaged term}
\wt{X}^{1}-\wt{X}^{0}=\sum_{J\in\hat{\cS}^{1}}(\overline{\text{Q}\Pi}\ci{J}^{N(J)}(F\circ\psi\ci{J,I_0})-\La F\Ra\ci{I_0}\1\ci{J}),
\end{align}
which implies
\begin{equation}
\label{from previous to next averaged term martingale differences}
\wt{X}^1-\wt{X}^{0}=\sum_{J\in\hat{\cS}^1}\overline{\text{Q}\Pi}\ci{J}^{N(J)}(\Delta^{2}\ci{J}(F\circ\psi\ci{J,I_0}))=
\sum_{J\in\hat{\cS}^1}\overline{\text{Q}\Pi}\ci{I_0}^{N(J)}(\Delta^{2}\ci{I_0}F)\circ\psi\ci{J,I_0}.
\end{equation}
For all $J\in\hat{\cS^1}$, we call the function $D\ci{J}F:=\overline{\text{Q}\Pi}\ci{J}^{N(J)}(\Delta^{2}\ci{J}(F\circ\psi\ci{J,I_0}))$ contribution of the starting interval $J$ to the martingale difference decomposition of $\wt{F}$.

We emphasize again that the iterative scheme consists in an iteration of the same fundamental construction, up to translating and rescaling. Therefore, an appropriately rescaled and translated copy of \eqref{from previous to next averaged term martingale differences} will hold for each iteration over every interval in the collections  $\ch^2(\cS^1),~\ch^2(\cS^2),\ldots$. Therefore, one can write
\begin{equation*}
\wt{X}^{k+1}-\wt{X}^{k}=\sum_{J\in\hat{\cS}^{k+1}}D\ci{J}F,
\end{equation*}
where for all $J\in\hat{\cS}^{k+1}$ we have $D\ci{J}F=\overline{\text{Q}\Pi}\ci{J}^{N(J)}(\Delta^{2}\ci{I}(F\circ\psi\ci{J,I}))$ for some $I\in\ch^{2k}(I_0)$, for all $k=1,2,\ldots$. The reason for the ``$2k$'' is again that at the $(k+1)$-th step we repeat the same fundamental process inside each grandchild of each regular stopping interval of the $k$-th step. In perticular
\begin{equation}
\label{iterated remodeling martingale difference decomposition}
\wt{F}=\La F\Ra\ci{[0,1)}+\sum_{J\in\hat{\cS}}D\ci{J}F
\end{equation}
in $L^2([0,1);\R^{n})$, where $\hat{\cS}:=\bigcup_{k=1}^{\infty}\hat{\cS}^{k}$ is the family of all starting intervals.

\begin{rem}\label{r:end av}

Note that $\La \overline{\text{Q}\Pi}\ci{I}^{N}f\Ra\ci{J}=\La f\Ra\ci{I}$, for all dyadic subintervals $J$ of $I$ that touch its boundary. In particular $\La \overline{\text{Q}\Pi}\ci{I}^{N}(\Delta^{2}\ci{I}f)\Ra\ci{J}=0$, for all dyadic subintervals $J$ of $I$ that touch its boundary.

This observation, coupled with \eqref{from previous to next averaged term martingale differences} and a simple inductive argument yields that for all $k=0,1,2\ldots$, the average of $\wt{X}^{k}$ over every dyadic interval that touches the boundary of $[0,1)$ is equal to $\La F\Ra\ci{[0,1)}$. It follows that the average of $\wt{F}$ over every dyadic interval that touches the boundary of $[0,1)$ is equal to $\La F\Ra\ci{[0,1)}$.
\end{rem}

\section{The case of dyadic models}\label{s:smooth dyadic models}

In this section we apply iterated remodeling to obtain examples for dyadic models with weights possessing the required smoothness.

\subsection{Estimate for Haar multipliers}\label{s:smooth Haar mult}

Let $p\in(1,\infty)$. Let $M>2$. Let $\delta>0$ be arbitrarily small. Recall the Haar multiplier $T_{\e}$ corresponding to any choice of signs $\e=(\e\ci{I})_{\ci{I\in\cD}}$:
\begin{equation*}
T_{\e}f:=\sum_{I\in\cD}\e\ci{I}(\ssDelta\ci{I}f)h\ci{I}.
\end{equation*}
Recall that in Subsection \ref{s: SM Haar mult} we constructed bounded weights $w\sigma$ on $[0,1)$ with $\sigma=w^{-1/(p-1)}$ and
\begin{equation*}
M\leq[w]\ci{A_{p},\cD},\;\;\La w\Ra{\ci{[0,1)}}\La\sigma\Ra{\ci{[0,1)}}^{p-1}\leq 2^{p}4eM
\end{equation*}
and $S\ut{d}_{w},~S\ut{d}_{\sigma}\leq 1+\delta$, and non-zero bounded functions $f\in L^{p}(\sigma),~g\in L^{p'}(w)$, such that for the functions $\bff=f\sigma,~\bg=gw$ there holds
\begin{equation}
\frac{\sup_{\e}|\La T_{\e}(f\sigma),gw\Ra|}{\Vert f\Vert \ci{L^{p}(\sigma)}\Vert g\Vert \ci{L^{p'}(w)}}=\frac{\sum_{I\in\cD}|I|\cdot |\ssDelta\ci{I}\bff|\cdot|\ssDelta\ci{I}\bg|}{\Vert f\Vert \ci{L^{p}(\sigma)}\Vert g\Vert \ci{L^{p'}(w)}}\gtrsim_{p} M.
\end{equation}
We apply the iterated remodeling transform on the martingale induced by the function $(w,\sigma,\bff,\bg)$, for an arbitrary choice of frequencies. As it had been observed in \ref{s:Measure-preserving transformation}, this is the same as applying the iterated remodeling transform separately to each of the functions $w,\sigma,\bff,\bg$, for the same choice of frequencies. Then, the new martingale is induced by the function $(\wt{w},\wt{\sigma},\wt{\bff},\wt{\bg})$, where tilde denotes just composition with the mesure preserving-transformation $\Psi:[0,1)\rightarrow[0,1)$ of \ref{s:Measure-preserving transformation}. Then $\wt{\sigma},\wt{w}$ are weights on $[0,1)$ with $\wt{\sigma}=\wt{w}^{-1/(p-1)}$ a.e. on $[0,1)$.

\subsubsection{Respecting dyadic Muckenhoupt constants}\label{s: smooth Haar mult dyadic Muck}

Remark \ref{r:pres aveg} shows that for all $I\in\cD$ there exists $J\in\cD$ (depending only on the choices of frequencies) such that $\La \wt{w}\Ra\ci{I}=\La w\Ra\ci{J}$ and $\La \wt{\sigma}\Ra\ci{I}=\La \sigma\Ra\ci{J}$. It follows immediately that $[\wt{w}]\ci{A_{p},\mathcal{D}}=[w]\ci{A_{p},\mathcal{D}}$.

\subsubsection{Dominating strong dyadic smoothness via dyadic one}\label{s: smooth Haar mult strong dyad smooth}

Let $\e>0$. Assume that $\delta$ is small enough, so that $(1+\delta)^{3}\leq1+\e$. We claim that $S\ut{sd}_{\wt{w}}\leq1+\e$. Indeed, let $X$ be the martingale induced by the function $w$. Recall from \ref{s:averaged counterparts} that $\wt{w}$ is realized as the limit of the sequence of averaged counterparts $\wt{X}^{0},\wt{X}^{1},\wt{X}^{2},\ldots$. Recall the expression \eqref{from previous to next averaged term}:
\begin{align*}
\wt{X}^{1}-\wt{X}^{0}=\sum_{J\in\hat{\cS}^{1}}(\overline{\text{Q}\Pi}\ci{J}^{N(J)}(w\circ\psi\ci{J,I_0})-\La w\circ\psi\ci{J,I_0}\Ra\ci{J}\1\ci{J}).
\end{align*}
Note that the function $\wt{X}^{0}$ is constant, so $S\ut{sd}\ci{\wt{X}^{0}}=1$, and also that $S\ut{d}_{w}\leq 1+\delta$ by construction. Then, the following lemma, proved by F. Nazarov in \cite[\S 10]{Nazarov}, shows that $S\ut{sd}\ci{\wt{X}^{1}}\leq 1+\e$. Induction then gives $S\ut{sd}\ci{\wt{X}^{k}}\leq 1+\e$, for all $k=0,1,2,\ldots$. It follows that $S\ut{sd}_{\wt{w}}\leq1+\e$, independently of the choices of frequencies. The lemma shows that replacing a portion of a strongly dyadically smooth weight with an averaged quasi-periodisation of another dyadically smooth weight preserves the strong dyadic smoothness of the original weight.

\begin{lm}\label{l:remodel doubl}
Let $w$ be a weight on an interval $I\in\cD$, and assume that $S\ut{sd}_{w}\leq 1+\e$ for some $\e>0$. Let $J$ be a dyadic subinterval of $I$, such that $w$ is constant on $J$. Let $v$ be a weight on $J$ such that $\langle v\rangle\ci{J}=\langle w\rangle\ci{J}$ and $S\ut{d}_{v}\leq1+\delta$, where $\delta>0$ satisfies $(1+\delta)^{3}\leq1+\e$. Consider the weight $\wt{w}:=w+(\overline{\text{Q}\Pi}\ci{J}^{N}v)\1\ci{J}-\langle v\rangle\ci{J}\1\ci{J}$ on $I$, i.e.
\begin{equation*}
\wt{w}(x)=
\begin{cases}
w(x),\text{ if }x\notin J\\\\
(\overline{\text{Q}\Pi}\ci{J}^{N}v)(x),\text{ if }x\in J
\end{cases}
,\qquad\forall x\in I.
\end{equation*}
Then, there holds $S\ut{sd}_{\wt{w}}\leq 1+\e$.
\end{lm}
\begin{proof}
Let $K,L\in\cD(I)$ be adjacent with $|K|=|L|$. If either both $K$ and $L$ are not contained in $J$ or both $K$ and $L$ touch the boundary of $J$, we have
\begin{equation*}
\frac{\La\wt{w}\Ra\ci{K}}{\La\wt{w}\Ra\ci{L}}=\frac{\La w\Ra\ci{K}}{\La w\Ra\ci{L}}\leq S\ut{sd}_{w}\leq1+\e.
\end{equation*}
If one of $K,L$ is contained in $J$ and does not touch the boundary of $J$, then it is clear that $\La \wt{w}\Ra\ci{K}=\La v\Ra\ci{K'}$ and $\La \wt{w}\Ra\ci{L}=\La v\Ra\ci{L'}$ for some $K',L'\in\bigcup_{k=0}^{2}\ch^{k}(J)$, therefore 
\begin{equation*}
\frac{\La\wt{w}\Ra\ci{K}}{\La\wt{w}\Ra\ci{L}}\leq(S\ut{d}_{v})^{3}\leq(1+\delta)^{3}\leq1+\e,
\end{equation*}
concluding the proof.
\end{proof}

\subsubsection{Extending the weights to the entire real line}\label{s: smooth Haar mult extend}

Consider now the weights $\wt{w}',\wt{\sigma}'$ on $\R $ given by
\begin{equation*}
\wt{w}'(x)=
\begin{cases}
\tilde{w}(x-k),\;\forall x\in(k,k+1),\text{ if }k\text{ is even}\\
\wt{w}(k+1-x),\;\forall x\in(k,k+1),\text{ if }k\text{ is odd}
\end{cases}
,\qquad\forall k\in\mathbb{Z},
\end{equation*}
and similarly for $\wt{\sigma}'$. Obviously $\wt{\sigma}'=(\wt{w}')^{-1/(p-1)}$. Translation and reflection invariance shows immediately that $[\wt{w}']\ci{A_{p},\cD}=[\wt{w}]{\ci{A_{p},\cD([0,1))}}$. Moreover, translation and reflection invariance yields that $S\ut{sd}_{\wt{w}'}$ over $[k,k+1)$ is equal to $S\ut{sd}_{\wt{w}}$, for all $k\in\mathbb{Z}$. Noticing now that for all adjacent $I,J\in\cD$ with $|I|=|J|$ whose common endpoint is an integer there holds $\La\wt{w}'\Ra\ci{I}=\La\wt{w}'\Ra\ci{J}$, we deduce $S\ut{sd}_{\wt{w}'}= S\ut{sd}_{\wt{w}}$. Similarly $S\ut{sd}_{\wt{\sigma}'}= S\ut{sd}_{\wt{\sigma}}$.

For any $\e>0$, one can then achieve $S_{\wt{w}'},S_{\wt{\sigma}'}\leq1+\e$ and $[\wt{w}']\ci{A_{p}}\lesssim_{p}[\wt{w}']\ci{A_{p},\mathcal{D}}=[w]\ci{A_{p},\mathcal{D}([0,1))}$ by taking $\delta>0$ sufficiently small, per Lemmas \ref{l:naz1} and \ref{l:naz2} respectively.

\begin{rem} \label{r:dep Muck char}
We notice that the above estimates yield that the Muckenhoupt characteristic $[\wt{w}']\ci{A_{p}}$ is comparable to $M$ but in an exponential way with respect to $p$. In fact, we get $M\leq[\tilde{w}']\ci{A_{p}}\leq 2^{p}5eM$. If one cares only about dyadic Muckenhoupt charasteristics and ignores the ``small step" requirement, then as we saw in Section \ref{s:LS ex} it is possible to give an example with dyadic Muckenhoupt characteristic comparable to $M$ within absolute constants.
\end{rem}

\subsubsection{Respecting weighted norms}\label{s: smooth Haar mult weighted norms}

Consider the functions $\wt{f}'=(\wt{\bff}/\wt{\sigma})\1\ci{[0,1)}$, $\wt{g}'=(\wt{\bg}/\wt{w})\1\ci{[0,1)}$ on the real line. Identically to the case of the ``small-step" transform, see \eqref{SM Haar mult weighted norms}, we have $\Vert\wt{f}'\Vert\ci{L^{p}(\wt{\sigma}')}=\Vert f\Vert\ci{L^{p}(\sigma)}$ and $\Vert\wt{g}'\Vert\ci{L^{p'}(\wt{w}')}=\Vert g\Vert\ci{L^{p}(w)}$.

\subsubsection{Getting the damage}\label{s: smooth Haar mult damage}

It remains now to verify that we get the desired damage. 

\begin{lm}\label{l:remodel pres}
Let $\bff,\bg,\wt{\bff},\wt{\bg}$ be as above. There holds
\begin{equation*}
\sum_{I\in\cD}|I|\cdot|\ssDelta\ci{I}\wt{\bff}|\cdot|\ssDelta\ci{I}\wt{\bg}|=\sum_{J\in\cD}|J|\cdot|\ssDelta\ci{J}\bff|\cdot|\ssDelta\ci{J}\bg|.
\end{equation*}
\end{lm}
\begin{proof}
First of all, since $\sum_{I\in\cS^1}|I|=|I_0|$, where $I_0:=[0,1)$,  \eqref{from previous to next averaged term martingale differences} coupled with a translation and rescaling argument yields
\begin{equation*}
\sum_{I\in\cS^1\cup\ch(\cS^1)}|I|\cdot|\ssDelta\ci{I}\wt{\bff}|\cdot|\ssDelta\ci{I}\wt{\bg}|=|I_0|\cdot|\ssDelta\ci{I_0}\bff|\cdot|\ssDelta\ci{I_0}\bg|+\sum_{J\in\ch(I_0)}|J|\cdot|\ssDelta\ci{J}\bff|\cdot|\ssDelta\ci{J}\bg|,
\end{equation*}
independently of the choices of frequencies. Since the iterative scheme consists in iteration of the same fundamental construction, up to translating and rescaling, over every interval in $\ch^2(\cS^1),~\ch^2(\cS^2),\ldots$, we deduce
\begin{equation*}
\sum_{I\in\cS^{k+1}\cup\ch(\cS^{k+1})}|I|\cdot|\ssDelta\ci{I}\wt{\bff}|\cdot|\ssDelta\ci{I}\wt{\bg}|=\sum_{J\in\ch^{2k}(I_0)}|J|\cdot|\ssDelta\ci{J}\bff|\cdot|\ssDelta\ci{J}\bg|+\sum_{J\in\ch(\ch^{2k}(I_0))}|J|\cdot|\ssDelta\ci{J}\bff|\cdot|\ssDelta\ci{J}\bg|,
\end{equation*}
for all $k=1,2,\ldots$. This yields immediately the desired result.
\end{proof}

\begin{rem} \label{r:max func} Consider the dyadic Hardy-Littlewood maximal functions $M\bff,M\wt{\bff}$ of $\bff,\wt{\bff}$ respectively. Then, similarly to Remark \ref{r:SM max func} we have that the function $|\wt{\bff}|$ is obtained from the function $|\bff|$ through the same iterated remodeling transform as the function $\wt{\bff}$ is obtained from the function $\bff$. Remark \ref{r:pres aveg} yields then $M\wt{\bff}=(M\bff)\circ\Psi$ a.e. on $[0,1)$.

This observation, coupled with Remark \ref{r:SM max func}, shows that any ``large step'' family of examples establishing sharpness of weighted estimates for the dyadic Hardy-Littlewood maximal function over $[0,1)$ (see \cite{Buckley}) yields a family of examples (on the entire real line) with weights of arbitrary smoothness achieving that, in exactly the same way that this was done for the Haar multipliers above.
\end{rem}

\begin{rem}
We see that in this simple case of dyadic models, the choices of frequencies were irrelevant. It is also clear that one could have considered just children of intervals instead of grandchildren. We will however see that in the more subtle case of the Hilbert transform, frequencies will have to be chosen appropriately in order to achieve localization of the action of the operator, and considering grandchildren instead of just children will be essential, given the nature of the special Haar shift.
\end{rem}

\subsection{Muckenhoupt weights taking only two values with prescribed smoothness}\label{s:counterint proof}

We now show how the discussion in Subsection \ref{s:smooth Haar mult} implies the result of Proposition \ref{p:counterint}.

Let $p\in(1,\infty)$. Let $Q>1$. Let $\e>0$ be arbitrarily small. Choose $A_0,B_0>0$ with $A_0B_0^{p-1}=Q$. By the results in the appendix we have that there exist $a_1,b_1,a_2,b_2>0$, such that $a_1b_1^{p-1}=a_2b_2^{p-1}=1$ and $A_0=(a_1+a_2)/2$, $B_0=(b_1+b_2)/2$. Consider the weights $w,\sigma$ on $[0,1)$ given by
\begin{equation*}
w:=a_11{\ci{I_{1}}}+a_21{\ci{J_{1}}},\qquad\sigma:=b_11{\ci{I_{1}}}+b_21{\ci{J_{1}}},
\end{equation*}
where $I_1=\left[0,\frac{1}{2}\right)$ and $J_1=\left[\frac{1}{2},1\right)$. Then $w,\sigma$ are bounded, $\sigma=w^{-1/(p-1)}$ and $[w]\ci{A_{p},\cD}=w([0,1))\sigma([0,1))^{p-1}=A_0B_0^{p-1}=Q$. It is also obvious that $S\ut{d}_{w},S\ut{d}_{\sigma}<\infty$. Choose a sufficiently large positive integer $d>100$. Apply ``small step" transform to the weights $w,\sigma$ of order $d$, in order to obtain new weights $\wt{w},\wt{\sigma}$ respectively on $[0,1)$, and then the iterated remodeling transform on the functions $\wt{w},\wt{\sigma}$, for an arbitrary choice of frequencies (the same for both functions), in order to obtain new weights $\wt{w}',\wt{\sigma}'$ respectively on $[0,1)$. Extend the latter weights to weights $\wt{w}'',\wt{\sigma}''$ respectively on $\R $ as in \ref{s: smooth Haar mult extend}. Then, combining the results of Subsections \ref{s: SM Haar mult} and \ref{s:smooth Haar mult} we have $\wt{\sigma}''=\wt{w}''^{-1/(p-1)}$, $Q\leq[\wt{w}'']\ci{A_{p}}\leq2^{p}(5/4)Q$ and $S_{\wt{w}''}, S_{\wt{\sigma}''}\leq 1+\e$, for small enough $\e$. Moreover, we have $\wt{w}''\in\lbrace a_1,a_2\rbrace$ a.e. on $\R $, since $\wt{w}'$ is obtain from $w$ via composition with measure-preserving transformations.

\section{The case of the Hilbert transform}\label{s:smooth ex Hilb}

In this section we apply iterated remodeling transform on the martingales in the ``small step" example of Subsection \ref{s:SM Haar shift}, in order to obtain a ``small step" example for the Hilbert transform, proving Theorem \ref{t:main res}. We then show how this leads to a counterexample to the $L^{p}$ version of Sarason's conjecture.

\subsection{Estimate for the Hilbert transform}\label{s:smooth Hilb}

We first recall what we achieved in Subsection \ref{s:SM Haar shift}. Recall the special Haar shift $T$ from Subsection \ref{s:SM Haar shift}:
\begin{equation*}
Tf=2\sum_{I\in\cD}(\ssDelta\ci{I}f)(h{\ci{I_{+}}}-h{\ci{I_{-}}}).
\end{equation*}
Let $p\in(1,\infty)$ and $M>2$. Let $\delta>0$ be arbitrarily small. We constructed bounded weights $w,\sigma$ on $[0,1)$, such that $\sigma=w^{-1/(p-1)}$,
\begin{equation*}
M\leq w([0,1))\sigma([0,1))^{p-1},\;[w]\ci{A_{p},\cD} \leq 2^{p}4Me
\end{equation*}
and $w([0,1))\sim M$, $\sigma([0,1))\sim_{p}1$, and also $S\ut{d}_{w},S\ut{d}_{\sigma}\leq 1+\delta$, and non-zero bounded functions $f\in L^{p}(\sigma)$, $g\in L^{p'}(w)$, such that
\begin{align}
\label{est Haar shift}
\La f\sigma,T(gw)\Ra=\sum_{J\in\cD}(\ssDelta\ci{J}(gw))(\ssDelta\ci{J_{+}}(f\sigma)-\ssDelta\ci{J_{-}}(f\sigma))|J|\geq C_{p} M\Vert f\Vert \ci{L^{p}(\sigma)}\Vert g\Vert \ci{L^{p'}(w)}.
\end{align}
Moreover, by construction for the functions $\bff:=f\sigma,~\bg=:gw$ there holds $\ssDelta\ci{I}\bg=0$, for all dyadic intervals $I$ of \emph{odd} generation, and $(\ssDelta\ci{I}\bg)(\ssDelta\ci{I_{-}}\bff)=0\leq(\ssDelta\ci{I}\bg)(\ssDelta{\ci{I_{+}}}\bff)$, for all dyadic intervals $I$ of \emph{even} generation. Note that then
\begin{equation*}
|\La \bff,T(\bg)\Ra|=\La \bff,T(\bg)\Ra=\sum_{J}(\ssDelta\ci{J}\bg)(\ssDelta\ci{J_{+}}\bff)|J|,
\end{equation*}
where the summation runs over all $J\in\cD$ that are of \emph{even} generation.

\subsubsection{Setting up iterated remodeling}\label{s: smooth Hilb set-up}

We apply the iterated remodeling transform on the functions $w,\sigma,\bff,\bg$, for some choices of frequencies to be determined later (the same for all functions), obtaining functions $\wt{w},\wt{\sigma},\wt{\bff},\wt{\bg}$ respectively. We extend $\wt{w},\wt{\sigma}$ to weights on the whole real line having the desired smoothness and Muckenhoupt characteristic properties, as in \ref{s: smooth Haar mult extend}. Let us abuse the notation and denote these extensions by the same letter.

\begin{rem}\label{r:pres end}
From Remark \ref{r:end av} we deduce that $\La\wt{w}\Ra\ci{I}=w([0,1))=\wt{w}([0,1))$, for all dydic subintervals $I$ of $[0,1)$ that touch its boundary, and similarly for $\wt{\sigma}$. This observation will be crucial later in Subsection \ref{s: counter sar}.
\end{rem}

We denote by $H$ the Hilbert transform on the real line. We consider the operator $H(\cdot \wt{\sigma})$, acting from $L^{p}(\wt{\sigma})$ into $L^{p}(\wt{w})$. Consider the functions $\wt{f}=(\wt{\bff}/\wt{\sigma})\1_{[0,1)}$, $\wt{g}=(\wt{\bg}/\wt{w})\1_{[0,1)}$ on the real line. Our goal is to show that if the frequencies are chosen appropriately through an inductive procedure, then one can achieve
\begin{equation}
\label{goal est}
|\La \wt{\bff},H(\wt{\bg})\Ra|=|\La\wt{f}\wt{\sigma},H(\wt{g}\wt{w})\Ra|\gtrsim_{p}M\Vert f\Vert \ci{L^{p}(\sigma)}\Vert g\Vert \ci{L^{p'}(w)}.
\end{equation}
Assuming that this has been achieved, we will have (since the Hilbert transform is antisymmetric)
\begin{align*}
\Vert  H\Vert {\ci{L^{p}(\wt{w})}}=
&\Vert  H(\cdot\wt{\sigma})\Vert {\ci{L^{p}(\wt{\sigma})\rightarrow L^{p}(\wt{w})}}\geq\frac{|\La H(\wt{f}\wt{\sigma}),\wt{g}\wt{w}\Ra|}{\Vert \wt{f}\Vert {\ci{L^{p}(\wt{\sigma})}}\Vert \wt{g}\Vert {\ci{L^{p'}(\wt{w})}}}=\frac{|\La\wt{f}\wt{\sigma},H(\wt{g}\wt{w})\Ra|}{\Vert f\Vert \ci{L^{p}(\sigma)}\Vert g\Vert \ci{L^{p'}(w)}}\gtrsim_{p}M
\end{align*}
and hence the desired ``small step" example for the Hilbert transform.

\subsubsection{Decomposing the bilinear form} \label{s:decomp}

We begin by writing the functions $\wt{\bff}$, $\wt{\bg}$ as the unconditional sums of their martingale differences in $L^2([0,1))$ (up to a constant) as in \eqref{iterated remodeling martingale difference decomposition}, i.e.
\begin{align}
\label{remodel marting decomp}
\wt{\bff}=\La\bff\Ra\ci{[0,1)}+\sum_{I\in\hat{\cS}}D\ci{I}\bff,\qquad\wt{\bg}=\La\bg\Ra\ci{[0,1)}+\sum_{I\in\hat{\cS}}D\ci{I}\bg,
\end{align}
and similarly for $\wt{\bg}$, where $\hat{\cS}:=\bigcup^{\infty}_{k=1}\cS^{k}$ is the family of all starting intervals and $D\ci{I}\bff,D\ci{I}\bg$ are the contributions of the starting interval $I$ to the martingale differences decomposition of $\wt{\bff},\wt{\bg}$ respectively. Since the Hilbert transform is bounded in $L^2(\R )$ and antisymmetric, we have
\begin{align}
\label{decomp}
&\La H(\wt{\bg}\1_{[0,1)}),\wt{\bff}\1_{[0,1)}\Ra
=\sum_{I\in\hat{\cS}}\La H(D\ci{I}\bg),D\ci{I}\bff\Ra+\text{cross terms},
\end{align}
where the cross terms consist of pairings involving either the average of $\bff$ or $\bg$ over $[0,1)$ and the contribution of some starting interval, or contributions of different starting intervals.

Our object is to show that the main term in the right-hand side of \eqref{decomp} produces the desired damage, while the sum of the cross terms can be forced to be arbitrarily close to $0$, through an appropriate choice of frequencies (thus essentially achieving localization of the action of the operator).

\subsubsection{Forcing the sum of the cross terms to be arbitrarily small}\label{s: choose remodel param}

We need the following lemma, whose statement is mentioned in \cite[\S 12]{Nazarov}, showing essentially that the functions $D\ci{I}\bff,~D\ci{I}\bg$ oscillate arbitrarily fast for large enough frequency $N(I)$. Recall from \eqref{from previous to next averaged term martingale differences} that for all $I\in\hat{\cS}$, there exist mean zero functions $\phi\ci{I},~\psi\ci{I}\in L^{\infty}(I)$ such that $D\ci{I}\bff=\overline{\text{Q}\Pi}\ci{I}^{N(I)}\phi\ci{I}$ and $D\ci{I}\bg=\overline{\text{Q}\Pi}\ci{I}^{N(I)}\psi\ci{I}$.

\begin{lm}\label{l:osc}
Let $I\in\cD$. Let $\phi\in L^{\infty}(I)$ with $\La\phi\Ra\ci{I}=0$. Then, there holds $\overline{\text{Q}\Pi}\ci{I}^{N}\phi\rightarrow 0$ weakly in $L^{q}(I)$ as $N\rightarrow\infty$, for all $q\in(1,\infty)$ and $\overline{\text{Q}\Pi}\ci{I}^{N}\phi\rightarrow0$ weakly$^{\ast}$ in $L^{\infty}(I)$ as $N\rightarrow\infty$.
\end{lm}
\begin{proof}
It is clear from definition \eqref{averaged quasi-periodisation} of averaged quasi-periodisations that for all $N=3,4,\ldots$, there holds $\La h,\overline{\text{Q}\Pi}\ci{I}^{N}\phi\Ra=0$, for all functions $h$ on $I$ that are constant on all intervals in $\ch^{N}(I)$. Note that $\overline{\text{Q}\Pi}\ci{I}^{N}\phi$, $N=3,4,\ldots$ are uniformly bounded (say by $\Vert \phi\Vert\ci{L^{\infty}(I)}$) in $L^{\infty}(I)$. Then, an ``$\frac{\e}{3}$ argument" yields the desired result.
\end{proof}

Now, for all $I\in\cD:=\cD([0,1))$, set $\text{rk}(I):=-\log_{2}(\ell(I))$. Recall that $\cD=\bigcup_{k=0}^{\infty}\cD_{k}$, where $\cD_{k}:=\lbrace I\in\cD:\;\text{rk}(I)=k\rbrace$ is finite, for all $k=0,1,2,\ldots$. It follows immediately that one can enumerate the elements of the subset $\hat{\cS}$ of $\cD$ as $I_0,I_1,I_2,I_3,\ldots$, where $I_0:=[0,1)$, such that for all $0\leq l<k$ there holds $\text{rk}(I_{l})\leq\text{rk}(I_{k})$. Note then that in particular, for all $0\leq l<k$ we have either $I_k\cap I_l=\emptyset$ or $I_k\subset I_l$.

Note also that for all $I\in\hat{\cS}$, the functions $\phi\ci{I},~\psi\ci{I}\in L^{\infty}(I)$ depend only on  $\bff,~\bg$ and the choices of frequencies for starting intervals strictly containing $I$. Therefore, if for some $k=1,2,3,\ldots$ we have already picked $N(I_{l})$, for all $l=0,\ldots,k-1$, then by Lemma \ref{l:osc} we can choose the frequency $N(I_{k})$, in a way depending only on the previous choices and the functions $\bff,~\bg$, such that
\begin{align*}
T\ci{k}&:=H(\La \bff\Ra\ci{I_0}\1\ci{I_0}),D\ci{I_{k}}\bg\Ra+\La H(D\ci{I_{k}}\bff),\La \bg\Ra\ci{I_0}\1\ci{I_0})\Ra
\\&+\left\La H(D\ci{I_{k}}\bff),\sum_{l=0}^{k-1}D\ci{I_{l}}\bg\right\Ra+\left\La H\bigg(\sum_{l=0}^{k-1}D\ci{I_{l}}\bff\bigg),D\ci{I_{k}}\bg\right\Ra
\end{align*}
is as small in absolute value as we want (since the Hilbert transform is bounded in $L^2(\R)$). In particular, we can achieve $|T_{k}|\leq\frac{\e'}{2^{k+1}}$, where $\e':=\frac{cC_{p}}{2}M\Vert f\Vert \ci{L^{p}(\sigma)}\Vert g\Vert \ci{L^{p'}(w)}$, provided the choice of $N(I_{k})$ is allowed to depend also on $M,p$ and the functions $w,\sigma$. Here, $c>0$ is an absolute constant to be determined in Lemma \ref{l:est}.

Clearly the sum of cross terms is equal to $\sum_{k=1}^{\infty}T\ci{k}$, thus one can force this sum to be less that $\e'$ in absolute value, by choosing the frequencies to be large enough, in a way depending only on $M,p$ and the functions $w,\sigma,f,g$.

This way of forcing the sum of the cross terms to be arbitrarily close to 0 in absolute value is essentially the same as in \cite[\S 11]{Nazarov}. The choice of $\epsilon'$ is also the same as in \cite[\S 11]{Nazarov}, up to the constant $c$.

\subsubsection{Getting the damage from the main term}\label{damage Hilb}

We will now show that the main term in the right-hand side of \eqref{decomp} produces the desired damage, independently of the above choice of frequencies. More precisely, we will show that
\begin{align}
\label{damage first term}
\sum_{I\in\hat{\cS}^1}\langle H(D\ci{I}\bg),D\ci{I}\bff\rangle\;\leq-c(\ssDelta\ci{I_0}\bg)(\ssDelta\ci{(I_{0})_{+}}\bff)|I_0|,
\end{align}
independently of the choice of frequencies for intervals in $\hat{\cS}^1$, where $I_0:=[0,1)$. Keeping in mind that iterated remodeling as described here moves two generation deep at each step, we deduce through a translating and rescaling argument that
\begin{align*}
\sum_{I\in\hat{\cS}}\langle H(D\ci{I}\bg),D\ci{I}\bff\rangle\;\leq-c\sum_{\substack{J\in\cD\\J\text{ is of even generation}}}(\ssDelta\ci{J}\bg)(\ssDelta\ci{J_{+}}\bff)|J|=-c\La \bff,T(\bg)\Ra,
\end{align*}
i.e.
\begin{align*}
-\sum_{I\in\hat{\cS}}\langle H(D\ci{I}\bg),D\ci{I}\bff\rangle\geq c\La \bff,T(\bg)\Ra=c|\La\bff, T(\bg)\Ra|.
\end{align*}
The last equation, coupled with \eqref{decomp}, \eqref{est Haar shift} and the choice of $\e'$, implies \eqref{goal est} (with constant $\frac{cC_{p}}{2}$), yielding the desired result.

We now establish \eqref{damage first term}. Recall that the regular stopping intervals in $\cS^1$ cover $I_0$ up to a set of zero measure, so it suffices to show that $\langle H(D\ci{I}\bg),D\ci{I}\bff\rangle\leq-c(\ssDelta\ci{I_0}\bg)(\ssDelta\ci{(I_{0})_{+}}\bff)|\bigcup\cR\ci{N(I)}(I)|$, for all $I\in\hat{\cS}^1$. Recall from \eqref{from previous to next averaged term martingale differences} that for all $I\in\hat{\cS}^1$, $D\ci{I}\bg$ is just a rescaled and translated copy of $\overline{\text{Q}\Pi}\ci{I_0}^{N(I)}(\Delta^2\ci{I_0}\bg)$ over $I$, and similarly for $\bff$. Therefore, it suffices only to prove that
\begin{equation*}
\La H(\overline{\text{Q}\Pi}\ci{I_0}^{N}(\Delta^2\ci{I_0}\bg)),\overline{\text{Q}\Pi}\ci{I_0}^{N}(\Delta^2\ci{I_0}\bff)\Ra\leq-c(\ssDelta\ci{I_0}\bg)(\ssDelta\ci{(I_{0})_{+}}\bff)|\bigcup\cR\ci{N}(I_0)|,\;\forall N=3,4,\ldots.
\end{equation*}
Let us fix a positive integer $N\geq 3$. Recall  that from the definition \eqref{second order martingale difference} of the second order martingale differences we have
\begin{equation*}
\Delta^{2}\ci{I_0}\bg=(\ssDelta\ci{I_0}\bg)h\ci{I_0}+(\ssDelta\ci{(I_0)_{-}}\bg)h\ci{(I_0)_{-}}+(\ssDelta\ci{(I_0)_{+}}\bg)h\ci{(I_0)_{+}},
\end{equation*}
and similarly for $\bff$. It follows from definition \eqref{averaged quasi-periodisation} of averaged quasi-periodisations, and the facts that $\Delta^{2}\ci{I_0}\bg$ has mean zero and that it is \emph{constant} on the grandchildren of $I_0$, that
\begin{equation*}
\overline{\text{Q}\Pi}\ci{I_0}^{N}(\Delta^{2}\ci{I_0}\bg)=\sum_{J\in\mathcal{G}}[(\ssDelta\ci{I_0}\bg)h\ci{J}+(\ssDelta\ci{(I_0)_{-}}\bg)h\ci{J_{+}}+(\ssDelta\ci{(I_0)_{+}}\bg)h\ci{J_{-}}],
\end{equation*}
and similarly for $\bff$, where $\mathcal{G}:=\cR\ci{N}(I_0)$. Recall that $\ssDelta\ci{(I_0)_{+}}\bg=\ssDelta\ci{(I_0)_{-}}\bg=(\ssDelta\ci{I_0}\bg)(\ssDelta\ci{(I_0)_{-}}\bff)=0$ and that the Hilbert transform is antisymmetric. It follows that
\begin{equation*}
\La H(\overline{\text{Q}\Pi}\ci{I_0}^{N}(\Delta^2\ci{I_0}\bg)),\overline{\text{Q}\Pi}\ci{I_0}^{N}(\Delta^2\ci{I_0}\bff)\Ra=
(\ssDelta\ci{I_0}\bg)(\ssDelta\ci{(I_{0})_{+}}\bff)\big\La H\big(\sum_{J\in\mathcal{G}}h\ci{J},\big),\sum_{J\in\mathcal{G}}h\ci{J_{+}}\big\Ra.
\end{equation*}
Coupled with the fact that $(\ssDelta\ci{I_0}\bg)(\ssDelta\ci{(I_0)_{+}}\bff)\geq0$, the following lemma yields then the desired result.

\begin{lm}\label{l:est}

(a) For all intervals $I,J$ in $\R $ with $|I|=|J|$ and $I\cap J=\emptyset$, there holds
\begin{equation*}
\La H(h\ci{I}),h{\ci{J_{+}}}\Ra+\La H(h\ci{J}),h{\ci{I_{+}}}\Ra<0.
\end{equation*}

(b) There holds $\big\La H\big(\sum_{J\in\mathcal{G}}h\ci{J}),\sum_{J\in\mathcal{G}}h\ci{J_{+}},\big\Ra\leq-c\left|\bigcup\mathcal{G}\right|$, where $c=-\La H(h_{[0,1)}),h_{\left[\frac{1}{2},1\right)}\Ra\in(0,\infty)$.
\end{lm}

\begin{proof}
(a) First of all, direct computation gives
\begin{equation*}
H(h_{[0,1)})(x)=\frac{1}{\pi}\ln\left(\frac{4|x(x-1)|}{(2x-1)^2}\right)\text{ for almost every }x\in\R\setminus\left\lbrace0,\frac{1}{2},1\right\rbrace,
\end{equation*}
so $H(h_{[0,1)})$ can be identified as a smooth function on $\R \setminus\left\lbrace0,\frac{1}{2},1\right\rbrace$. Direct computation shows then that $H(h{\ci{[0,1)}})$ is strictly increasing and strictly concave in $(1,\infty)$, and strictly decreasing in $\left(\frac{1}{2},1\right)$.

Let now $I,J$ be intervals in $\R $ with $|I|=|J|$ and $I\cap J=\emptyset$. Without loss of generality, we may assume that $\inf J\geq\sup I$. Note that $H(h_{[0,1)})(1-x)=H(h_{[0,1)})(x)$, for all $x\in\R \setminus\left\lbrace0,\frac{1}{2},1\right\rbrace$. It follows that for almost every $x\in\R$, if we denote by $s(x)$ the symmetric point to $x$ with respect to the center of $I$, then we have $H(h\ci{I})(s(x))=H(h\ci{I})(x)$. Then, a simple symmetry and translation argument, illustrated in Figure \ref{intervals}, shows that
\begin{equation*}
\La H(h\ci{J}),h\ci{I_{+}}\Ra=-\La H(h\ci{I}),h\ci{J_{-}}\Ra.
\end{equation*}

\begin{figure}[h]\centering
\includegraphics[width=\linewidth]{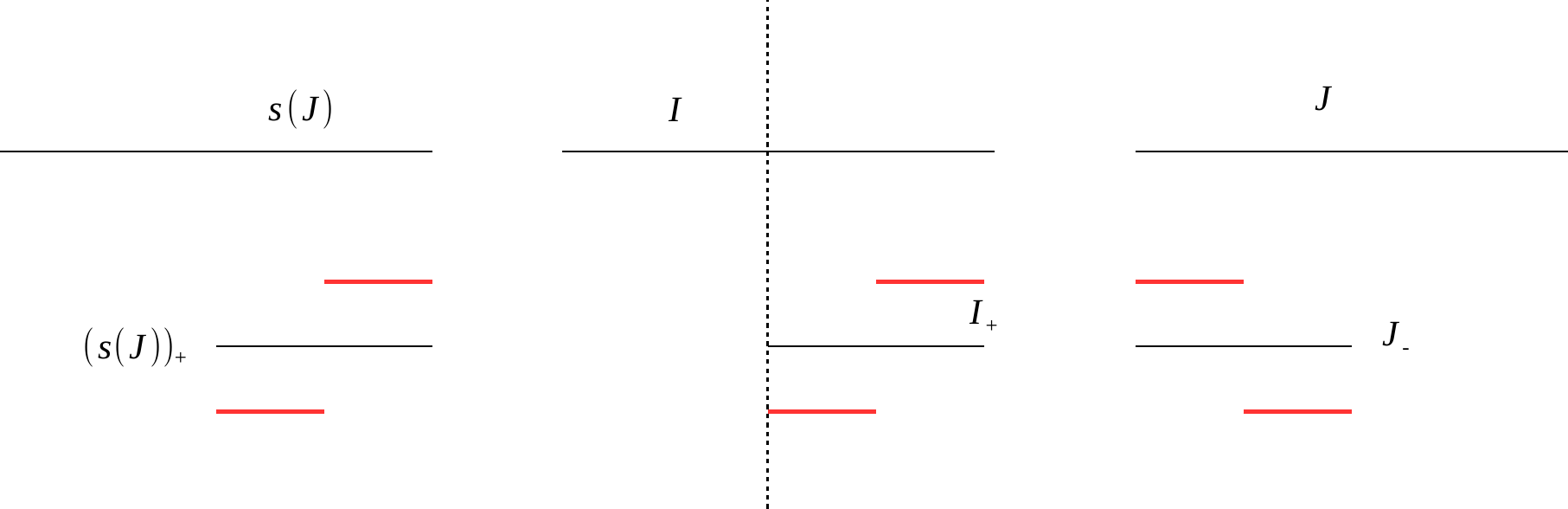}
\caption{Illustration of $\La H(h\ci{J}),h\ci{I_{+}}\Ra=-\La H(h\ci{I}),h\ci{J_{-}}\Ra$}
\label{intervals}
\end{figure}

Therefore, rescaling and translating we obtain
\begin{align*}
\La H(h\ci{I}),h\ci{J_{+}}\Ra+\La H(h\ci{J}),h\ci{I_{+}}\Ra&=\La H(h\ci{I}),h\ci{J_{+}}-h\ci{J_{-}}\Ra
=|I|\La H(h\ci{[0,1)}),h\ci{K_{+}}-h\ci{K_{-}}\Ra,
\end{align*}
for some interval $K$ in $\R$ with $|K|=1$ and $\inf K\geq 1$. Therefore, it suffices to prove that the continuous function $\La H(h\ci{[0,1)}),h\ci{\left[a,a+\frac{1}{2}\right)}\Ra$, $a\in[1,\infty)$ is strictly decreasing. This follows immediately from the fact that the function $H(h\ci{[0,1)})$ is strictly concave in $(1,\infty)$.

(b) Since $H(h\ci{[0,1)})$ is strictly decreasing in $\left(\frac{1}{2},1\right)$, we have $c:=-\La H(h_{[0,1)}),h_{\left[\frac{1}{2},1\right)}\Ra\in(0,\infty)$. Moreover, rescaling and translating we obtain $ \La H(h\ci{I}),h{\ci{I_{+}}}\Ra=|I|\La H(h_{[0,1)}),h\ci{\left[\frac{1}{2},1\right)}\Ra$, for all intervals $I$ in $\R$. Since the intervals in $\mathcal{G}$ are pairwise disjoint and have the same length, we deduce from (a)
\begin{align*}
\big\La H\big(\sum_{J\in\mathcal{G}}h\ci{J}\big),\sum_{J\in\mathcal{G}}h\ci{J_{+}}\big\Ra&=
\frac{1}{2}\sum_{\substack{J,K\in\mathcal{G}\\J\neq K}}(\langle H(h\ci{J}),h\ci{K_{+}}\Ra+
\La H(h\ci{K}),h\ci{J_{+}}\Ra)
+\sum_{J\in\mathcal{G}}\La H(h\ci{J}),h\ci{J_{+}}\Ra\\
&\leq-c\sum_{J\in\mathcal{G}}|J|=-c\left|\bigcup\mathcal{G}\right|,
\end{align*}
concluding the proof.
\end{proof}

\begin{rem}\label{r: comp ex}
The constructions show that for every fixed $M,\delta$ and $p$, one can give examples for the Hilbert transform and Haar multipliers differing only in the function $f$ (and in particular one can take $g=-\1_{[0,1)}$ in both cases).
\end{rem}

\begin{rem}
If we were interested just in two-weight estimates, then F. Nazarov's remodeling from \cite{Nazarov} would suffice, i.e. one could completely ignore exceptional stopping intervals (except for $[0,1)$ of course), and in fact one could even stop only after a finite number of steps, without losing damage or smoothness of weights. Iteration here only guarantees that the transforms are measure-preserving, so that one-weight situations remain such after applying them.
\end{rem}

\subsection{Counterexample to \texorpdfstring{$L^{p}$}{Lp} version of Sarason's conjecture}\label{s: counter sar}

Here we describe how the family of examples of Subsection \ref{s:smooth Hilb} will provide through a direct sum of singularities type construction a counterexample to the analog of Sarason's conjecture for every fixed $p$. Roughly speaking, by direct sum construction one should understand that the unit interval is partitioned into subintervals $J_1,J_2,\ldots$, and then each $J_{k}$ is equipped with an (appropriately shifted and rescaled) example from the previous section, in such a way that estimates of the norm of the operator blow up as $k\rightarrow\infty$.

Fix $p\in(1,\infty)$. Let $\delta>0$ be sufficiently small. For all $k=1,2,\ldots$, by Subsection \ref{s:smooth Hilb} we have that there exist bounded weights $w_{k},\sigma_{k}$ on $[0,1)$ with
\begin{equation*}
[w_{k},\sigma_k]\ci{A_{p},\cD}\sim_{p} k,\;\;w_{k}([0,1))\sim k,\;\;\sigma_{k}([0,1))\sim_{p}1,
\end{equation*}
 and $S\ut{sd}_{w_{k}},S\ut{sd}_{\sigma_{k}}\leq 1+\delta$, and non-zero functions $f_{k}\in L^{p}(\sigma_{k})$, $g_{k}\in L^{p'}(w_{k})$, such that
\begin{equation}
\label{sequence of examples}
|\La H(f_{k}\sigma_{k}\1_{[0,1)}),g_{k}w_{k}\1_{[0,1)}\Ra|\gtrsim_{p}k\Vert f_{k}\Vert{\ci{L^{p}(\sigma_{k})}}\Vert g_{k}\Vert{\ci{L^{p'}(w_{k})}}.
\end{equation}
Set $I_0:=[0,1)$ and 
\begin{equation*}
I_{k}=\left[0,\frac{1}{2^{k}}\right),\qquad J_{k}=\left[\frac{1}{2^{k}},\frac{1}{2^{k-1}}\right),\qquad k=1,2,\ldots.
\end{equation*}
For all $k=1,2,\ldots$, consider the weights $\wt{w}_{k},\wt{\sigma}_{k}$ on $J_{k}$ that are obtained as rescaled and shifted copies of the weights $\frac{1}{w_{k}([0,1))}w_{k},\frac{1}{\sigma_{k}([0,1))}\sigma_{k}$ respectively on the interval $J_{k}=\left[\frac{1}{2^{k}},\frac{2}{2^{k}}\right)$, i.e.
\begin{equation*}
\wt{w}_{k}(x)=\frac{1}{w_{k}([0,1))}w_{k}(2^{k}x-1),\;\;\;\wt{\sigma}_{k}(x)=\frac{1}{\sigma_{k}([0,1))}\sigma_{k}(2^{k}x-1),\;\forall x\in J_{k},
\end{equation*}
and consider also similarly rescaled and shifted copies $\wt{f}_{k},\wt{g}_{k}$ of the functions $f_{k},g_{k}$ respectively on the interval $J_{k}$. For all $k=1,2,\ldots$, we extend the functions $\wt{f}_{k},\wt{g}_{k}$ on the whole real line by letting them vanish outside of $J_{k}$. Consider the weights $\wt{w},\wt{\sigma}$ on $[0,1)$ given by $\wt{w}(x)=\wt{w}_{k}(x)$, for all $x\in J_{k}$, for all $k=1,2,\ldots$, and similarly for $\wt{\sigma}$. We extend the weights $\wt{w},\wt{\sigma}$ to weights on the whole real line, as in \ref{s: smooth Haar mult extend}, and abusing the notation we denote the extended weights by the same letter. Then, translation and rescaling invariance shows that
\begin{equation}
\label{sequence of functions}
|\La H(\wt{f}_{k}\wt{\sigma}\1_{[0,1)}),\wt{g}_{k}\wt{w}\Ra|\gtrsim_{p}k^{1/p'}\Vert\wt{f}_{k}\Vert{\ci{L^{p}(\wt{\sigma})}}\Vert\wt{g}_{k}\Vert\ci{L^{p'}(\wt{w})}.
\end{equation}
It follows that $\Vert H(\cdot\wt{\sigma}\1_{[0,1)})\Vert{\ci{L^{p}(\wt{\sigma})\rightarrow L^{p}(\wt{w})}}=\infty$. An easy application of the closed graph theorem implies then that there exists $f\in L^{p}(\wt{\sigma})$ with $H(f\wt{\sigma}\1_{[0,1)})\notin L^{p}(\wt{w})$. For instance, one can use the facts that 
\begin{equation*}
\Vert f\wt{\sigma}\1_{[0,1)}\Vert{\ci{L^{1}(\R )}}\leq\wt{\sigma}([0,1))^{1/p'}\Vert f\Vert{\ci{L^{p}(\wt{\sigma})}},\;\forall f\in L^{p}(\wt{\sigma}),
\end{equation*}
and that the linear operator $H:L^{1}(\R )\rightarrow L^{1,\infty}(\R )$ is bounded. 

It remains now to prove that the joint ``fattened'' $A_p$ characteristic of the weights $\wt{w},\wt{\sigma}$, that is the quantity
\begin{equation*}
\sup_{\lambda\in\mathbb{C}_{+}}\left(\int_{\R }\frac{(\text{Im}(\lambda))^{p-1}}{|x-\lambda|^{p}}\wt{w}(x)dx\right)\left(\int_{\R }\frac{(\text{Im}(\lambda))^{p'-1}}{|x-\lambda|^{p'}}\wt{\sigma}(x)dx\right)^{p-1},
\end{equation*}
is finite. As in Subsection \ref{s:smooth Haar mult}, it suffices to prove that
\begin{equation}
\label{muck}
[\wt{w},\wt{\sigma}]\ci{A_{p},\cD([0,1))}\sim_{p}1
\end{equation}
and
\begin{equation}
\label{smooth}
S\ut{sd}_{\wt{w}},~S\ut{sd}_{\wt{\sigma}}\leq 1+\delta.
\end{equation}
Note that translation and rescaling invariance yields immediately that condition \eqref{muck} is fulfilled over $J_{k}$, for all $k=1,2,\ldots$. To check it over intervals that are not contained in any $J_{k}$, it suffices to note that $\La\wt{w}\Ra{\ci{J_{k}}}=1$, for all $k=1,2,\ldots$, and similarly for $\wt{\sigma}$.

Moreover, translation and rescaling invariance yields immediately that condition \eqref{smooth} is fulfilled over $J_{k}$, for all $k=1,2,\ldots$. Thus, it suffices to check that it still holds for adjacent dyadic intervals of equal length whose common endpoint is also an endpoint of some $J_{k}$. To that end, notice that for all $k=1,2,\ldots$, by Remark \ref{r:pres end} we have $\La w_{k}\Ra{\ci{[0,a)}}=\La w_{k}\Ra{\ci{[a,1)}}=w_{k}([0,1))$, for all $a\in(0,1)$, and similarly for $\sigma_{k}$. It follows that $\La\wt{w}\Ra\ci{J}=1$, for all $J\in\mathcal{D}(J_{k})$ sharing an endpoint with $J_{k}$, and similarly for $\wt{\sigma}$, concluding the proof.

\begin{rem}
\label{r: exponent sarason conj}
It is clear that the proof remains valid if we have \eqref{sequence of functions} with $k$ raised to any (fixed) positive exponent. Thus, the proof remains valid if we have \eqref{sequence of examples} with $k$ raised to any (fixed) exponent greater than $1/p$. Therefore, as long as the Muckenhoupt characteristic estimate in the ``large step" examples features an exponent greater than $1/p$, the $L^{p}$ version of Sarason's conjecture cannot be true.
\end{rem}

\section{Appendix}\label{s: appendix}

\subsection{Facts about simply symmetric random walks}\label{s: rand walks}

We give here the proof of Lemma \ref{l: res rand walks}. It can be found in any probability theory textbook (see e.g. \cite{prob}). We do not follow the notation from Section 2.

Let $(\Omega,\mathcal{F},\mathbb{P},\mathbb{F}=(\mathcal{F}_{n})^{\infty}_{n=0})$ be a filtered probability space. Let $(\omega_n)^{\infty}_{n=1}$ be a sequence of random variables on $\Omega$, such that for all $n=1,2,\ldots$ the random variable $\omega_{n}$ is $\mathcal{F}_{n}$-measurable with $\mathbb{P}(\omega_{n}=1)=\mathbb{P}(\omega_{n}=-1)=\frac{1}{2}$, and such that the $\sigma$-algebras $\sigma(\omega_{n},\omega_{n+1},\ldots)$ and $\mathcal{F}_{n-1}$ are independent. Set $S_0=0$ and $S_{n}=\sum_{k=1}^{n}\omega_{k}$, for all $n=1,2,\ldots$. Then, $S=(S_{n})^{\infty}_{n=0}$ is a martingale on $\Omega$. In the statement and the proof of the following lemma, we denote $x\wedge y:=\min(x,y)$.

\begin{lm} \label{l: fin stop}
Let $a,b\in(0,\infty)$. Consider the stopping times $\tau^1,\tau^2,\tau$ on $\Omega$ given by
\begin{equation*}
\tau^1:=\inf\lbrace n\in\mathbb{N}:\;S_{n}=b\rbrace,\qquad\tau^2:=\inf\lbrace n\in\mathbb{N}:\;S_{n}=-a\rbrace,\qquad\tau:=\tau^1\wedge\tau^2.
\end{equation*}

(a) There holds $\tau^1,\tau^2<\infty$ a.e. on $\Omega$.

(b)There holds $\mathbb{P}(\tau=\tau^1)=\frac{a}{a+b}$ and $\mathbb{P}(\tau=\tau^2)=\frac{b}{a+b}$.
\end{lm}
\begin{proof}
Let $\theta\in(0,\infty)$ be arbitrary. Consider the martingale $M$ given by
\begin{equation*}
M_{n}:=\frac{e^{\theta S_{n}}}{(\cosh\theta)^{n}},\qquad n=0,1,2,\ldots
\end{equation*}
(note that $0<M_{n}\leq\frac{e^{n\theta}}{(\cosh\theta)^{n}}$, for all $n=0,1,2,\ldots$). By optional sampling theorem, we have that the stopped process $M^{\tau^1}:=(M_{n\wedge\tau^1})^{\infty}_{n=0}$ is also a martingale. We notice that
\begin{equation*}
0<M_{n\wedge\tau^1}=\frac{e^{\theta S_{n\wedge\tau^1}}}{(\cosh\theta)^{n\wedge\tau^1}}\leq e^{\theta b},\qquad \forall n=0,1,2,\ldots,
\end{equation*}
thus $M^{\tau^1}$ is uniformly bounded. Therefore, by basic convergence facts for martingales it follows that $M^{\tau^1}$ is uniformly integrable, therefore there exists $X\in L^1(\Omega)$ such that $M^{\tau^1}_{n}\rightarrow X$ a.e. pointwise on $\Omega$ as $n\rightarrow\infty$. It is clear that $\lim_{n\rightarrow\infty}M^{\tau^1}_{n}(x)=\frac{e^{\theta S_{\tau^1(x)}(x)}}{(\cosh\theta)^{\tau^1(x)}}$, for all $x\in\Omega$ with $\tau^1(x)<\infty$, and that $M^{\tau^1}_{n}(x)=\frac{e^{\theta S_{n}(x)}}{(\cosh\theta)^{n}}$, for all $n=0,1,2,\ldots$, for all $x\in\Omega$ with $\tau^1(x)=\infty$. Then, for all $x\in\Omega$ with $\tau^1(x)=\infty$, we have
\begin{equation*}
M^{\tau^1}_{n}(x)=\frac{e^{\theta S_{n}(x)}}{(\cosh\theta)^{n}}\leq\frac{e^{\theta b}}{(\cosh\theta)^n},\qquad\forall n=0,1,2,\ldots,
\end{equation*}
therefore since $\cosh\theta>1$ we obtain $X(x)=0$. It follows that
\begin{equation*}
\mathbb{E}\left[\frac{e^{\theta S_{\tau^1}}}{(\cosh\theta)^{\tau^1}}\1_{\lbrace\tau^1<\infty\rbrace}\right]=\mathbb{E}[X]=\mathbb{E}[M_0]=1,
\end{equation*}
therefore since $S_{\tau^1}=b$ on $\lbrace\tau^1<\infty\rbrace$ we obtain
\begin{equation*}
\mathbb{E}[(\cosh\theta)^{-n}\1_{\lbrace\tau^1<\infty\rbrace}]=e^{-\theta b}.
\end{equation*}
Since $\cosh\theta>1$, for all $\theta>0$, taking the limit as $\theta\rightarrow 0^{+}$ and applying the Dominated Convergence Theorem we obtain $\mathbb{P}(\tau^1<\infty)=1$. Similarly $\tau^2<\infty$ a.e. on $\Omega$.

(b) Set $\mathbb{P}(\tau=\tau^1)=p_1$ and $\mathbb{P}(\tau=\tau^2)=p_2$. Then, since $\tau^1,\tau^2<\infty$ a.e. on $\Omega$ we obtain $\tau^1\neq \tau^2$ a.e. on $\Omega$, therefore $p_1+p_2=1$. We also have $p_1=\mathbb{P}(S_{\tau}=b)$ and $p_2=\mathbb{P}(S_{\tau}=-a)$. An application of the optional sampling theorem yields $\mathbb{E}[S_{\tau}]=0$, i.e. $bp_1-ap_2=0$. Therefore $p_1=\frac{a}{a+b}$ and $p_1=\frac{b}{a+b}$.
\end{proof}

\subsection{Stopping on the lower hyperbola}\label{s: stop lower hyperb}

We give here the proof of Lemma \ref{l: stop lower hyperb}.

Let $p\in(1,\infty)$. Let $x,y>0$ be arbitrary, such that $xy^{p-1}\geq1$. We claim that there exist $a_1,b_1,a_2,b_2>0$ with $a_2\leq x\leq a_1$ and $b_1\leq y\leq b_2$, such that $a_1b_1^{p-1}=a_2b_2^{p-1}=1$ and $x=\frac{a_1+a_2}{2}$, $y=\frac{b_1+b_2}{2}$.

Indeed, consider the function $f:(0,2y)\rightarrow(0,\infty)$ given by $f(b)=\frac{1}{b^{p-1}}+\frac{1}{(2y-b)^{p-1}}$, for all $b\in(0,2y)$. We have $\lim_{b\rightarrow0^{+}}f(b)=\infty$ and $f(y)=\frac{2}{y^{p-1}}\leq 2x$. Therefore, an application of the Intermediate Value Theorem yields that there exists $b_1\in(0,y]$ with $f(b_1)=2x$. Then, we take $b_2=2y-b_1$ and $a_1=b_1^{1-p}$, $a_2=b_2^{1-p}$.

\subsection{Getting a little above the upper hyperbola}\label{s: upper hyperb}

We give here the proof of Lemma \ref{l: upper hyperb}.

Let $p\in(1,\infty)$. Let $x_1,y_1,x_2,y_2>0$ and $A>0$, such that
\begin{equation*}
x_1y_1^{p-1},~\left(\frac{x_1+x_2}{2}\right)\left(\frac{y_1+y_2}{2}\right)^{p-1},~x_2y_2^{p-1}\leq A.
\end{equation*}
We will show that
\begin{equation*}
(ax_2+(1-a)x_1)(ay_2+(1-a)y_1)^{p-1}\leq 2^{p}A,\qquad\forall a\in[0,1].
\end{equation*}
If $x_1\leq x_2$ and $y_1\leq y_2$, or $x_1\geq x_2$ and $y_1\geq y_2$, then we have nothing to show. Assume now that either  $x_2> x_1$ and $y_1< y_2$, or $x_1> x_2$ and $y_2> y_1$. Replacing if necessary $A$ by $A^{p'-1}$, $p$ by $p'$, and $x_{i}$ by $y_{i}$ for $i=1,2$, we can without loss of generality assume that there holds $x_1> x_2$ and $y_2> y_1$. Set
\begin{equation*}
x=\frac{x_1-x_2}{x_1+x_2},\;y=\frac{y_2-y_1}{y_2+y_1},\;B=\frac{A}{\left(\frac{x_1+x_2}{2}\right)\left(\frac{y_1+y_2}{2}\right)^{p-1}}.
\end{equation*}
Then, we have $x,y\in(0,1)$, $B\geq 1$ and $(1+x)(1-y)^{p-1},~(1-x)(1+y)^{p-1}\leq B$, and we want to show that
\begin{equation*}
\sup_{s\in[-1,1]}(1-sx)(1+sy)^{p-1}\leq 2^{p}B.
\end{equation*}
This is clear, because $B\geq 1$ and $(1-sx)(1+sy)^{p-1}\leq 2\cdot 2^{p-1}=2^{p}$, for all $s\in[-1,1]$, concluding the proof.

\begin{rem}
\label{r: nonimprov upper hyperb}
Although the above estimate is crude, it can be seen that in general one cannot obtain an estimate better that $2^{p}/{p}$ as $p\rightarrow\infty$.
\end{rem}

\subsection{A counterexample}\label{s: counter two-weight Muck}

We show here that finiteness of joint Muckenhoupt $A_{p}$ characteristic does not quarantee two-weight estimates for the Hilbert transform $H$. We will use a modified version of F. Nazarov's example in \cite[p.~1]{Nazarov}. Let $p\in(1,\infty)$. Consider the weights $w,\sigma$ on $\R $ given by
\begin{equation*}
w(t):=|t|^{p-1},\;\;
\sigma(t):=
\begin{cases}
|t|^{-p/(p-1)},\text{ if }|t|>1\\
1,\text{ if }|t|\leq 1
\end{cases}
,\qquad\forall t\in\R .
\end{equation*}
We show first that $[w,\sigma]\ci{A_{p}}<\infty$. It is clear that $\La w\Ra_{[a,b)}\La \sigma\Ra_{[a,b)}^{p-1}\lesssim_{p} 1$, for all $a,b\in\R $ with $-2\leq a<b\leq 2$. For all $a\in(1,\infty)$, we have
\begin{align*}
\La w\Ra_{[0,a)}\La\sigma\Ra_{[0,a)}^{p-1}\sim_{p} a^{p-1}\left(\frac{1+1-a^{-1/(p-1)}}{a}\right)^{p-1}\leq 2^{p-1}.
\end{align*}
Then, for all $a,b\in[0,\infty)$ with $0<a\leq\frac{b}{2}$, we have $b-a\geq\frac{b}{2}$, therefore
\begin{equation*}
\La w\Ra_{[a,b)}\La\sigma\Ra_{[a,b)}^{p-1}\lesssim_{p}\La w\Ra_{[0,b)}\La\sigma\Ra_{[0,b)}^{p-1}\lesssim_{p}1.
\end{equation*}
Moreover, for all $a,b\in[0,\infty)$ with $0<1<\frac{b}{2}<a<b$, we have $w(t)\sim_{p} a^{p-1}$, for all $t\in[a,b)$ and $\sigma(t)\sim_{p} a^{-p/(p-1)}$, for all $t\in[a,b)$, therefore
\begin{equation*}
\La w\Ra_{[a,b)}\La\sigma\Ra_{[a,b)}^{p-1}\sim_{p} a^{p-1}(a^{-p/(p-1)})^{p-1}=a^{-1}<1.
\end{equation*}
Thus $\La w\Ra_{[a,b)}\La\sigma\Ra_{[a,b)}^{p-1}\lesssim_{p}1$, for all $a,b\in[0,\infty)$ with $a<b$. This implies $\La w\Ra_{[-b,-a)}\La\sigma\Ra_{[-b,-a)}^{p-1}\lesssim_{p}1$, for all $a,b\in[0,\infty)$ with $a<b$. Moreover, for all $a,b\in(0,\infty)$, setting $c=\max(a,b)$ and noticing that $b+a\geq c$ we obtain
\begin{align*}
\La w\Ra_{[-a,b)}\La\sigma\Ra_{[-a,b)}^{p-1}\lesssim_{p}\La w\Ra_{[-c,c)}\La\sigma\Ra_{[-c,c)}^{p-1}=\La w\Ra_{[0,c)}\La\sigma\Ra_{[0,c)}^{p-1}\lesssim_{p}1,
\end{align*}
yielding the desired result.

Consider now the function $f:=\1\ci{[0,1)}$. We have $\Vert f\Vert \ci{L^{p}(\sigma)}=1$, just as in \cite[p.~1]{Nazarov}. Direct computation shows then that $H(f\sigma)(t)=H(\1_{[0,1)})(t)=\frac{1}{\pi}\ln\left(\frac{t}{t-1}\right)$ for almost every $t\in(1,\infty)$, so since $\lim_{t\rightarrow\infty}t\ln\left(\frac{t}{t-1}\right)=1$ we deduce $H(f\sigma)(t)\sim\frac{1}{t}$ for almost every $t\in(2,\infty)$, thus $|H(f\sigma)(t)|^{p}w(t)\sim_{p}\frac{1}{t}$ for almost every $t\in(2,\infty)$, just as in \cite[p.~1]{Nazarov}, thus $H(f\sigma)\notin L^{p}(w)$.

\subsection{Proofs of F. Nazarov's lemmas}\label{s: naz lemma proof}

We give here the proofs of F. Nazarov's lemmas from \cite{Nazarov}.

\begin{proof}[Proof of Lemma \ref{l:naz1}]
We follow the proof in \cite[\S 6]{Nazarov}. Let $\e>0$ be arbitrary. We have $\lim_{\delta\rightarrow0^{+}}(1+\delta)^{1/\sqrt{\delta}}=\lim_{\delta\rightarrow0^{+}}(1+\delta)=1$, therefore there exists $\delta\in\left(0,\frac{1}{4}\right)$ such that
\begin{equation*}
(1-2\sqrt{\delta})(1+\delta)^{-2/\sqrt{\delta}}>(1+\e)^{-1/2},\;\;\;(1+2\sqrt{\delta})(1+\delta)^{2+2/\sqrt{\delta}}<(1+\e)^{1/2}.
\end{equation*}
Let now $w$ be a weight on $\R$ with $S\ut{sd}_{w}\leq1+\delta$.

\textbf{Claim.} For all intervals $I$ in $\R$, for all $J\in\cD$ with $|J|\leq\sqrt{\delta}|I|\leq 2|J|$ and containing one of the endpoints of $I$, there holds $\langle w\rangle\ci{J}/\langle w\rangle\ci{I}, ~\langle w\rangle\ci{I}/\langle w\rangle\ci{J}\leq (1+\e)^{1/2}$.

Assume the claim for the moment. Let $I$ be an arbitrary interval in $\R$. There exists $J\in\cD$, such that $2|J|\leq\sqrt{\delta}|I|\leq 4|J|$ and $J$ contains the center of $I$. Then, by the claim, applied for $I_{-},J$ and $I_{+},J$, we have
\begin{equation*}
\frac{\La\rho\Ra\ci{J}}{\La\rho\Ra{\ci{I_{-}}}},\frac{\La\rho\Ra\ci{J}}{\La\rho\Ra{\ci{I_{+}}}}\leq(1+\e)^{1/2},\;\;\;\frac{\La\rho\Ra{\ci{I_{+}}}}{\La\rho\Ra\ci{J}},\frac{\La\rho\Ra{\ci{I_{-}}}}{\La\rho\Ra\ci{J}}\leq(1+\e)^{1/2},
\end{equation*}
therefore $\La\rho\Ra{\ci{I_{+}}}/\La\rho\Ra{\ci{I_{-}}},~\La\rho\Ra{\ci{I_{-}}}/\La\rho\Ra{\ci{I_{+}}}\leq 1+\e$, yielding the desired result.

We now prove the claim. Let $I$ be an interval in $\R $, and let $J\in\cD$ with $|J|\leq\sqrt{\delta}|I|\leq2|J|$, containing one of the endpoints of $I$.

Set $J_{\ast}=\lbrace K\in\cD:\;|K|=|J|,\;K\subseteq I\rbrace$ and $I_{\ast}=\bigcup J_{\ast}$. Clearly $J_{\ast}\neq\emptyset$, since $|J|<\frac{1}{2}|I|$. It is clear that 
\begin{equation*}
\# J_{\ast}\leq\frac{|I|}{|J|}\leq\frac{2}{\sqrt{\delta}}.
\end{equation*}
For all $K\in J_{\ast}$, there exist $l\in\lbrace1,\ldots,\# J_{\ast}\rbrace$ and $J_{1},\ldots,J_{l+1}\in\cD$ of length equal to $|J|$, such that $J_{1}=K$, $J_{l+1}=J$ and $J_{i},J_{i+1}$ are adjacent or coincide, for all $i=1,\ldots,l$, therefore
\begin{equation*}
\frac{\La w\Ra\ci{K}}{\La w\Ra\ci{J}}=\prod_{i=1}^{l}\frac{\La w\Ra\ci{J_{i}}}{\La w\Ra\ci{J_{i+1}}}\geq(1+\delta)^{-l}\geq(1+\delta)^{-2/\sqrt{\delta}},
\end{equation*}
thus
\begin{equation*}
\frac{\La w\Ra\ci{I_{\ast}}}{\La w\Ra\ci{J}}=\frac{|J|}{|I_{\ast}|}\sum_{K\in J_{\ast}}\frac{\La w\Ra\ci{K}}{\La w\Ra\ci{J}}\geq\frac{|J|}{|I_{\ast}|}(\# J_{\ast})(1+\delta)^{-2/\sqrt{\delta}}=(1+\delta)^{-2/\sqrt{\delta}}.
\end{equation*}
Note also that $|I_{\ast}|\geq |I|-2|J|\geq(1-2\sqrt{\delta})|I|$, therefore
\begin{align*}
\frac{\La w\Ra\ci{I}}{\La w\Ra\ci{J}}\geq\frac{|I_{\ast}|}{|I|}\frac{\La w\Ra\ci{I_{\ast}}}{\La w\Ra\ci{J}}\geq(1-2\sqrt{\delta})(1+\delta)^{-2/\sqrt{\delta}}
\geq(1+\e)^{-1/2}.
\end{align*}
Set also $J^{\ast}=\lbrace K\in\cD:\;|K|=|J|,\;K\cap I\neq\emptyset\rbrace$ and $I^{\ast}=\bigcup J^{\ast}$. It is clear that
\begin{equation*}
\# J^{\ast}\leq\frac{|I|}{|J|}+2\leq\frac{2}{\sqrt{\delta}}+2.
\end{equation*}
Then, similarly to previously we have
\begin{equation*}
\frac{\La w\Ra\ci{I^{\ast}}}{\La w\Ra\ci{J}}\leq (1+\delta)^{2+2/\sqrt{\delta}}.
\end{equation*}
Note also that $|I_{\ast}|\leq |I|+2|J|\leq (1+2\sqrt{\delta})|I|$, therefore
\begin{align*}
\frac{\La w\Ra\ci{I}}{\La w\Ra\ci{J}}\leq\frac{|I^{\ast}|}{|I|}\frac{\La w\Ra\ci{I^{\ast}}}{\La w\Ra\ci{J}}\leq(1+2\sqrt{\delta})(1+\delta)^{2+2/\sqrt{\delta}}
\leq(1+\e)^{1/2},
\end{align*}
concluding the proof.
\end{proof}

\begin{proof}[Proof of Lemma \ref{l:naz2}] 
We follow the proof in \cite[\S 11]{Nazarov}. Set $\e=(25/16)^{1/p}-1$. Choose $\delta\in\left(0,\frac{1}{4}\right)$ as in the proof of Lemma \ref{l:naz1} for this $\e$. Let $\rho$ be a weight on $\R $ with $[\rho]\ci{A_{p},\cD}<\infty$ and $S_{\rho}^{\text{sd}},S_{\tau}^{\text{sd}}\leq 1+\delta$, where $\tau=\rho^{-1/(p-1)}$. Let $I$ be an arbitrary interval in $\R $. By the proof of  Lemma \ref{l:naz1} we have that there exists $J\in\cD$ such that
\begin{equation*}
\langle\rho\Ra\ci{I}\leq(1+\e)^{1/2}\langle\rho\Ra\ci{J},\;\;\;
\La\tau\Ra\ci{I}\leq(1+\e)^{1/2}\La\tau\Ra\ci{J},
\end{equation*}
therefore
\begin{equation*}
\La\rho\Ra\ci{I}\La\tau\Ra\ci{I}^{p-1}\leq(1+\e)^{p/2}\La\rho\Ra\ci{J}\La\tau\Ra\ci{J}^{p-1}\leq\frac{5}{4}[\rho]\ci{A_{p},\cD}.
\end{equation*}
It follows that $[\rho]\ci{A_{p},\cD}\leq [\rho]\ci{A_{p}}\leq(5/4)[\rho]\ci{A_{p},\cD}$, concluding the proof.
\end{proof}

\end{document}